\documentclass[11pt]{article}
\usepackage{amsmath,amsfonts,amssymb,amsthm}
\usepackage{graphics}
\usepackage{epsfig}
\usepackage{fullpage}
\usepackage{color}
\usepackage{psfrag,expdlist}
\usepackage{caption}
\usepackage{subcaption}
\usepackage{stmaryrd,mathtools,relsize}
\usepackage{cancel}
\usepackage{tikz}
\usetikzlibrary{arrows.meta}
\usetikzlibrary{decorations.pathreplacing,calligraphy}
\usepackage{xfrac,bigints}
\usepackage{enumitem}
\usepackage{pgfplots}

\usepackage{tikz-3dplot} 
\usetikzlibrary{decorations.pathreplacing, intersections, arrows.meta} 
\tikzset{
font=\small, 
point/.style={fill, circle, inner sep=1.2pt}, 
>={Straight Barb[round,angle=60:1.2mm 1]} 
}

\newcommand{\supp}{\mbox{supp} \, }

\title{Axisymmetric flows with swirl for Euler \\
and Navier-Stokes equations}

\author{
Theodoros Katsaounis
\thanks{Department of Mathematics and Applied Mathematics, University of Crete, Heraklion, 70013 Crete, Greece.
Email: {\tt thodoros.katsaounis@uoc.gr}}
\thanks{Institute of Applied and Computational Mathematics,
FORTH, Heraklion, Greece.}
\and
Ioanna Mousikou
\thanks{Computer, Electrical, Mathematical Sciences \& Engineering Division, King Abdullah University of Science and Technology (KAUST), Thuwal, Saudi Arabia.
Email: {\tt ioanna.mousikou@kaust.edu.sa} }
\and 
Athanasios E. Tzavaras
\thanks{Computer, Electrical, Mathematical Sciences \& Engineering Division, King Abdullah University of Science and Technology (KAUST), Thuwal, Saudi Arabia.
Email: {\tt athanasios.tzavaras@kaust.edu.sa} }
\date{}
}

\renewcommand{\footnote}{\endnote}

\def\charf {\mbox{{\text 1}\kern-.24em {\text l}}}

\newcommand{\ignore}[1]{}

\newtheorem{lemma}{Lemma}[section]
\newtheorem{theorem}[lemma]{Theorem}

\newtheorem{proposition}[lemma]{Proposition}
\newtheorem{corollary}[lemma]{Corollary}
\newtheorem{definition}[lemma]{Definition}

\def\torus{{{\text{\rm T}} \kern-.42em {\text{\rm T}}}}
\def\T3{\torus^3}

\begin{document}
\maketitle
\baselineskip=12pt

\numberwithin{equation}{section}

\begin{abstract}
\noindent
We consider the incompressible axisymmetric Navier-Stokes equations with swirl as an idealized model for tornado-like flows. Assuming an infinite vortex line 
which interacts with a boundary surface resembles the tornado core, we look for stationary self-similar solutions of the axisymmetric Euler and axisymmetric Navier-Stokes equations. We are particularly interested in the connection of the two problems in the zero-viscosity limit.  First, we construct a class of explicit stationary self-similar solutions for the axisymmetric Euler equations. Second, we consider the possibility of discontinuous solutions and prove that there do not exist self-similar stationary Euler solutions with slip discontinuity. This nonexistence result is extended to a class of flows where there is mass input or mass loss through the vortex core. Third, we consider solutions of the Euler equations as zero-viscosity limits of solutions to Navier-Stokes. Using techniques from the theory of Riemann problems for conservation laws, we prove that, under certain assumptions, stationary self-similar solutions of the axisymmetric Navier-Stokes equations converge to stationary self-similar solutions of the axisymmetric Euler equations as $\nu\to0$. This allows to characterize the type of Euler solutions that arise via viscosity limits. 
\end{abstract}
%
%
%
\tableofcontents



\section{Introduction}
\label{sec:intro}
Tornadoes are among the most extreme and destructive weather phenomena and,  due to the numerous casualties and substantial damage in property,
 their modeling has attracted considerable interest.  At the core of mathematical models for the description of tornadoes is the concept of swirling flows. 
 Assuming the tornado structure does not change significantly in time and that a vortex line resembles the tornado core, \cite{Mor66}, a common model for tornado-like flows  is the stationary incompressible axisymmetric Navier-Stokes equations. Namely, the following system in cylindrical coordinates is considered
\begin{subequations}
\label{axi-sys}
\begin{align}
u \frac{\partial u}{\partial r} + w \frac{\partial u}{\partial z} - \frac{v^2}{r} & = \nu \Big[\frac{1}{r}\frac{\partial }{\partial r} \Big(r \frac{\partial u}{\partial r}\Big)+ \frac{\partial^2 u}{\partial z^2} -  \frac{u}{r^2}  \Big] - \frac{\partial p}{\partial r}, \\
u \frac{\partial v}{\partial r} + w \frac{\partial v}{\partial z} + \frac{uv}{r} & = \nu \Big[\frac{1}{r}\frac{\partial }{\partial r} \Big(r \frac{\partial v}{\partial r}\Big)+ \frac{\partial^2 v}{\partial z^2} -  \frac{v}{r^2}  \Big], \\
u \frac{\partial w}{\partial r} + w \frac{\partial w}{\partial z} \qquad  & = \nu \Big[\frac{1}{r}\frac{\partial }{\partial r} \Big(r \frac{\partial w}{\partial r}\Big)+ \frac{\partial^2 w}{\partial z^2}  \qquad \Big] - \frac{\partial p}{\partial z}, \\
\frac{1}{r} \frac{\partial }{\partial r} (ru) + \frac{\partial w}{\partial z}  & = 0,
\end{align}
\end{subequations}
where $\vec{u}=(u,v,w): \mathbb{R}^3\times \mathbb{R}_+ \to \mathbb{R}^3$ is the velocity, $p : \mathbb{R}^3 \times \mathbb{R}_+\to \mathbb{R}$ is the pressure and $\nu \geq 0$ is the coefficient of kinematic viscosity. 
Modeling tornadoes necessitates the consideration of certain additional factors, like rotation induced from the cloud and buoyancy effects.
The reader is referred to App. \ref{sec:tornado} for a quick presentation of models and  references.
Nevertheless, system \eqref{axi-sys} is considered as the base-model for describing their core structure, \cite{Rot13}.

Long \cite{Long58,Long61} introduced a self-similar ansatz and reduced the stationary axisymmetric Navier-Stokes equations to a system of ordinary differential equations focusing on the  boundary layer description.  Independently, Goldshtik  \cite{Gold60}
examined the existence of self-similar solutions as the Reynolds number varies; he concluded his system 
 is not solvable for all Reynolds numbers and characterized this loss of existence as a 'paradox'. Few years later, Serrin \cite{Serrin} presented
 another self-similar ansatz and showed that only three types of solution can be associated with the interaction of an infinite vortex line with a boundary plane:
 the flow can be either ascending, or descending, or a combination of these two profiles, i.e. downward near the vortex line and inward near the r-axis.
 The latter is usually referred as a double-celled vortex and Serrin \cite{Serrin} presented extensive comparisons of such solutions to tornadoes.
 Many authors subsequently studied the system \eqref{axi-sys} on occasion considering more general families of self-similar solutions, for example \cite{FFFMB95,BDSS2014,GS90}, or more general geometries including for example conical flows,  \cite{GS90,Gold90,Sozou92,SWK94,FFA00,Shtern12}.

This work has two objectives: First, to consider the full system \eqref{axi-sys} and provide a complete study of the stationary self-similar solutions of axisymmetric Navier-Stokes and of axisymmetric Euler equations. Second, to compare those equations and study the emergence of Euler solutions
from the Navier-Stokes solutions  as the viscosity $\nu$ goes to zero. One novel feature
for the Euler equations is the consideration of solutions with a slip discontinuity in the velocities. 
Regarding the Navier-Stokes equations, we obtain an alternative
derivation of the equations studied by Serrin \cite{Serrin} and  study the emergence of stationary self-similar solutions 
of the Euler equations as zero-viscosity limits from corresponding solutions of the Navier-Stokes equations. Our analysis indicates
that first there are no solutions with slip-discontinuity for the stationary, self-similar Euler equations, and second that the solutions of Navier-Stokes
approach the single cell solution of the Euler equations in the zero-viscosity limit. Numerical calculations indicate
the presence of double-celled solutions but this happens for finite (even relatively large viscosities) and does not persist as the
viscosity decreases.

We now provide an outline of the work. We introduce the self-similar transformation
\begin{align*}
u(r,z) = \frac{1}{r} U(\xi), \quad v(r,z) = \frac{1}{r} V(\xi), \quad  w(r,z) =\frac{1}{r}  W(\xi),\quad  p(r,z) =\frac{1}{r^2}  P(\xi) \, ,
\end{align*}
in $\xi = \frac{z}{r}$ to  \eqref{axi-sys} which captures the scale-invariance of Navier-Stokes equations. This leads, for the axisymmetric Euler equations, to
a system of ordinary differential equations:
\begin{subequations}
\label{Eul1}
\begin{align}
\bigg[\frac{\theta^2}{2} + (1+\xi^2)P \bigg]' &= -\xi V^2, \\
V' \theta &= 0, \\
\bigg[\theta^2 - \xi \Big(\frac{\theta^2}{2}\Big)' + P \bigg]' &= 0,
\end{align}
\end{subequations}
 for $(\theta, V, P)$, where $\theta$ is a stream-function with $U = - \theta^\prime$ and $W = \theta - \xi \theta^\prime$. The same transformation leads
for the axisymmetric Navier-Stokes equations to the system
\begin{subequations}
\label{NS1}
\begin{align}
\bigg[\frac{\theta^2}{2} + (1+\xi^2)P \bigg]' &=  \nu \bigg[\xi\theta - (1+\xi^2) \theta' \bigg]' -\xi V^2, \\
V' \theta &= \nu \Big[3\xi  V' + (1+\xi^2) V''\Big], \\
\bigg[\theta^2 - \xi \Big(\frac{\theta^2}{2}\Big)' + P \bigg]' &= \nu \Big[\xi\theta - \xi^2 \theta' - \xi (1+\xi^2) \theta'' \Big]' \, , 
\end{align}
\end{subequations}
augmented with the physically relevant boundary conditions (see section \ref{sec:axi} for the derivations).

First, we consider  \eqref{Eul1}. Imposing no-penetration conditions on the boundary and the vortex core we derive explicit solutions that correspond to either ascending or descending flows. To examine whether a double-celled vortex may occur for the Euler equations, we consider the existence of solutions with slip discontinuity in velocity.
These would emerge from the Rankine-Hugoniot conditions for \eqref{Eul1},
\begin{equation}\label{intro:RH}
\llbracket P  \rrbracket  = 0, \qquad 
\llbracket \theta V  \rrbracket = 0, \qquad 
\llbracket \theta \, \theta'  \rrbracket = 0,   \qquad 
\end{equation}
where $\llbracket \cdot \rrbracket$ is the jump operator. We prove that there do not exist solutions of the self-similar Euler equations with discontinuities at a finite number of points that satisfy the jump conditions \eqref{intro:RH}. This nonexistence result is then extended to a class of flows where there is mass input or loss through the vortex line. For the aforementioned flows, the explicit continuous solutions are also derived.
The non-existence result leads to conjecture that the double-celled vortices cannot persist in the limit as the viscosity goes to zero.

Next, we focus on the stationary self-similar axisymmetric Navier-Stokes  system \eqref{NS1} in the context of a vortex core interacting with a boundary. 
This problem is reduced to a coupled integro-differential system 
for $(\theta(\xi), V(\xi)) $
\begin{subequations}
\label{NS2}
\begin{align}
\frac{\theta^2}{2} - \nu \bigg[(1+\xi^2) \theta' + \xi \theta \bigg]  &= \mathcal{G}\left(\xi\right), \\
\nu (1+\xi^2) V'' + \Big(3\nu\xi - \theta\Big) V' &= 0,
\end{align}
\end{subequations}
with boundary conditions
\begin{equation}
\theta = \theta' = V = 0  \,\, \textrm{ at } \;  \xi = 0,  \qquad 
V \to V_\infty , \,\, U \to 0 \, , \, \,  \textrm{ as} \; \xi \to \infty.
\end{equation}
where the functional $\mathcal{G}$ depending on $V(\cdot)$ and $\xi$ is defined via
\begin{equation}\label{intro-defG}
\mathcal{G}\left(\xi\right) = \mathcal{G}\left(V(\cdot), \xi\right) =  \xi \sqrt{1+\xi^2} \int_\xi^\infty \frac{1}{\zeta^2(1+\zeta^2)^\frac{3}{2}} \bigg( \int_0^\zeta s V^2(s) ds\bigg) d\zeta + \,\, {E_0} \phi(\xi).
\end{equation}
The system is recast into two equivalent integrodifferential formulations and provides a common framework encompassing the pioneering
works   on this problem of Goldshitk \cite{Gold60},  Serrin \cite{Serrin}  and Goldstick-Shtern \cite{GS90}.  Existence of solutions for \eqref{NS2} is provided under the conditions described in \cite{Serrin}. Moreover the system \eqref{NS2} is solved numerically using an iterative solver and a numerical bifurcation diagram is presented, see Figure \ref{bfd1}. 

We then focus on the zero-viscosity limit from Navier-Stokes to  Euler in the self-similar stationary setting. This leads to study the 
zero-viscosity limit $\nu \to 0$ for the integrodifferential system \eqref{NS2}. The problem is conveniently recasted as
an autonomous system,
\begin{equation}
\label{intro-ns}
\begin{aligned}
\nu \, \frac{d \Theta_\nu}{dx}  &=  \tfrac{1}{2} \Theta_\nu^2  - \mathcal{F}\left(V_\nu ; x\right), 
\\
\nu \frac{d^2 {V_\nu}}{dx^2} &= \Theta_\nu  \, \frac{d {V_\nu}}{dx} ,
\end{aligned}
\end{equation}
where $\mathcal{F}\left(V_\nu ; x\right)$ is the analog of \eqref{intro-defG} (and is defined in \eqref{defF} see section \ref{section:equiv-form}).
An analysis of the possible configurations of solutions of \eqref{intro-ns} (based on ideas of Goldstick \cite{Gold60}) leads to
a-priori estimates for solutions and classifying  the potential shapes of the stream function to three configurations.
To investigate  the zero-viscosity limit,  we employ compactness methods and exploit techniques developed in the theory of zero-viscosity limits 
for Riemann problems of conservation laws, \cite{Tz96,Pap99}. The derivative $dV$ is viewed as a limit of a family
of probability measures capturing the averaging process as $\nu\to 0$. 
We pursue two theories: the first is based on $L^p$ estimates for the stream function and variation estimates for the velocity.
It uses weak convergence methods and leads to the convergence Theorem \ref{maintheorem}. The latter is not fully satisfactory 
due to the weak bounds available for the stream function.
In a second step, under more restrictive conditions on the data for the problem, we obtain uniform variation estimates and
invoke Helly's theorem to conclude almost everywhere convergence
(along subsequences) from the Navier-Stokes solutions to the Euler solutions.
The convergence results are stated in Theorem \ref{secondtheorem}, Proposition \ref{prop:compactness}, and Corollary \ref{final}.
They permit to characterize admissibility restrictions for the stationary Euler solutions emerging from Navier-Stokes in this setting.

Finally, we study the boundary layer in the case of a model problem for \eqref{intro-ns} that captures the essential behavior of our system.  We use 
methods of asymptotic analysis, namely inner and outer solutions and matched asymptotic expansions. 
We deduce that the boundary layer is of order $\nu^{2/3}$ and provide an
asymptotic description of the stream function  $\Theta_\nu$ (see \eqref{asymd-theta}) and of $V_\nu$  (see \eqref{V-asympt}).

The work is organized as follows: 
In Section \ref{sec:axi}, we  introduce the self-similar ansatz and derive the systems \eqref{Eul1} and \eqref{NS1} that are used in the rest of this work. 
Sections \ref{sec:euler} and \ref{sec:disceuler} are devoted to stationary self-similar axisymmetric Euler equations. In section \ref{sec:euler}, an explicit continuous solution of \eqref{Eul1} is derived, while in section \ref{sec:disceuler} we investigate the existence of discontinuous solutions with finite number of discontinuities.
 In  section \ref{sec:disceuler}, Euler solutions are also studied for flows 
with input of mass through the vortex line. In Section \ref{sec:NS}, we study self-similar stationary solutions for axisymmetric Navier-Stokes.
After presenting some of their properties, 
we examine their limiting behavior as $\nu\to0$ using compactness methods and analytical ideas from the theory of Borel measures. 
The results are stated in Theorem \ref{maintheorem}, Theorem \ref{secondtheorem} and Corollary \ref{final}.
In Section \ref{bound-layer},
there is an asymptotic analysis of the boundary layer for a model problem, where inner and outer expansions leading to an explicit form of an asymptotic solution.
In Section \ref{sec:numerics}, a numerical scheme is presented for solving the stationary, self-similar system \eqref{intro-ns}.
We present indicative profiles of typical flows that appear for various values of the parameters, and calculate a 
bifurcation of solutions in terms of the dimensionless parameters $(\frac{E_0}{V_\infty} ,  \frac{\nu}{V_\infty})$. 
The diagram is computed by solving the  system \eqref{intro-ns} numerically and characterizing solutions according  to the zone they belong, see Sections \ref{sec:existence}
and \ref{sec:bifurcation}. 

In Appendix \ref{sec:cyl-coord} we list Navier-Stokes in cylindrical coordinates. As our study is motivated from the study of tornadoes, we present  for the convenience of the reader  in Appendix \ref{sec:tornado} a quick review of models used in the literature for tornado modeling.

%
\section{Preliminaries}
\label{sec:axi}

We consider the equations of motion for an incompressible homogeneous viscous fluid, with constant density $\rho=1$ :
\begin{subequations}
\label{NS}
\begin{align}
\vec{u}_t + (\vec{u} \cdot \nabla) \vec{u} &= - \nabla p + \nu \, \Delta \vec{u} , \\
\nabla\cdot \vec{u} & = 0,
\end{align}
\end{subequations}
where $\vec{u} : \mathbb{R}^3\times \mathbb{R}_+ \to \mathbb{R}^3$ is the velocity vector of the fluid, $p : \mathbb{R}^3 \times \mathbb{R}_+\to \mathbb{R}$ is the pressure and $\nu \geq 0$ is the coefficient of kinematic viscosity. The first equation represents the conservation of momentum 
while the second is the incompressibility condition and may be interpreted as describing conservation of mass. For $\nu>0$, the system \eqref{NS} is the so-called Navier-Stokes equations, while when viscocity effects are omitted the corresponding system \eqref{NS} with $\nu =0$ is the (incompressible) Euler equations. 

\subsection{Axisymmetric Navier - Stokes equations}
In three space dimensions, introducing cylindrical coordinates $(r,\vartheta,z)$:  $x_1 = r \,\cos\vartheta$, $x_2 = r\,\sin\vartheta$, $x_3 = z$, 
and the attached unit vector system 
\begin{equation*}
\vec{e}_r = (\cos\vartheta,\sin\vartheta,0), \,\  \vec{e}_\vartheta = (-\sin\vartheta,\cos\vartheta,0), \,\ \vec{e}_z = (0, 0, 1) \, , 
\end{equation*}
we express the velocity vector $\vec{u}$ in cylindrical coordinates as
\begin{equation*}
\vec{u} = u(r,\vartheta,z,t) \vec{e}_r + v(r,\vartheta,z,t) \vec{e}_{\vartheta} + w(r,\vartheta,z,t) \vec{e}_z.
\end{equation*}
A flow is called axisymmetric if the velocity  $\vec{u}$ does not depend on the azimuthal angle $\vartheta$, i.e.
\vspace{-5pt}
\begin{equation*}
\vec{u} = u(r,z,t) \vec{e}_r + v(r,z,t) \vec{e}_{\vartheta} + w(r,z,t) \vec{e}_z.
\vspace{-5pt}
\end{equation*}
The velocity component $v$ in the direction of $\vec{e}_\vartheta$ is called the swirl velocity. If the swirl is everywhere equal to zero, i.e. $v\equiv 0$, we name such flows as flows without swirl. Otherwise we call them axisymmetric flows with swirl.

Using that the velocity $\vec{u}$ does not depend on $\vartheta$, we derive the axisymmetric Navier-Stokes equations which have the following form
(see App \ref{sec:cyl-coord})
\begin{subequations}
\begin{align}
\label{tNS1}
\frac{\partial u}{\partial t} + u \frac{\partial u}{\partial r} + w \frac{\partial u}{\partial z} - \frac{v^2}{r} & = \nu \Big[\frac{1}{r}\frac{\partial }{\partial r} \Big(r \frac{\partial u}{\partial r}\Big)+ \frac{\partial^2 u}{\partial z^2} -  \frac{u}{r^2}  \Big] - \frac{\partial p}{\partial r}, \\
\label{tNS2}
\frac{\partial v}{\partial t} + u \frac{\partial v}{\partial r} + w \frac{\partial v}{\partial z} + \frac{uv}{r} & = \nu \Big[\frac{1}{r}\frac{\partial }{\partial r} \Big(r \frac{\partial v}{\partial r}\Big)+ \frac{\partial^2 v}{\partial z^2} -  \frac{v}{r^2}  \Big], \\
\label{tNS3}
\frac{\partial w}{\partial t} + u \frac{\partial w}{\partial r} + w \frac{\partial w}{\partial z} \qquad  & = \nu \Big[\frac{1}{r}\frac{\partial }{\partial r} \Big(r \frac{\partial w}{\partial r}\Big)+ \frac{\partial^2 w}{\partial z^2}  \qquad \Big] - \frac{\partial p}{\partial z}, \\
\label{tNS4}
\frac{1}{r} \frac{\partial }{\partial r} (ru) + \frac{\partial w}{\partial z}  & = 0.
\end{align}
\end{subequations}

For general axisymmetric flows, the vorticity vector $\vec{\omega} = \nabla \times \vec{u}$ is expressed as 
\begin{align*}
\vec{\omega} &= \big({\omega}_r, {\omega}_{\vartheta}, {\omega}_z\big) 
= \bigg(-\frac{\partial v}{\partial z}, \frac{\partial u}{\partial z}-\frac{\partial w}{\partial r}, \, \frac{1}{r}\frac{\partial}{\partial r} (rv) \bigg) \, ,
\end{align*}
where $\omega_\vartheta := \frac{\partial u}{\partial z}-\frac{\partial w}{\partial r}$ is the component of vorticity in the $\vec{e}_\vartheta$-direction. Note that
for flows {\it without swirl} only the vorticity in the direction of $\vec{e}_\vartheta$ survives: $\vec{\omega} = \omega_\vartheta \vec{e}_\vartheta$.

The continuity equation \eqref{tNS4} may be integrated using  the axisymmetric stream function $\psi(r,z,t)$ by expressing the velocity components $u$ and $w$ in terms of $\psi$ as follows
\begin{equation*}
u = - \frac{1}{r} \frac{\partial \psi}{\partial z} \quad \text{and} \quad  w= \frac{1}{r} \frac{\partial \psi}{\partial r}.
\end{equation*}
The stream function $\psi$ and the vorticity component $\omega_\vartheta$ are independent of the swirl velocity $v$; hence, an equivalent formulation of $\eqref{tNS1} - \eqref{tNS4}$ is given by
\begin{subequations}
\begin{align}
\label{V1}
\frac{\partial {\omega}_{\vartheta}}{\partial t} + u \frac{\partial {\omega}_{\vartheta}}{\partial r} + w \frac{\partial {\omega}_{\vartheta}}{\partial z} - {\omega}_{\vartheta} \frac{u}{r} & = \nu \Big[\Delta {\omega}_{\vartheta} -  \frac{ {\omega}_{\vartheta}}{r^2}  \Big]  + \frac{\partial }{\partial z} \Big(\frac{v^2}{r} \Big),\\
\label{V2}
\frac{\partial v}{\partial t} + u \frac{\partial v}{\partial r} + w \frac{\partial v}{\partial z} + v \frac{u}{r} & = \nu \Big[\Delta v -  \frac{v}{r^2}  \Big], \\
\label{V3}
 {\omega}_{\vartheta} & = - \frac{1}{r} \Delta \psi + \frac{2}{r^2} \frac{\partial \psi}{\partial r},   \\
\label{V4}
\Big( u,w \Big) & =  \frac{1}{r} \nabla^{\perp} \psi,
\end{align}
\end{subequations}
in terms of $\omega_\theta$,  $v$ and  $\psi$, where
\begin{equation*}
\Delta f  = \frac{1}{r}\frac{\partial }{\partial r} \Big(r \frac{\partial f}{\partial r}\Big)+ \frac{\partial^2  f}{\partial z^2}  \qquad \text{and} \qquad \nabla^{\perp} f =  \Bigg(- \frac{\partial f}{\partial z} ,  \frac{\partial f}{\partial r} \Bigg).
\end{equation*}
This formulation is equivalent with the three-dimensional Navier-Stokes equations, \cite{Majda}.

\subsection{Self-Similar Formulation} 
In the analysis of partial differential equations it is often beneficial to seek solutions that conform with symmetry properties of the problem, for instance 
invariance under rotations, dilations, etc. The invariance of the Navier-Stokes equations under the scaling
\begin{equation}
\vec{u}(t,r,z) = \lambda \, \vec{u}(\lambda^2 t,\lambda r, \lambda z) \quad \textrm{and} \quad p(t,r,z) = \lambda^2\, p(\lambda^2 t,\lambda r, \lambda z),
\end{equation}
for $\lambda>0$, suggests to look for self-similar solutions of \eqref{tNS1}-\eqref{tNS4} that follow the ansatz 
\begin{equation}
\label{ansatz-time}
\begin{aligned}
& u(t,r,z) = \frac{1}{r} U(s,\xi), \quad v(t,r,z) = \frac{1}{r} V(s,\xi), \quad  w(t,r,z) =\frac{1}{r}  W(s,\xi) , 
\\[5pt]
&  p(t,r,z) =\frac{1}{r^2}  P(s,\xi), \quad \psi(t,r,z) = r  \Psi(s,\xi),  \quad \omega_\vartheta(t,r,z) =\frac{1}{r^2}  \Omega(s,\xi), 
\end{aligned}
\end{equation}
in the variables $\xi = \frac{z}{r}$ and $s= \frac{r}{\sqrt{t}}$. Such an ansatz induces a singularity at $r=0$. In the literature related to tornadoes the
singularity is considered to represent the line vortex resembling the tornado core, \cite{Mor66}. Using the ansatz \eqref{ansatz-time} in \eqref{tNS1}-\eqref{tNS4}, 
a tedious but straightforward calculation yields the system of partial differential equations for $(U,V,W)$,
\begin{equation}
\begin{aligned}
s\Big(U-\frac{s^2}{2}\Big) U_s  + \Big(W-\xi U\Big) U_{\xi}  - \Big(U^2+V^2\Big) &=  \nu \Big [ \mathcal{H}U - U \Big ] + 2P + \xi P_{\xi} - sP_s, \\
s\Big(U-\frac{s^2}{2}\Big) V_s  + \Big(W-\xi U\Big) V_{\xi} \hspace{73pt}  &=  \nu \Big[ \mathcal{H}V - V \Big], \\
s\Big(U-\frac{s^2}{2}\Big) W_s  + \Big(W-\xi U\Big) W_{\xi} \hspace{20pt}  - \,\,  UW \hspace{7pt}&=  \nu \Big[ \mathcal{H}W \qquad \Big] - P_{\xi},  \\
sU_s + W_{\xi} - \xi U_{\xi} &= 0 ,
\end{aligned}
\end{equation}
where the operator $\mathcal{H}$ is defined as 
\begin{equation}
\label{defH}
\mathcal{H}f = \bigg[ -\xi f + \Big((1+\xi^2)f\Big)_\xi  \bigg]_\xi - sf_s - 2 s \xi f_{s\xi}  + s^2 f_{ss}.
\end{equation}
The same ansatz transforms the velocity-vorticity equations $\eqref{V1} - \eqref{V4}$ into the form 
\begin{equation}
\begin{aligned}
s\Big(U-\frac{s^2}{2}\Big) \Omega_s  + \big(W-\xi U\big) \Omega_{\xi}  - 3U\Omega  -2VV_{\xi} &=  \nu \big[ \mathcal{D}\Omega - \Omega \big], \\
s\Big(U-\frac{s^2}{2}\Big) V_s  + \big(W-\xi U\big) V_{\xi} \hspace{83pt} &=  \nu \big[ \mathcal{H}V - V \big],  \\
\Psi -\xi \Psi_{\xi} - s\Psi_s + 2s\xi \Psi_{s\xi}  - s^2 \Psi_{ss}  - (1+\xi^2)\Psi_{\xi\xi}   &= \quad \Omega, \\
\Big(- \Psi_{\xi} , \Psi - \xi \Psi_{\xi} + s \Psi_s \Big) & = \big( U,W \big), 
\end{aligned}
\end{equation}
where $\mathcal{H}$ is defined in \eqref{defH} and $\mathcal{D}$ is given by
\begin{equation}
\label{defD}
\mathcal{D}f =  4f +5 \xi f_{\xi} - 3 sf_s - 2 s \xi f_{s\xi}  + s^2 f_{ss} +(1+\xi^2)f_{\xi\xi}.
\end{equation}
The choice of ansatz is not unique and various self-similar transformations can be used for the Navier-Stokes equations, for instance, one may consider an ansatz in variables $\xi = \frac{z}{r}$ and $\tau= \frac{t}{r^2}$.


\subsection{Stationary Self-Similar Formulation}
The tornadic funnel typically moves slowly compared to the internal swirling velocities, \cite{Rot13}. This motivates to
represent the core of the tornado via a vortex singularity, following \cite{Long58,Gold60}, and 
to study stationary axi-symmetric self-similar solutions for the Navier-Stokes equations.
Letting $\xi = \frac{z}{r}$, we seek for self-similar stationary solutions of \eqref{tNS1}-\eqref{tNS4} in the form
\begin{equation}\label{ansatz}
\begin{aligned}
& u(r,z) = \frac{1}{r} U(\xi), \quad v(r,z) = \frac{1}{r} V(\xi), \quad  w(r,z) =\frac{1}{r}  W(\xi) , \\
&  p(r,z) =\frac{1}{r^2}  P(\xi), \quad \psi(r,z) = r  \Psi(\xi),  \quad \omega_\vartheta(r,z) =\frac{1}{r^2}  \Omega(\xi). 
\end{aligned}
\end{equation}
Such solutions satisfy the system of ordinary differential equations
\label{sss-ns}
\begin{subequations}
\begin{align}
\label{eq-1}
-U \Big(\xi U\Big)' + U' W - V^2 &=  \nu \Big[\mathcal{L} U - U \Big] + 2P + \xi P' ,\\
-U \Big(\xi V\Big)' + V' W + UV &=  \nu \Big[\mathcal{L} V - V \Big], \\
\label{eq-3}
-U \Big(\xi W\Big)' + W' W &=  \nu \mathcal{L} W  - P', \\
W' - \xi U' &= 0 
\end{align}
\end{subequations}
where $\xi \in \mathbb{R}_+$ and the operator $\mathcal{L}$ defined as 
\begin{equation}\label{defL}
\mathcal{L}f = \bigg[ -\xi f + \Big((1+\xi^2)f\Big)'  \bigg]  '  \, . 
\end{equation}
If we multiply $\eqref{eq-1}$ by $\xi$ and subtract it from \eqref{eq-3}, we can rewrite the system as
\begin{subequations}
\begin{align*}
\bigg[\frac{1}{2}\Big(W - \xi U\Big)^2 + (1+\xi^2)P \bigg]'  &=   \nu \big\{ \mathcal{L}W - \xi \mathcal{L}U + \xi U \big\} - \xi V^2 ,\\
V' \Big(W - \xi U\Big) &= \nu \big\{ \mathcal{L}V - V \big\},  \\
\bigg[W \Big(W - \xi U\Big) + P \bigg]' &= \nu \mathcal{L}W, \\
\Big(W - \xi U\Big)' &= -U .
\end{align*}
\end{subequations}

For convenience, we introduce a new function $\theta(\xi)$ with $\theta = W -\xi U$. One checks from \eqref{V4} and \eqref{ansatz} that $\theta (\xi) = \Psi (\xi)$
is the self-similar version of the stream function. Now, $U$ and $W$ 
are expressed in terms of $\theta$ by
\begin{equation}\label{stream}
U = - \theta'  \, , \qquad W = \theta - \xi \theta^\prime \, .
\end{equation} 
A tedious but straightforward computation renders \eqref{sss-ns} to a non-linear system of ordinary differential equations for $(\theta, V, P)$,
\begin{subequations}
\label{ssform}
\begin{align}
\bigg[\frac{\theta^2}{2} + (1+\xi^2)P \bigg]' &=  \nu \bigg[\xi\theta - (1+\xi^2) \theta' \bigg]' -\xi V^2, \\
V' \theta &= \nu \Big[3\xi  V' + (1+\xi^2) V''\Big],  \label{n2}  \\
\bigg[\theta^2 - \xi \Big(\frac{\theta^2}{2}\Big)' + P \bigg]' &= \nu \Big[\xi\theta - \xi^2 \theta' - \xi (1+\xi^2) \theta'' \Big]',  
\end{align}
\end{subequations}
where the viscosity $\nu$ is a parameter and $(U,W)$ are determined by \eqref{stream}. The aim is to study  \eqref{ssform} in both the viscous $\nu>0$ as well as the inviscid $\nu=0$ cases, and to investigate the limiting relationship. Sections \ref{sec:euler} and \ref{sec:disceuler} are devoted to the inviscid case $\nu=0$: In Section \ref{sec:euler}  an explicit solution for \eqref{ssform}  (with $\nu=0$) is obtained. The existence of solutions with slip discontinuities is examined in Section \ref{sec:disceuler}. In Section \ref{sec:NS}, we express the system \eqref{ssform} for $\nu>0$ into an equivalent integrodifferential formulation and study its limit as $\nu\to0$. Numerical solutions are illustrated in Section \ref{sec:numerics}.

\section{A stationary self-similar solution for the axisymmetric Euler equations}
\label{sec:euler}
Consider first the Euler equations  \eqref{ssform} with $\nu = 0$:
\begin{subequations}
\begin{align}
\label{theta-eq}
\bigg[\frac{\theta^2}{2} + (1+\xi^2)P \bigg]' &= -\xi V^2 \\
\label{v-eq}
V' \theta &= 0 \\
\label{p-eq}
\bigg[\theta^2 - \xi \Big(\frac{\theta^2}{2}\Big)' + P \bigg]' &= 0\\
\label{u-eq}
\theta' &= - U
\end{align}
\end{subequations}
Equation $\eqref{v-eq}$ implies that if $\theta(\xi) \ne 0$ then $V(\xi)$ is a constant function. So,  assume first that $\theta \neq 0$ for $\xi \in (0, \infty)$ and that 
\begin{equation}\label{eqnV}
V \equiv V_{0},
\end{equation}
where $V_{0}$ is a constant proportional to the circulation around the axis. (The case where $\theta$ vanishes and where $V$ may exhibit discontinuities is examined in the next section.)
Substituting \eqref{eqnV} yields 
\begin{subequations}
\begin{align}
\label{sseuler-eq1}
\xi (1+\xi^2) \Big(\frac{\theta^2}{2}\Big)' - (1+2 \xi^2) \Big(\frac{\theta^2}{2}\Big) &= -\frac{\xi^2}{2} {V_0}^2 - E_0 (1+ \xi^2)+ A_0 \\
P &=  \xi \Big(\frac{\theta^2}{2}\Big)' - \theta^2 + E_0 
\end{align}
\end{subequations}
where $E_{0}$ and $A_{0}$ are integration constants. Dividing equation \eqref{sseuler-eq1} by $\xi^2 (1+\xi^2)^{\frac{3}{2}}$ gives
\begin{equation*}
\Bigg[ \frac{1}{\xi \sqrt{1+\xi^2}} \frac{\theta^2}{2} \Bigg] ' = -\frac {{V_0}^2}{2} \Bigg[ \frac{\xi} {\sqrt{1+\xi^2}} \Bigg] '  + {E_0}  \Bigg[ \frac {\sqrt{1+\xi^2}} {\xi} \Bigg] '  - A_0  \Bigg[ \frac{1+2\xi^2} {\xi \sqrt{1+\xi^2}} \Bigg] ' 
\end{equation*}
After an integration, we eventually obtain an explicit form for the stream function $\theta$ and an explicit solution for the Euler system \eqref{theta-eq} - \eqref{u-eq}
\begin{subequations}
\label{euler-sol}
\begin{align} 
\frac{\theta^2}{2} &= - \Big(\frac{V_0^2}{2} + 2{A_0} - {E_0}\Big) \xi^2 + {k_0} \xi \sqrt{1+\xi^2} + \Big({E_0} - {A_0}\Big)  
\label{eqn11}
\\
U &= \frac{1}{\theta} \Bigg[ 2\Big(\frac{V_0^2}{2} + 2{A_0} - {E_0}\Big) \xi - k_0 \frac{1+2\xi^2} {\sqrt{1+\xi^2}} \Bigg]  
\label{eqn12}
\\
W &= \frac{1}{\theta} \Bigg[k_0 \frac{\xi} {\sqrt{1+\xi^2}} + 2\Big({E_0} - {A_0}\Big)\Bigg] \\
P &= - k_0 \frac{\xi} {\sqrt{1+\xi^2}} + 2{A_0} - {E_0} 
\end{align}
\end{subequations}
where $V(\xi)$ is given by \eqref{eqnV} and $E_0$, $A_{0}$ and $k_0$ are integration constants.

We supplement system \eqref{theta-eq}-\eqref{u-eq} with suitable boundary conditions. Recall that our objective is to model a vortex line that interacts with a boundary surface. We impose no-penetration boundary condition at the boundary, namely $w = 0$ at $z = 0$. Moreover, we assume that $u = 0$ as $r \to 0$ which ensures that no mass is added or lost through the vortex line. Expressing these conditions in terms of the self-similar functions \eqref{ansatz} leads to
\begin{equation}
W(0) = 0 \quad \textrm{and} \quad U(\xi) \to 0 \quad \textrm{as} \quad \xi \to \infty,
\end{equation}
or in terms of $\theta$
\begin{equation}
\label{bc-euler}
\theta(0) = 0 \quad \textrm{and} \quad \theta'(\xi) \to 0 \quad \textrm{as} \quad \xi \to \infty \, .
\end{equation}
The boundary conditions applied to \eqref{euler-sol} provide the relations
\begin{subequations}
\begin{align}
\label{BC1}
{E_0} &= {A_0} \\
\label{BC2}
k_0  &=   \frac{{V_0}^2}{2} + {E_0} 
\end{align}
\end{subequations}
The derivation of \eqref{BC1}  is direct. To derive relation \eqref{BC2}, note that \eqref{eqn11}, \eqref{eqn12} imply
\begin{align*}
U = \frac{1}{\theta} \Bigg[ 2\Big(\frac{V_0^2}{2} + {E_0}\Big) \xi - k_0 \frac{1+2\xi^2} {\sqrt{1+\xi^2}} \Bigg] 
&=\frac{ 2\Big(\frac{V_0^2}{2} + {E_0}\Big) - k_0 \frac{1+2\xi^2} {\xi\sqrt{1+\xi^2}}}{\pm\sqrt{- 2\Big(\frac{V_0^2}{2} + {E_0}\Big)  + 2{k_0}\frac{\sqrt{1+\xi^2}} {\xi}}} 
\\[8pt]
&\xrightarrow{\xi\to\infty}  
\mp {\sqrt{- 2\Big(\frac{V_0^2}{2} + {E_0}\Big)  + 2{k_0}}} 
\end{align*}
Thus, as $\xi\to\infty$
 \[ U (\infty) = 0 =\mp{\sqrt{- 2\Big(\frac{V_0^2}{2} + {E_0}\Big)  + 2{k_0}}} \quad \]
 implies \eqref{BC2}.

Substituting \eqref{BC1}-\eqref{BC2}  into \eqref{euler-sol}, we obtain the explicit formula
\begin{equation}
\label{euler-sol-theta}
\theta^2  =  2{k_0}  \ \phi(\xi)  
\end{equation}
with $\phi$ the function
\begin{equation}
\label{defphi}
\phi(\xi) :=  \xi \big( \sqrt{1+\xi^2} - \xi \big) 
\end{equation}

\begin{lemma}{\bf [Properties of $\phi(\xi)$]}.
\label{phi}
The function $\phi(\xi)$ in \eqref{defphi} is shown in Figure \ref{fig:phi} and has the properties:
\begin{enumerate}
\item $\phi(\xi)$ is non-negative and bounded, 
$0 < \phi(\xi) < \frac{1}{2} \quad \forall \quad 0<\xi<\infty$.
\item $\phi(0) = 0$ and $\lim_{\xi\to \infty} \phi(\xi) =  \frac{1}{2}$
\item $\phi(\xi)$ is increasing and concave with
\begin{align*}
\phi'(\xi) = \frac{1}{\sqrt{1+\xi^2}} \big(1-2 \phi(\xi) \big) >0, \quad
\phi''(\xi) = - \frac{1-2 \phi(\xi)}{(1+\xi^2)^\frac{3}{2}} \big(\xi+2\sqrt{1+\xi^2} \big) < 0
\end{align*}
\end{enumerate} 
\end{lemma}
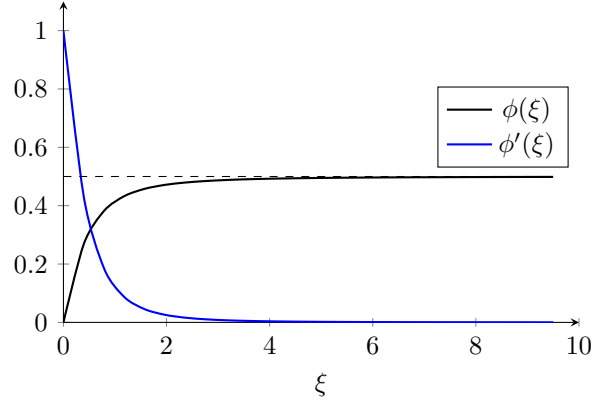
\begin{figure}[htbp]
\centering
\begin{tikzpicture}
  \begin{axis} [axis lines=left, 
  xlabel = $\xi$,
  xmin=0, xmax=10,
    ymin=0, ymax=1.1, yscale=.75
]
        \addplot [domain=0:9.5, smooth, thick] { x*sqrt(1+x^2) - x*x }; \addlegendentry{$\phi(\xi)$}
         \addplot [domain=0:9.5, smooth, blue, thick] {(1-2* (x*sqrt(1+x^2) - x*x))/sqrt(1+x^2)}; \addlegendentry{$\phi'(\xi)$}
           \addplot [domain=0:9.5, smooth, dashed,thin] { 0.5 };
  \end{axis}
\end{tikzpicture}
\caption{Functions $\phi(\xi)$  and $\phi'(\xi)$} \label{fig:phi}
\end{figure}
The sign of $\phi$ provides a constraint on the existence of $\theta$. From \eqref{euler-sol-theta}, $\theta(\xi)$ is well-defined iff $k_0 = E_0 + \frac{V_0^2}{2}>0$. Under this restriction, we derive an explicit family of solutions depending on two parameters $V_0$,  $E_0$
\begin{equation}\label{eulersolution}
\begin{aligned}
&\theta = \pm \sqrt{2 {k_0} \phi(\xi)} \quad 
&&U = \mp \sqrt{\frac{k_0}{2\phi(\xi)}} \Bigg[ \frac{1-2\phi(\xi)} {\sqrt{1+\xi^2}} \Bigg]  \qquad 
&&&V = V_{0} \\
&W = \pm \sqrt{\frac{k_0}{2\phi(\xi)}} \frac{\xi} {\sqrt{1+\xi^2}}  \quad 
&&P = \bigg(k_0 - \frac{V_0^2}{2}\bigg) - k_0 \frac{\xi} {\sqrt{1+\xi^2}}  \quad &&&
\end{aligned} 
\end{equation}
Observe that $\theta(\xi)$ can be either positive or negative. If $\theta$ is positive, then the flow is directed inward near the plane $z = 0$ and upward near the vortex line. Conversely, if $\theta$ is negative, the flow is directed outward near the plane $z = 0$ and downward near the vortex line, see Figure \ref{fig:th}. 
\begin{figure}[htbp]
   \begin{subfigure}{0.5\textwidth}
     \centering
     \includegraphics[scale=0.5]{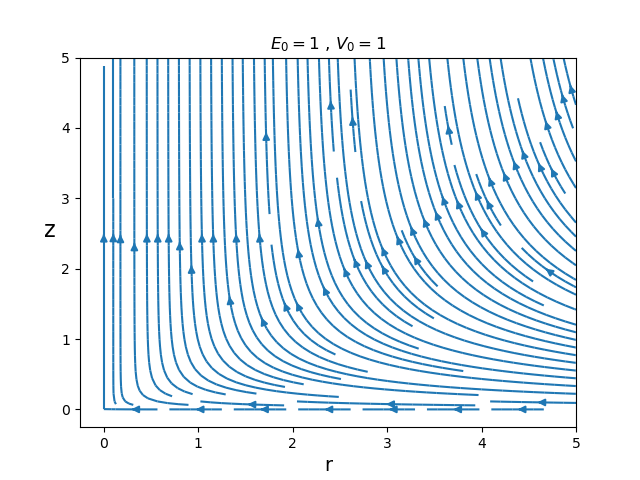}
      \caption{$\theta>0$}
   \end{subfigure}\hfill
   \begin{subfigure}{0.5\textwidth}
     \centering
     \includegraphics[scale=0.5]{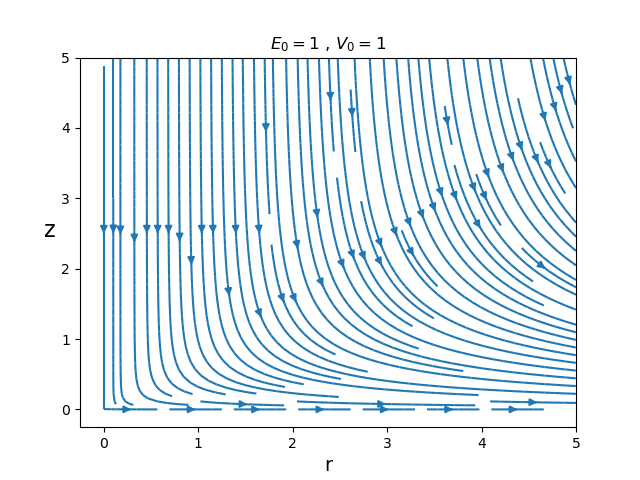}
     \caption{$\theta<0$}
   \end{subfigure}
    \caption{Velocity vector field $(u,w)$ in $(r,z)$ plane.} \label{fig:th}
\end{figure}

Note that
$P(\xi=0) = E_0 $ and $P(\xi=\infty)= - \frac{V_0^2}{2}$ and, as $P(\xi)$ is decreasing, the values of $P$ range 
from $-\frac{V_0^2}{2}$ to $E_0$. The constraint $k_0 >0$ may be interpreted as
\begin{equation*}
k_0 = E_0 + \frac{V_0^2}{2} = P(\xi=0) - P(\xi=\infty)>0 \, .
\end{equation*}
As $p = \frac{P}{r^2}$ and $v = \frac{V}{r}$ the constraint  may be thought as presenting a difference 
between  pressure induced forces  at the top of the vortex core and the intersection of the vortex core with the boundary.
$V_0$ is a measure of the circulation around the vortex core.
\begin{figure}[htbp]
\centering
\includegraphics[width=.45\linewidth]{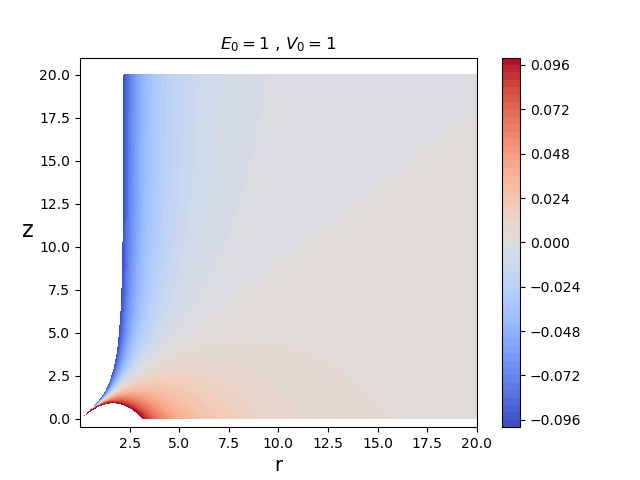}
\includegraphics[width=.45\linewidth]{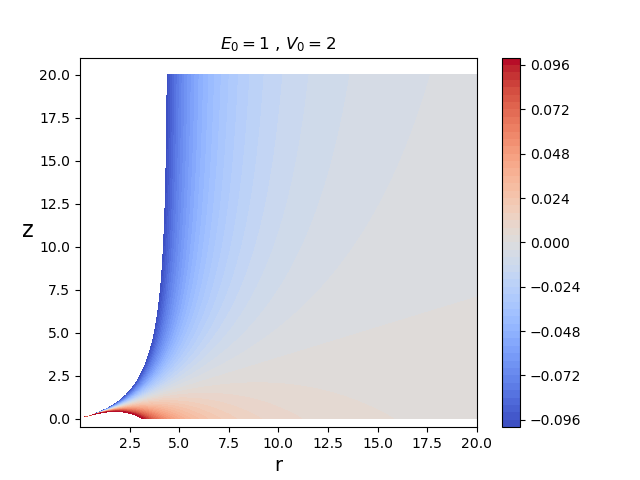}
\includegraphics[width=.45\linewidth]{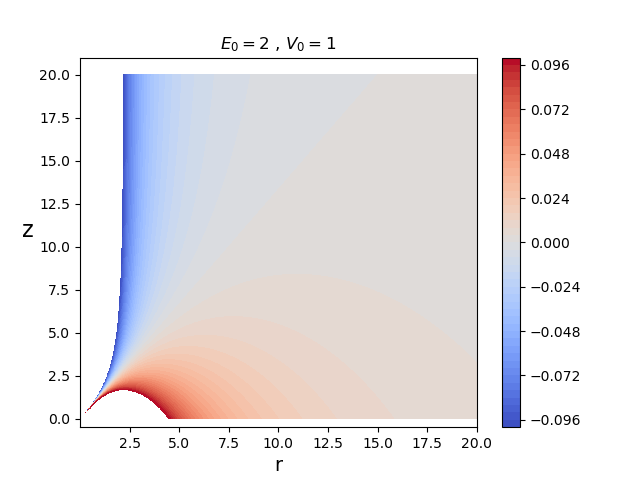}
\includegraphics[width=.45\linewidth]{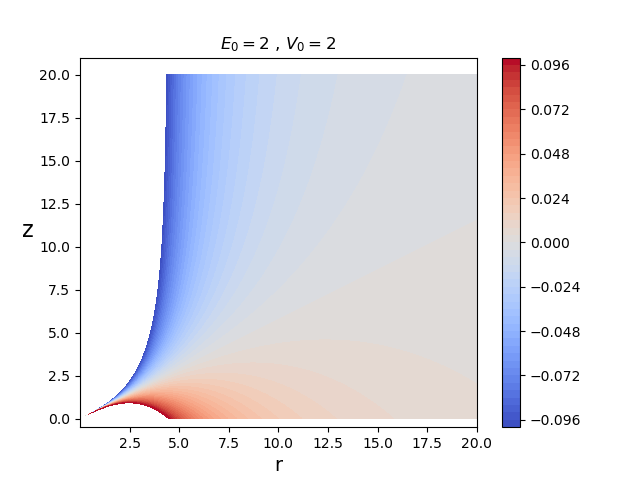}
\caption{Contour plots of pressure $p$ in $(r,z)$ plane for $V_0 = 1,2$ and $E_0=1,2$}
\label{fig:pres}
\end{figure}
Returning to the original variables through \eqref{ansatz} we obtain
\begin{equation}
\begin{aligned}
&u(r,z) = \mp \sqrt{\frac{k_0}{2}}  \frac{  (\sqrt{r^2 + z^2} - z )^2  }{ r \sqrt{  z  (r^2 + z^2) (\sqrt{r^2 + z^2} -z) }}
\quad
&&v(r,z) = \frac{V_0}{r}
\\
&w(r,z) = \pm \sqrt{\frac{k_0}{2}}  \frac{  z  }{ \sqrt{  z  (r^2 + z^2) (\sqrt{r^2 + z^2} -z) }} 
\quad
&&p(r,z) =  \frac{1}{r^2} \Big ( \big (k_0 - \frac{V_0^2}{2}\big) - k_0 \frac{z} {\sqrt{r^2 + z^2}} \Big )
\end{aligned}
\end{equation}
The streamlines of the vector field $(u,w)$ are sketched in Figure \ref{fig:th}, showing the two type of flows. In  Figure \ref{fig:pres} contour plots of the pressure are shown for four different choices of the parameters $E_0, \ V_0$.
\section{Do the axisymmetric Euler equations admit self-similar discontinuous solutions?}\label{sec:disceuler}
We consider the Euler equations \eqref{theta-eq} - \eqref{u-eq} and focus on solutions where $\theta$ vanishes and $V$ exhibits discontinuities. We study
weak solutions for the system \eqref{theta-eq} - \eqref{u-eq} and investigate the existence of discontinuous solutions subject to the boundary conditions used before.
\subsection{Discontinuous solutions - Weak Formulation}
Let $\varphi \in C^{\infty}_c ((0,\infty))$. Multiplying the equations \eqref{theta-eq} - \eqref{u-eq} by $\varphi$ and integrating by parts, we obtain
\begin{subequations}
\label{weak1}
\begin{align}
\label{weak-eq1}
-\int_0^\infty \bigg(\frac{\theta^2}{2} + (1+\xi^2)P \bigg) \varphi'(\xi) d\xi &=  -\int_0^\infty \xi V^2(\xi) \varphi(\xi) d\xi, \\
\label{weak-eq2}
-\int_0^\infty \bigg(\theta(\xi) V(\xi)\bigg) \varphi'(\xi) d\xi &= -\int_0^\infty U(\xi) V(\xi) \varphi(\xi) d\xi, \ \\
\label{weak-eq3}
-\int_0^\infty \bigg(\theta^2 - \xi \Big(\frac{\theta^2}{2}\Big)' + P \bigg) \varphi'(\xi) d\xi &= 0,\\
\label{weak-eqU}
-\int_0^\infty \theta \varphi'(\xi) d\xi &= -\int_0^\infty U(\xi) \varphi(\xi) d\xi.
\end{align}
\end{subequations}

\begin{definition}
\label{weak-def}
The function $(\theta,V,P)$ of class $\theta\in W^{1,1}((0,\infty))$, $V, P \in BV((0,\infty)) \cap L^{\infty}((0,\infty))$ is a weak solution of the system \eqref{theta-eq} - \eqref{u-eq} if \eqref{weak1} holds for any $\varphi \in C^{\infty}_c ((0,\infty))$.
\end{definition}
In the definition \ref{weak-def}, $BV$ denotes the space of functions of bounded variation. This property ensures that right and left limits of $(\theta, V, P)$ exist at any point $\xi$.
Since $\theta$ belongs to $W^{1,1}((0,\infty))$, then $\theta$ is absolutely continuous and of bounded variation. The derivative of $\theta$, i.e. $\theta'=-U$, exists almost everywhere and is Lebesgue integrable on $(0,\infty)$.

Suppose $(\theta,V,P)$ is a weak solution of \eqref{theta-eq} - \eqref{u-eq}. Using the theory of Sobolev spaces, \cite{Bre2010},  \cite{Foll2013}, one may explore properties of system \eqref{weak1}. Recalling that $\theta$ and as a consequence $\theta^2$ are absolutely continuous, we get that $\theta^2 + \xi \Big(\frac{\theta^2}{2}\Big)' \in L^1_{loc}((0,\infty))$. Also, $P \in L^{\infty}((0,\infty)) \subset  L^1_{loc}((0,\infty))$. Then, using \cite[Lemma 8.1]{Bre2010}, \eqref{weak-eq3} implies
there exists a constant $A$ such that
\begin{equation}
\label{f1}
\theta^2 - \xi \Big(\frac{\theta^2}{2}\Big)' + P = A,
\end{equation}
almost everywhere (a.e). The same lemma applied to \eqref{weak-eq1},  \eqref{weak-eq2} and \eqref{weak-eqU} as follows. Using \cite[Lemma 8.2]{Bre2010} 
we rewrite \eqref{weak-eq1},  \eqref{weak-eq2} and \eqref{weak-eqU} in the form
\begin{subequations}
\begin{align}
&\quad -\int_0^\infty \bigg(\frac{\theta^2}{2} + (1+\xi^2)P + \int_0^\xi \zeta V^2(\zeta) d\zeta\bigg) \varphi'(\xi) d\xi = 0. \\
&\quad -\int_0^\infty \bigg(\theta(\xi) V(\xi) + \int_0^\xi U(\zeta) V(\zeta) d\zeta \bigg) \varphi'(\xi) d\xi = 0.  \\
&\quad-\int_0^\infty \bigg(\theta + \int_0^\xi U(\zeta) d\zeta\bigg) \varphi'(\xi) d\xi = 0 \quad \forall \, \varphi.
\end{align}
\end{subequations}
The integrals $\int_0^\xi \zeta V^2(\zeta) d\zeta$, $\int_0^\xi U(\zeta) V(\zeta) d\zeta$ and $\int_0^\xi U(\zeta) d\zeta$ for $U=-\theta' \in L^1$ are well-defined.
Using \cite[Lemma 8.1]{Bre2010}, there are constants $B$, $C$ and $D$ such that
\begin{subequations}
\begin{align}
\label{f2}
\frac{\theta^2}{2} + (1+\xi^2)P + \int_0^\xi \zeta V^2(\zeta) d\zeta = B \quad\textrm{a.e.}, \\
\label{f3}
\theta(\xi) V(\xi) + \int_0^\xi U(\zeta) V(\zeta) d\zeta = C \quad\textrm{a.e.}, \\
\label{f4}
\theta(\xi)  + \int_0^\xi U(\zeta) d\zeta = D \quad\textrm{a.e.}.
\end{align}
\end{subequations}
Letting $\xi\to a_-$ and $\xi\to b_+$ for $0<a,b<\infty$ in \eqref{f1} and \eqref{f2} -\eqref{f4}, one may conclude that $(\theta,V,P)$ a weak solution of
 \eqref{theta-eq} - \eqref{p-eq} satisfies for $a<b$
\begin{subequations}
\label{weak2}
\begin{align}
\bigg(\frac{\theta^2(\xi)}{2} + (1+\xi^2)P(\xi)  \bigg)\Bigg|_{a-}^{b_+}  &= - \int_a^b \zeta V^2(\zeta) d\zeta, \\
\bigg(\theta(\xi) V(\xi)\bigg) \Bigg|_{a-}^{b_+} &= - \int_a^b U(\xi) V(\xi) d\xi, \ \\
\bigg(\theta^2 - \xi \Big(\frac{\theta^2}{2}\Big)' + P(\xi) \bigg)\Bigg|_{a-}^{b_+} &= 0, \\
\theta(\xi) \Big|_{a-}^{b_+} &= - \int_a^b U(\xi)  d\xi,
\end{align}
\end{subequations}

\subsubsection*{Jump Discontinuities}
Let $(\theta,V,P)$ be a weak solution of \eqref{theta-eq} - \eqref{u-eq} defined by definition \ref{weak-def}. Utilizing  \eqref{weak2}, we will compute the Rankine-Hugoniot jump conditions associated with the system. Suppose there exists a discontinuity at some point $\xi = \sigma$, $0<\sigma<\infty$ as in Figure \ref{fig:disc}. Letting $a,b \to \sigma$, \eqref{weak2} reduces to
\begin{subequations}
\label{weak3}
\begin{align}
\label{w1}
\frac{1}{2}\Big(\theta_+^2 - \theta_-^2\Big) + (1+\sigma^2)\Big(P_+ - P_-\Big)  &= 0, \\
\label{w2}
\theta_+ V_+ - \theta_- V_- &= 0,  \\
\label{w3}
\bigg(\theta_+^2 - \sigma \Big(\frac{\theta_+^2}{2}\Big)' - \Big(\theta_-^2 - \sigma \Big(\frac{\theta_-^2}{2}\Big)'\Big) \bigg)  + \bigg(P_+ - P_- \bigg)  &= 0, \\
\label{w4}
\theta_+ - \theta_-  &= 0,
\end{align}
\end{subequations}
where $\theta_\pm = \theta(\sigma\pm)$, $V_\pm = V(\sigma\pm)$ and $P_\pm = P(\sigma\pm)$ denote the one-sided limits at $\xi =\sigma$. Note that all one-sided limits exist and are finite since $(\theta,V,P)$ are of bounded variation.

The last equation in \eqref{weak3} implies that $\theta$ must be continuous for any $\xi \in (0, \infty)$, that is to say
\begin{equation*}
\theta_- =\theta_+,
\end{equation*}
and thus, \eqref{w1}-\eqref{w3} provide the Rankine-Hugoniot jump conditions
\vspace{-7pt}
\begin{subequations}
\label{jump}
\begin{align}
\label{jump-P}
\llbracket P  \rrbracket  = 0, \\
\llbracket \theta V  \rrbracket = 0, \\
\label{jump-th}
\llbracket \theta \, \theta'  \rrbracket = 0,
\end{align}
\end{subequations}
where $\llbracket F \rrbracket = F_+ - F_-$ is the jump operator. It is easy to see that $P(\xi)$ must also be continuous for any $\xi \in (0, \infty)$. On the other hand, for ${\theta}(\sigma)=0$ \eqref{jump} implies that $V(\xi)$ and $\theta'(\xi)$ have a jump discontinuity at $\xi = \sigma$. (The case where $\theta(\sigma) \neq 0$ implies that $V$ is continuous and leads to solutions described in Section 3.)
\vspace{20pt}
\begin{figure}[htbp]
\centering
\begin{tikzpicture}[xscale=5, yscale=20]
\draw [<->] (0,0.2) -- (0,0) -- (1.1,0); 
\node [left] at (0,0.2) {$z$}; \node [left] at (0,0.1) {$\xi\to\infty$};
\node [below] at (1.1,0) {$r$};
\draw [blue] (0,0) -- (1,0.15); \node [right] at (1,0.15) {$\xi = \sigma$}; 
\node [right] at (0.6,0.07) {${\theta}_-(\sigma) = 0$}; \node [left] at (0.5,0.1) {${\theta}_+(\sigma) = 0$}; 
\node [below] at (0.5,0) {$\xi=0$}; 
\end{tikzpicture}
\caption{Domain of solutions with discontinuity} \label{fig:disc}
\end{figure}
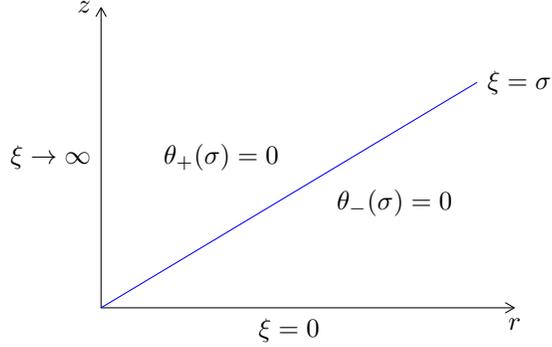

\subsubsection{Non-Existence of Discontinuous Solutions}
Consider $(\theta,V,P)$ a weak solution of \eqref{theta-eq} - \eqref{p-eq} given by definition \eqref{weak-def} with a discontinuity at a fixed point $\xi = \sigma$, $\sigma \in (0,\infty)$. Recall that $\theta'$ and $V$ may exhibit discontinuities, it suggests looking for solutions in the form
\begin{subequations}
\begin{align*}
\theta(\xi) =
\left\{
\begin{aligned}
& \theta_{-}(\xi), \\
&\theta_{+}(\xi) 
\end{aligned}
\right. 
\quad \text{and} \quad
V(\xi) =
\left\{
\begin{aligned}
& V_-(\xi) \,\, , \quad &&\xi \in (0,\sigma) \\
& V_+ (\xi)\,\, , \quad &&\xi \in (\sigma,\infty)
\end{aligned}
\right.
\end{align*}
\end{subequations}
We solve the system \eqref{theta-eq} - \eqref{p-eq} on $(0,\sigma)$ and $(\sigma,\infty)$ independently. Similar calculations to those described in Section \ref{sec:euler} lead again to a solution $\theta$ in form \eqref{eqn11} on each domain. Note that $V(\xi)$ is again a constant function with $V_-$ and $V_+$ constants. Since discontinuous solutions occur when ${\theta}(\sigma)=0$, the relation for $\theta$ becomes
\begin{equation}
\label{th-disc1}
\frac{\theta_\pm^2}{2} = {k_\pm} \bigg(\phi(\xi) - \phi(\sigma)\bigg) + \Big({k_\pm}-\frac{V_\pm^2}{2} - 2{A_\pm} + {E_\pm}\Big) \big(\xi^2 - \sigma^2\big).
\end{equation} 
where $k_\pm, A_\pm$ and $E_\pm$ are integration constants. We next impose the boundary conditions \eqref{bc-euler}, that is
\[
{\theta}_-(0) = 0 \quad \textrm{and} \quad {\theta'_+}(\xi) \to 0 \quad \textrm{as} \quad \xi \to \infty,
\]
which imply
\begin{align}
\label{bc-ds1}
{k_-}-\frac{V_-^2}{2} - 2{A_-} + {E_-} &= - {k_-} \frac{1}{\sigma^2}  \phi(\sigma), \\
\label{bc-ds2}
k_+ - \frac{{V_+}^2}{2} - 2{A_+} + {E_+} &= 0.
\end{align}
The derivation of \eqref{bc-ds1} is direct. To derive relation \eqref{bc-ds2}, observe that \eqref{th-disc1} yields
\begin{align*}
\theta'_+ &= \frac{1}{\theta} \bigg[k_+ \phi'(\xi) + 2\bigg(k_+ - \frac{V_+^2}{2}- 2A_+ + E_+\bigg) \xi \bigg] \\
&= \frac{k_+ \frac{\phi'(\xi)}{\xi} + 2\bigg(k_+ - \frac{V_+^2}{2}- 2A_+ + E_+\bigg)} {\pm\sqrt{2k_+ \frac{\phi(\xi)-\phi(\sigma)}{\xi^2} + 2\bigg(k_+ - \frac{V_+^2}{2}- 2A_+ + E_+\bigg) \frac{\xi^2-\sigma^2}{\xi^2}}}\\
&\xrightarrow []{\xi\to\infty}  \pm \sqrt{2\bigg(k_+ - \frac{V_+^2}{2}- 2A_+ + E_+\bigg)} ,
\end{align*}
Therefore, as $\xi \to \infty$
\[\theta'(\infty) = 0 = k_+ - \frac{{V_+}^2}{2} - 2{A_+} + {E_+},\]
provides \eqref{bc-ds2}.
Substituting now \eqref{bc-ds1}-\eqref{bc-ds2} into \eqref{th-disc1}, we obtain the explicit formula
\begin{align}
\label{th-disc2}
\frac{\theta^2}{2}(\xi) =
\left\{
\begin{aligned}
&{k_-} \Bigg[\phi(\xi) - \phi(\sigma)  -  \frac{\xi^2-\sigma^2}{\sigma^2}\phi(\sigma)\Bigg] \,\, , \quad &&\xi \in (0,\sigma) \\
&{k_+} \bigg[ \phi(\xi) - \phi(\sigma) \bigg] \,\, , \quad &&\xi \in (\sigma,\infty)
\end{aligned}
\right.
\end{align}
From \eqref{euler-sol} and \eqref{th-disc2}, we then derive an explicit family of discontinuous solutions depending on parameters $V_\pm, E_\pm$ and $\sigma$ 
\begin{subequations}
\label{euler-disc}
\begin{align}
&U =
\left\{
\begin{aligned}
&\frac{k_-}{\theta_-} \,\, \Bigg[\frac{2\phi(\sigma)}{\sigma^2} \xi -  \phi'(\xi)\Bigg]  \\
& - \frac{k_+}{\theta_+} \,\, \phi'(\xi)
\end{aligned}
\right.\, \, && \qquad
W =
\left\{
\begin{aligned}
& \frac{k_-}{\theta_-} \ \frac{\xi} {\sqrt{1+\xi^2}} \,\, , \quad &&\xi \in (0,\sigma) \\
&\frac{k_+}{\theta_+} \Bigg[\frac{\xi} {\sqrt{1+\xi^2}} - 2 \phi(\sigma)\Bigg] \,\, , \quad &&\xi \in (\sigma,\infty)
\end{aligned}
\right.\\
&V =
\left\{
\begin{aligned}
& V_-  \\
& V_+ 
\end{aligned}
\right.\qquad && \qquad
P = 
\left\{
\begin{aligned}
&  E_- - k_- \frac{\xi} {\sqrt{1+\xi^2}}  \,\, , \quad &&\xi \in (0,\sigma) \\
& E_+ - {k_+} \Bigg(\frac{\xi} {\sqrt{1+\xi^2}} - 2\phi(\sigma) \Bigg) \,\, , \quad &&\xi \in (\sigma,\infty)
\end{aligned}
\right.
\end{align}
\end{subequations}
where $E_- = k_-\big(1 + \frac{\phi(\sigma)}{\sigma^2}\big) - \frac{V_-^2}{2}$ and $E_+ = k_+ ( 1-2\phi(\sigma)) -\frac{V_+^2}{2} $. Although solutions $U(\xi)$ and $W(\xi)$ tend to infinity near the discontinuity at $\xi=\sigma$, they achieve the required regularity in definition \eqref{weak-def}. In particular, examining the denominator yields
\begin{subequations}
\begin{align*}
U(\xi) \sim W(\xi) \sim
\left\{
\begin{aligned}
&\frac{k_-}{\theta_-} \sim \sqrt{\frac{k_-}{\Big(\phi(\xi) - \phi(\sigma)\Big) -  \frac{\phi(\sigma)}{\sigma^2} \Big(\xi^2 - \sigma^2\Big)}}\sim \frac{\sqrt{k_-}}{\sqrt{\lvert \xi -\sigma \rvert}}  \,\, , \quad &&\textrm{as} \quad \xi \to \sigma_- \\
& \frac{k_+}{\theta_+} \sim \sqrt{\frac{k_+}{\Big(\phi(\xi) - \phi(\sigma)\Big)}} \sim \frac{\sqrt{k_+}}{\sqrt{\lvert \xi -\sigma \rvert}} \,\, , \quad &&\textrm{as} \quad \xi \to \sigma_+
\end{aligned}
\right.
\end{align*}
\end{subequations}
and therefore the singularities $\xi=\sigma$ are integrable, i.e. $U,W \in L^1_{loc}((0,\infty))$. Analysis of these singularities leads to the following theorem.

\begin{theorem}
\label{theorem1}
Let $(\theta, V, P)$ be a weak solution of \eqref{theta-eq} - \eqref{p-eq} of class $\theta\in W^{1,1}((0,\infty))$, $V,P \in BV((0,\infty)) \cap L^{\infty}((0,\infty))$ which satisfies the boundary conditions
\begin{equation*}
V(0+)=V_0, \quad P(0+) = E_0, \quad \theta(0) = 0, \quad \theta'(\infty)=0.
\end{equation*}
There does not exist a solution $(\theta, V, P)$ with a discontinuity at a single point that fulfils jump conditions \eqref{jump}. 
\end{theorem}

\begin{proof}
Suppose there exists a weak solution $\theta$ given by \eqref{th-disc2} that satisfies jump conditions \eqref{jump}. Consider the condition \eqref{jump-th} applied to \eqref{th-disc2}, namely
\[\theta_+(\sigma)\theta'_+(\sigma) =  \theta_-(\sigma) \theta'_-(\sigma).\]
It provides a constraint on the constants $k_+,k_-$
\begin{align}
\label{eq1}
\frac{k_+}{k_-} &= 1 - 2\frac{\phi(\sigma)}{\sigma \, \phi'(\sigma)}<0,
\end{align} 
which implies that $k_+$ and $k_-$ must have different signs. The derivation of \eqref{eq1} is direct.

Next, we determine the signs of $k_+,k_-$. Since $\theta^2 (\xi) > 0$ for $\xi \ne \sigma$, relation \eqref{th-disc2} imposes restrictions on the signs of $k_+,k_-$. Recall Lemma \eqref{phi}, the relation \eqref{th-disc2} implies that $k_+>0$. To find the sign of $k_-$, set
\begin{align}
\label{eq-J}
J(\xi) = \phi(\xi) -  \phi(\sigma) -  \frac{\phi(\sigma)}{\sigma^2} (\xi^2 - \sigma^2),
\end{align}
for $\xi\in(0,\sigma)$ and observe that $J(0)= J(\sigma) = 0$ and $J''<0$. Hence $J(\xi)>0$, $\forall \xi\in(0,\sigma)$ and $k_->0$. This contradicts with \eqref{eq1} and thus, there does not exist a discontinuous solution $(\theta, V, P)$ that fulfils the jump conditions \eqref{jump}.
\end{proof}

\subsubsection{Non-Existence of Multiple Discontinuities}
Let $(\theta,V,P)$ be a weak solution of \eqref{theta-eq} - \eqref{u-eq} with discontinuities at multiple points $\xi = \sigma_i$, $0<\sigma_i <\infty$, $i>0$. We consider first the case where the solution exhibits discontinuities at two points $\sigma_1$ and $\sigma_2$, $\sigma_1<\sigma_2$. This suggests seeking solutions in the form
\begin{subequations}
\begin{align*}
\theta (\xi) =
\left\{
\begin{aligned}
& \theta_{1}(\xi) \\
&\theta_{2}(\xi) \\
&\theta_{3}(\xi) 
\end{aligned}
\right. 
\quad \text{and} \quad
V (\xi)=
\left\{
\begin{aligned}
& V_1 (\xi)\,\, , \quad &&\xi \in (0,\sigma_1) \\
& V_2 (\xi)\,\, , \quad &&\xi \in (\sigma_1,\sigma_2)  \\
& V_3 (\xi) \,\, , \quad &&\xi \in (\sigma_2,\infty)
\end{aligned}
\right.
\end{align*}
\end{subequations}
To examine whether such solutions exist, it is sufficient to derive the solution $\theta$ by solving system \eqref{theta-eq} - \eqref{p-eq} on each domain independently. Same calculations as those described in Section \ref{sec:euler} lead to functions $\theta_1, \theta_2$ and $\theta_3$ of the form \eqref{eqn11}, that is
\begin{equation}
\label{eq6}
\frac{\theta_j^2}{2} = {k_j} \phi(\xi) + \Big({k_j}-\frac{V_j^2}{2} - 2{A_j} + {E_j}\Big) \xi^2 + \Big({E_j} - {A_j}\Big), \quad j=1,2,3
\end{equation}
where $A_j,E_j, k_j$ are integration constants. Note that $V(\xi)$ is again a constant function with $V_1, V_2$ and $V_3$ constants. Recall that discontinuous solutions occur when $\theta(\sigma_1)=\theta(\sigma_2)=0$, it implies
\[\theta_1(\sigma_1) = \theta_2(\sigma_1) = \theta_2(\sigma_2) = \theta_3(\sigma_2) = 0,\]
and thus, $\theta$ takes the following form
\begin{equation}
\label{eq7}
\frac{\theta^2}{2} =
\left\{
\begin{aligned}
&\frac{\theta_{1}^2}{2} = {k_1} \Bigg[\phi(\xi) - \phi(\sigma_1) + \Big({k_1}-\frac{V_1^2}{2} - 2{A_1} + {E_1}\Big) (\xi^2-\sigma_1^2)\Bigg] \,\, , \quad &&\xi \in (0,\sigma_1) \\
&\frac{\theta_{2}^2}{2} = {k_2} \Bigg[\phi(\xi) - \phi(\sigma_1)  -  \frac{\phi(\sigma_2)-\phi(\sigma_1)}{\sigma_2^2 - \sigma_1^2} (\xi^2-\sigma_1^2) \Bigg] \,\, , \quad &&\xi \in (\sigma_1,\sigma_2) \\
&\frac{\theta_{3}^2}{2}= {k_3} \bigg[ \phi(\xi) - \phi(\sigma_2) + \Big({k_3}-\frac{V_3^2}{2} - 2{A_3} + {E_3}\Big) (\xi^2-\sigma_2^2) \bigg] \,\, , \quad &&\xi \in (\sigma_2,\infty)
\end{aligned}
\right.
\end{equation}

Imposing now the boundary conditions \eqref{bc-euler} into \eqref{eq7} yields the explicit formula of discontinuous solution $\theta$ 
\begin{equation}
\label{disc-sol3}
\frac{\theta^2}{2} =
\left\{
\begin{aligned}
&{k_1} \Bigg[\phi(\xi) - \phi(\sigma_1)  -  \frac{\xi^2-\sigma_1^2}{\sigma_1^2}\phi(\sigma_1)\Bigg] \,\, , \quad &&\xi \in (0,\sigma_1) \\
&{k_2} \Bigg[\phi(\xi) - \phi(\sigma_1)  -  \frac{\phi(\sigma_2)-\phi(\sigma_1)}{\sigma_2^2 - \sigma_1^2} (\xi^2-\sigma_1^2) \Bigg] \,\, , \quad &&\xi \in (\sigma_1,\sigma_2) \\
&{k_3} \bigg[ \phi(\xi) - \phi(\sigma_2) \bigg] \,\, , \quad &&\xi \in (\sigma_2,\infty)
\end{aligned}
\right.
\end{equation}
Note that functions $\theta_1$ and $\theta_3$ have the same form as $\theta_-$ and $\theta_+$ in \eqref{th-disc2} for $\sigma=\sigma_1$ and $\sigma=\sigma_2$ respectively. Likewise $k_-$ and $k_+$, we conclude that $k_1>0$ and $k_3>0$ using the same reasoning. To define the sign of $k_2$ in the remaining case, we consider the following lemma.
\begin{lemma}
\label{lemma1}
Let $\xi \in (\sigma_1,\sigma_2)$. Then,
\begin{equation}
J(\xi) = \phi(\xi) - \phi(\sigma_1)  -  \frac{\phi(\sigma_2)-\phi(\sigma_1)}{\sigma_2^2 - \sigma_1^2} (\xi^2-\sigma_1^2) >0.
\end{equation}
\end{lemma}

\begin{proof}
For $\xi \in (\sigma_1,\sigma_2)$, set
\begin{align*}
J(\xi) &= \phi(\xi) - \phi(\sigma_1)  -  \frac{\phi(\sigma_2)-\phi(\sigma_1)}{\sigma_2^2 - \sigma_1^2} (\xi^2-\sigma_1^2) =(\xi^2-\sigma_1^2) \Big[F(\xi)  -  F(\sigma_2) \Big]
\end{align*} 
where $F(\xi)=\frac{\phi(\xi) - \phi(\sigma_1)}{\xi^2-\sigma_1^2}$. From the mean value theorem, observe
\[F'(\xi) = \frac{1}{\xi^2-\sigma_1^2}\bigg[\phi'(\xi) - \frac{2 \xi}{\xi+\sigma_1} \frac{\phi(\xi) - \phi(\sigma_1)}{\xi-\sigma_1} \bigg] = \frac{1}{\xi^2-\sigma_1^2}\bigg[\phi'(\xi) - \frac{2 \xi}{\xi+\sigma_1} \phi'(c) \bigg],\]
where $c \in (\sigma_1,\xi)$. Recall Lemma \eqref{phi}, the concavity of $\phi$ implies that $F$ is decreasing 
\begin{align*}
F'(\xi) \leq \frac{\phi'(c)}{\xi^2-\sigma_1^2}\bigg[1 - \frac{2 \xi}{\xi+\sigma_1}\bigg] <0.
\end{align*}
Hence, we get
\[ \Big[F(\xi)  -  F(\sigma_2) \Big]>0, \quad  \xi \in (\sigma_1,\sigma_2),\]
which completes the proof.
\end{proof}
Lemma \eqref{lemma1} implies that $k_2>0$. Considering the signs of constants $k_1,k_2$ and $k_3$, we proceed to the main theorem about the non-existence of solutions with discontinuities at two points.

\begin{theorem}
\label{theorem2}
Let $(\theta, V, P)$ be a weak solution of \eqref{theta-eq} - \eqref{p-eq} of class $\theta\in W^{1,1}((0,\infty))$, $V \in BV((0,\infty)) \cap L^{\infty}((0,\infty))$ and $P \in BV((0,\infty)) \cap L^{\infty}((0,\infty))$ which satisfies the boundary conditions 
\begin{equation*}
V(0+)=V_0, \quad P(0+) = E_0, \quad \theta(0) = 0, \quad \theta'(\infty)=0.
\end{equation*}
There does not exist a solution $(\theta, V, P)$ with discontinuities at two points that satisfies jump conditions \eqref{jump}. 
\end{theorem}

\begin{proof}
Suppose there exists a weak solution $\theta$ given by \eqref{disc-sol3} that satisfies jump conditions \eqref{jump} on both discontinuity points $\sigma_1$ and $\sigma_2$. Consider first the condition \eqref{jump-th} applied to \eqref{disc-sol3} at $\xi=\sigma_1$. That is,
\[\theta_1(\sigma_1)\theta'_1(\sigma_1) =  \theta_2(\sigma_1) \theta'_2(\sigma_1).\]
This yields a constrain on constants $k_1,k_2$
\begin{align}
\label{eq3}
\frac{k_1}{k_2} &= -\sqrt{1+\sigma_1^2} \bigg(1 - 2\frac{\phi(\sigma_2)-\phi(\sigma_1)}{\sigma_2^2 - \sigma_1^2} \frac{\sigma_1}{\phi'(\sigma_1)} \bigg) \phi'(\sigma_1).
\end{align}
Note that $\phi'(\sigma_1) - 2\frac{\phi(\sigma_1)}{\sigma_1}=-\frac{1}{\sqrt{1+\sigma_1^2}}$.

Since $k_1>0$ and $k_2>0$, the relation \eqref{eq3} holds iff 
\begin{equation}
\label{eq8}
1 - 2\frac{\phi(\sigma_2)-\phi(\sigma_1)}{\sigma_2^2 - \sigma_1^2} \frac{\sigma_1}{\phi'(\sigma_1)} <0.
\end{equation}
Using the mean value theorem, one may rewrite \eqref{eq8} as
\begin{align}
\label{eq17}
0> 1 - 2\frac{\phi(\sigma_2)-\phi(\sigma_1)}{\sigma_2^2 - \sigma_1^2} \frac{\sigma_1}{\phi'(\sigma_1)} = 1 - \frac{2 \sigma_1}{\sigma_2 + \sigma_1} \frac{\phi(\sigma_2)-\phi(\sigma_1)}{\sigma_2 - \sigma_1} \frac{1}{\phi'(\sigma_1)}= 1 - \frac{2 \sigma_1}{\sigma_2 + \sigma_1} \frac{\phi'(c)}{\phi'(\sigma_1)}
\end{align}
for $c \in (\sigma_1,\sigma_2)$. Recall $\phi(\xi)$ is concave, it implies that $\phi'(\sigma_1)>\phi'(c)$ and
\[1 - \frac{2 \sigma_1}{\sigma_2 + \sigma_1} \frac{\phi'(c)}{\phi'(\sigma_1)} > \frac{\sigma_2-\sigma_1}{\sigma_2 + \sigma_1}>0,\]
which contradicts \eqref{eq17}. Hence, the jump condition \eqref{jump-th} is not satisfied at $\xi=\sigma_1$.

Next, we consider the condition \eqref{jump-th} applied to \eqref{disc-sol3} at $\xi=\sigma_2$. That is,
\[\theta_3(\sigma_2)\theta'_3(\sigma_2) =  \theta_2(\sigma_2) \theta'_2(\sigma_2).\]
This provides a restriction on constants $k_2,k_3$
\begin{align}
\label{eq4}
\frac{k_3}{k_2} &= 1 - 2\frac{\phi(\sigma_2)-\phi(\sigma_1)}{\sigma_2^2 - \sigma_1^2} \frac{\sigma_2}{\phi'(\sigma_2)}.
\end{align}
Since $k_2>0$ and $k_3>0$, the relation \eqref{eq4} holds iff 
\begin{equation*}
1 - 2\frac{\phi(\sigma_2)-\phi(\sigma_1)}{\sigma_2^2 - \sigma_1^2} \frac{\sigma_2}{\phi'(\sigma_2)}>0.
\end{equation*}
From the mean value theorem, we obtain
\begin{align*}
0< 1 - 2\frac{\phi(\sigma_2)-\phi(\sigma_1)}{\sigma_2^2 - \sigma_1^2} \frac{\sigma_2}{\phi'(\sigma_2)} = 1 - \frac{2 \sigma_2}{\sigma_2 + \sigma_1} \frac{\phi'(c)}{\phi'(\sigma_2)},
\end{align*}
for some $c \in (\sigma_1,\sigma_2)$. The concavity of $\phi$ leads again to a contradiction. Therefore, there does not exist a discontinuous solution $(\theta, V, P)$ with discontinuities at two points that satisfies the jump conditions \eqref{jump}.
\end{proof}
Recall theorem \ref{theorem1} and theorem \ref{theorem2}, we observe that jump condition \eqref{jump-th} is not satisfied neither near the boundary nor across the vortex line, despite the existence of an intermediate region. Therefore, a similar outcome will also be obtained if a finite number of intermediate layers are considered. This leads to the following corollary. 
\begin{corollary}
Let $(\theta, V, P)$ be a weak solution of \eqref{theta-eq} - \eqref{p-eq} of class $\theta\in W^{1,1}((0,\infty))$, $V \in BV((0,\infty)) \cap L^{\infty}((0,\infty))$ and $P \in BV((0,\infty)) \cap L^{\infty}((0,\infty))$ which satisfies the boundary conditions 
\begin{equation*}
V(0+)=V_0, \quad P(0+) = E_0, \quad \theta(0) = 0, \quad \theta'(\infty)=0.
\end{equation*}
There does not exist a solution $(\theta, V, P)$ with discontinuities at a finite number of points that fulfils the jump conditions \eqref{jump}. 
\end{corollary}

\subsection{Interaction of vortex with boundary}
Motivated by the study of Euler equations presented in the previous sections, we are interested in extending it to a class of flows where there is mass input or loss through the vortex line. In other words, we assume that the vortex line can be either a source or a sink of the fluid motion. Since that the vortex line resembles the tornado core, this assumption is not unnatural for a tornado-like flow. In terms of the self-similar functions \eqref{ansatz}, the assumption leads to
\begin{equation*}
U(\xi) \to U_\infty \quad \textrm{as} \quad \xi \to \infty,
\end{equation*}
or in terms of $\theta$
\begin{equation}
\label{bc-new}
\theta'(\xi) \to -U_\infty \quad \textrm{as} \quad \xi \to \infty,
\end{equation}
where $U_\infty$ is a non-zero constant. 

\subsubsection{Continuous solutions}
Consider first the case where solutions are continuous. Same calculations as those described in Section \ref{sec:euler} lead again to solutions in form \eqref{euler-sol}. The boundary condition \eqref{bc-new} on the vortex line applied to \eqref{eqn11} yields
\begin{equation}
\label{bc-new1}
\frac{1}{2} \,  U_\infty^2 = k_0 - \frac{V_0^2}{2}- 2A_0 + E_0.
\end{equation}
To derive this, note that \eqref{eqn11} implies
\begin{align*}
\theta' = \frac{1}{\theta} \bigg[k_0 \phi'(\xi) + 2\bigg(k_0 - \frac{V_0^2}{2}- 2A_0 + E_0\bigg) \xi\bigg]
\xrightarrow []{\xi\to\infty} {\pm\sqrt{2\bigg(k_0 - \frac{V_0^2}{2}- 2A_0 + E_0\bigg)}}.
\end{align*}
Thus, as $\xi\to \infty$
\[\theta'(\infty) =  -U_\infty =  {\pm\sqrt{2\bigg(k_0 - \frac{V_0^2}{2}- 2A_0 + E_0\bigg)}}.\]
In addition, we impose a no-penetration boundary condition at $\xi=0$, i.e. $\theta(0)=0$. Applying this to \eqref{eqn11} provides
\begin{align}
\label{bc-new2}
A_0 = E_0.
\end{align}
As a result the relation \eqref{bc-new1} reduces to
\begin{equation}
\label{bc-new3}
\frac{1}{2} U_\infty^2 = k_0 - \frac{V_0^2}{2}- E_0.
\end{equation}
Substituting \eqref{bc-new2} and \eqref{bc-new3} into \eqref{euler-sol}, we obtain an explicit family solutions of $\eqref{theta-eq} - \eqref{p-eq}$ depending on parameters $U_\infty, V_0$ and $E_0$
\begin{subequations}
\begin{align}
&\frac{\theta^2}{2} = k_0 \phi(\xi) + \frac{1}{2} U_\infty^2  \xi^2, \qquad
U = - \frac{1}{\theta} \bigg[\frac{k_0} {\sqrt{1+\xi^2}} \ \big(1-2\phi(\xi)\big) +  U_\infty^2 \xi \bigg], \\
&V = V_{0}, \qquad
W = \frac{k_0}{\theta} \frac{\xi} {\sqrt{1+\xi^2}}, \qquad
P = E_0 - k_0 \frac{\xi} {\sqrt{1+\xi^2}},
\end{align} 
\end{subequations} 
where \[k_0= \frac{1}{2} U_\infty^2 + \frac{V_0^2}{2}+E_0.\]
We next investigate the existence of discontinuous solutions in this framework.

\subsubsection{Nonexistence of Discontinuous Solutions}
Let $(\theta,V,P)$ be a weak solution of \eqref{theta-eq} - \eqref{u-eq} with a discontinuity at a fixed point $\xi = \sigma$, $\sigma \in (0,\infty)$, which satisfies the boundary condition \eqref{bc-new}. This suggests seeking solutions in the form
\begin{subequations}
\begin{align*}
\theta (\xi) =
\left\{
\begin{aligned}
& \theta_{-} (\xi), \\
&\theta_{+} (\xi)
\end{aligned}
\right. 
\quad \text{and} \quad
V (\xi) =
\left\{
\begin{aligned}
& V_- (\xi) \,\, , \quad &&\xi \in (0,\sigma) \\
& V_+ (\xi)\,\, , \quad &&\xi \in (\sigma,\infty)
\end{aligned}
\right.
\end{align*}
\end{subequations}
Repeating calculations described in Section \ref{sec:euler}, we obtain again solutions in form \eqref{eqn11} on $(0,\sigma)$ and $(\sigma,\infty)$ respectively. The system \eqref{theta-eq} - \eqref{p-eq} is solved on each domain independently. Note that $V(\xi)$ is again a constant function with $V_-$ and $V_+$ constants. Recall that discontinuous solutions occur when ${\theta}(\sigma)=0$, the relation for $\theta$ becomes
\begin{equation}
\label{th-disc3}
\frac{\theta_\pm^2}{2} = {k_\pm} \bigg(\phi(\xi) - \phi(\sigma)\bigg) + \Big({k_\pm}-\frac{V_\pm^2}{2} - 2{A_\pm} + {E_\pm}\Big) \big(\xi^2 - \sigma^2\big).
\end{equation} 
where $k_\pm, A_\pm$ and $E_\pm$ are integration constants. 

We impose a no-penetration boundary condition at the boundary and the condition \eqref{bc-new} at the vortex line, namely
\[\theta_-(0) = 0, \quad \textrm{and} \quad \theta'_+(\xi) \to -U_\infty \quad \textrm{as} \quad \xi \to \infty,\]
which imply
\begin{align}
\label{bc1}
{k_-}-\frac{V_-^2}{2} - 2{A_-} + {E_-} &= - {k_-} \frac{1}{\sigma^2} \\
\label{bc2}
k_+ - \frac{V_+^2}{2}- 2A_+ + E_+ &= \frac{1}{2}  U_\infty^2
\end{align}
To derive \eqref{bc2}, note that \eqref{th-disc3} provides
\[
\theta'_+ \xrightarrow []{\xi\to\infty} {\pm\sqrt{2\bigg(k_+ - \frac{V_+^2}{2}- 2A_+ + E_+\bigg)}}, \]
and thus,
\[\theta'_+(\infty) = - U_\infty =  \pm\sqrt{2\bigg(k_+ - \frac{V_+^2}{2}- 2A_+ + E_+\bigg)}. \]
Substituting \eqref{bc1} and \eqref{bc2} into \eqref{th-disc3}, we obtain the explicit formula 
\begin{align}
\label{disc-sol2}
\frac{\theta^2}{2} =
\left\{
\begin{aligned}
&{k_-} \Bigg[\phi(\xi) - \phi(\sigma)  -  \frac{\xi^2-\sigma^2}{\sigma^2}\phi(\sigma)\Bigg] \,\, , \quad &&\xi \in (0,\sigma) \\
&{k_+} \bigg[ \phi(\xi) - \phi(\sigma) \bigg] + \frac{1}{2}  U_\infty^2 (\xi^2-\sigma^2) \,\, , \quad &&\xi \in (\sigma,\infty)
\end{aligned}
\right.
\end{align}
where $k_+, k_-$  are constants. From \eqref{euler-sol} and \eqref{disc-sol2}, one may derive the explicit formulas of discontinuous solutions $U, W$ and $P$. 

Examining now whether discontinuous solutions in the form \eqref{disc-sol2} are feasible leads to the following theorem.

\begin{theorem}
Let $(\theta, V, P)$ be a weak solution of \eqref{theta-eq} - \eqref{p-eq} of class $\theta\in W^{1,1}((0,\infty))$, $V \in BV((0,\infty)) \cap L^{\infty}((0,\infty))$ and $P \in BV((0,\infty)) \cap L^{\infty}((0,\infty))$ which satisfies the boundary conditions 
\begin{equation*}
V(0+)=V_0, \quad P(0+) = E_0, \quad \theta(0) = 0, \quad \theta'(\infty)=U_\infty.
\end{equation*}
There does not exist a solution $(\theta, V, P)$ with a discontinuity at a single point that fulfils jump conditions \eqref{jump}. 
\end{theorem}

\begin{proof}
Suppose $\theta$ is a weak solution given by \eqref{disc-sol2} that satisfies jump conditions \eqref{jump}. 
Then \eqref{jump-th} applied to \eqref{disc-sol2} provides a constraint between the constants $k_+,k_-$ and the parameter $U_\infty^2$,
\begin{align}
\label{eq2}
-k_+ \phi^\prime (\sigma)  = k_- \Big ( 2 \frac{\phi(\sigma)}{\sigma} - \phi^\prime (\sigma) \Big )  +  U_\infty^2 \sigma  \, .
\end{align}
Using \eqref{defphi} we compute $2 \frac{\phi(\sigma)}{\sigma} - \phi^\prime (\sigma) = \frac{1}{\sqrt{1 + \sigma^2}} > 0$.

Since $\theta^2 (\xi) > 0$ for $\xi \ne \sigma$, relation \eqref{disc-sol2} imposes restrictions on the signs of $k_+,k_-$. 
First, observe that $\frac{\theta_-^2}{2}(\xi) = k_- \,J(\xi)$ where $J(\xi)$ is given by 
\eqref{eq-J}. Since $J(\xi)>0$ for $\xi\in(0,\sigma)$, it implies that $k_->0$.

If it were $k_+ > 0$ this would contradict \eqref{eq2}.  Consider now the possibility that the constants are  $k_- > 0$, $k_+ < 0$. 
Then \eqref{disc-sol2} dictates that
$$
k_+   \frac{ \phi(\xi) - \phi(\sigma)}{\xi - \sigma}  + \frac{1}{2}  U_\infty^2 (\xi + \sigma) > 0 \quad \mbox{for $\xi > \sigma$}
$$
The latter implies, as $\xi \to \sigma_+$,
$$
k_+ \phi^\prime (\sigma) + U_\infty^2 \sigma \ge 0
$$
and contradicts via \eqref{eq2} that $k_- > 0$. Hence, there does not exist a discontinuous solution $(\theta, V, P)$ that fulfils the jump conditions \eqref{jump}.
\end{proof}


\section{Stationary self-similar axisymmetric Navier-Stokes equations}
\label{sec:NS}
In this section we study self-similar solutions for the  stationary  axisymmetric Navier-Stokes equations \eqref{axi-sys}. 
The {\it ansatz}  \eqref{ansatz} leads to the system \eqref{ssform} in the variables $(\theta, V, P)$, where the velocities $(U, W)$ are determined 
from the stream function $\theta$ via \eqref{stream}. 
We  first provide a convenient reformulation of \eqref{ssform} as a coupled integrodifferential system and 
then study the limiting behavior as the viscosity $\nu \to 0$, the form of the boundary layer,  and identify conditions on admissible solutions of the Euler equations.

\subsection{Formulation via an integro-differential system}
The system \eqref{ssform}  for $\nu > 0$ gives after an integration 
\begin{subequations}
\begin{align}
\label{n1}
\frac{\theta^2}{2} + (1 + \xi^2) P &= \nu \Big [ \xi \theta - (1 + \xi^2) \theta^\prime \Big ] - \int_{0}^{\xi} s V^2 \,ds + A_0
\\
\label{n3}
\theta^2  - \xi \Big(\frac{\theta^2}{2}\Big)'  + P  &= \nu \Big[\xi\theta - \xi^2 \theta' - \xi (1+\xi^2) \theta'' \Big]  + {E_0}  ,
\end{align}
\end{subequations}
where $A_0$, $E_0$ are  integration constants. 

We impose a no-slip boundary condition $\vec{u} = 0$ at the boundary $z=0$,  which implies 
$u = v = w = 0$ at $z = 0$. Expressed in terms of self-similar functions it yields
\begin{equation}
\label{bd1}
U(0) = V(0) = W(0) = 0,   \quad \textrm{and thus} \quad \theta(0) = \theta'(0) = 0.
\end{equation}
We also require that no mass is added or lost through the vortex line, that is $u (2 \pi r) \to 0$ as $r \to 0$.  In the self-similar setting
the condition at the vortex core becomes
\begin{equation}
\label{bd2}
U = - {\theta}' \to 0 \quad \mbox{ as \; $\xi \to \infty$}
\end{equation}
Since \eqref{n2} is second order, an additional  condition is required for $V(\xi)$. We request the vortex line to have constant swirl:
\begin{equation}
\label{bd3}
V \to V_\infty, \,\, \textrm{as} \quad \xi \to \infty.
\end{equation}

Then  \eqref{bd1} together with \eqref{n1} and \eqref{n3} imply that $E_0 = A_0 =  P(0)$.
Eliminating $P$ from \eqref{n1} and \eqref{n3} gives 
\begin{equation}
\begin{aligned}
\xi (1+\xi^2) \Big(\frac{\theta^2}{2}\Big)' - (1+2\xi^2) \frac{\theta^2}{2}& + \nu \bigg[\xi^3 \theta + (1-\xi^4) \theta' - \xi(1+\xi^2)^2 \theta'' \bigg] \nonumber\\
&= - \int_{0}^{\xi} s V^2 \,ds - {E_0} \xi^2
\end{aligned}
\end{equation}
Upon dividing by $\xi^2 (1+\xi^2)^{\frac{3}{2}}$ and using the calculus identities
\begin{align}
\label{calc1}
&\frac{d}{d\xi} \frac{1}{\xi (1+\xi^2)^{\frac{1}{2}} } = - \frac{1+2\xi^2}{\xi^2 (1+\xi^2)^{\frac{3}{2}}} \qquad \qquad
&&\frac{d}{d\xi} \frac{ (1+\xi^2)^{\frac{1}{2}} }{\xi} =  -  \frac{1}{ \xi^2 (1+\xi^2)^{\frac{1}{2}}} 
\\
\label{calc2}
&\frac{d}{d\xi} \frac{1}{ (1+\xi^2)^{\frac{1}{2}} } = - \frac{\xi}{ (1+\xi^2)^{\frac{3}{2}}} \qquad  \qquad  
&&\frac{d}{d\xi} \Big ( \frac{1}{\xi \sqrt{1+ \xi^2} \big (\xi + \sqrt{1+\xi^2} \big )^2 } \Big )  = - \frac{1}{ \xi^2 (1 + \xi^2)^{\frac{3}{2}}}
\end{align}
we obtain after some rearrangements the equation
\begin{equation}\label{diffeq10}
\frac{d}{d\xi} \left ( \frac{1}{\xi\sqrt{1+\xi^2}} \frac{\theta^2}{2} - \nu \Big[\frac{1}{\sqrt{1+\xi^2}} \theta + \frac {\sqrt{1+\xi^2}} {\xi} \theta' \Big] \right ) 
=  - \frac{1}{\xi^2 (1+\xi^2)^{\frac{3}{2}}}  \int_{0}^{\xi}  s V^2 \,ds -  \frac{d}{d\xi} \Big ( \frac{E_0 \xi }{ (1+\xi^2)^{\frac{1}{2}}} \Big )  \, .
\end{equation}

Using the boundary condition \eqref{bd2}, $\theta^\prime (\infty) = 0$, which also implies that $\lim_{\xi \to \infty} \frac{\theta(\xi)}{\xi} = 0$, we 
integrate \eqref{diffeq10}
over the interval $(\xi, \infty)$ to obtain
$$
\frac{\theta^2}{2} - \nu \Big ( 1+\xi^2) \theta' + \xi \theta \Big ) =  \xi(1+\xi^2)^{\frac{1}{2}} \int_\xi^\infty \frac{1}{\zeta^2(1+\zeta^2)^\frac{3}{2}} \bigg( \int_0^\zeta s V^2(s) ds\bigg) d\zeta + \,\, {E_0} \Big ( \xi (1+\xi^2)^{\frac{1}{2}} - \xi^2 \Big ) \, .
$$
The latter is an integrodifferential equation depending on $\mathcal{G} (V ; \xi)$  a functional on $V$,
\begin{equation}
\label{defG}
\begin{aligned}
 \mathcal{G}\left(V ; \xi\right)   &:= \xi \sqrt{1+\xi^2} \int_\xi^\infty \frac{1}{\zeta^2(1+\zeta^2)^\frac{3}{2}} \bigg( \int_0^\zeta s V^2(s) ds\bigg) d\zeta 
 \\
 &= \xi \sqrt{1+\xi^2} \int_\xi^\infty  \frac{1}{s \sqrt{1+s^2} \big (s + \sqrt{1+s^2} \big )^2 } s V^2 ds + \frac{1}{ (\xi + \sqrt{1 + \xi^2} )^2} \int_0^\xi s V^2 ds \, ,
\end{aligned}
\end{equation}
where the last equality follows  via integration by parts using \eqref{calc2}.

In summary,  the system \eqref{ssform} with boundary conditions \eqref{bd1}-\eqref{bd3} is transformed to a coupled 
integrodifferential system for $(\theta, V)$,
\begin{subequations}
\label{eq-ns-v1}
\begin{align}
\label{sys1}
\frac{\theta^2}{2} - \nu \bigg[(1+\xi^2) \theta' + \xi \theta \bigg]  &= \mathcal{G}\left( V ; \xi\right)  + E_0 \phi (\xi) \\
\label{sys2}
\nu (1+\xi^2) V'' + \Big(3\nu\xi - \theta\Big) V' &= 0 \, ,
\end{align}
\end{subequations}
where $\mathcal{G} (V ; \xi)$ is the functional in \eqref{defG}
and $\phi(\xi)$ the function in \eqref{defphi} that  played a prominent role
in the Euler equations.
The system is supplemented with the boundary conditions
\begin{subequations}
\label{bc_xi}
\begin{align}
\theta (0) =  V (0) = 0 , \,\, &\textrm{ at } \quad \xi = 0 \\
V \to V_\infty , \,\, &\textrm{ as} \quad \xi \to \infty
\end{align}
\end{subequations}
and satisfies $\theta^\prime(0) = \theta^\prime(\infty) = 0$.
The parameter $E_0= P(0)$ reflects on the pressure $p = \frac{1}{r^2} P(\xi)$ and $P$ 
may be computed by \eqref{n3}.

\subsection{Alternative equivalent formulations}
\label{section:equiv-form}
Next, we derive two more equivalent formulations of the system \eqref{eq-ns-v1} with simpler structure. This will benefit the analysis and provides a different perspective on the problem. 

First we set
\begin{equation}\label{firsttrans}
\hat{\theta} (\xi) = \sqrt{1+\xi^2} \, \theta(\xi) \quad \textrm{and} \quad	\hat{V} (\xi) =V(\xi), \quad \xi\in[0,\infty),
\end{equation}
and obtain
\begin{subequations}
\label{eq-ns-v2}
\begin{align}
\frac{\hat{\theta}^2}{2} - \nu (1+\xi^2)^\frac{3}{2} \, \hat{\theta}'  &= (1+\xi^2) \Big (  \mathcal{G} \big (\hat{V}(\xi) ; \xi \big) + E_0 \phi (\xi) \Big ) , \\
\nu\Big ( (1+\xi^2)^\frac{3}{2} \hat{V}'\Big )^\prime &=  \hat{\theta} \hat{V}',
\end{align}
\end{subequations}
with boundary conditions \eqref{bc_xi}.

An alternative form of \eqref{eq-ns-v1} is derived by introducing a change of variable $\xi \to x$ such that $\frac{d}{dx} = (1 + \xi^2)^{\frac{3}{2}} \frac{d}{d\xi}$.
This is achieved by defining $x$ by
\begin{equation}
\label{def_x}
x = \frac{\xi}{\sqrt{1+ \xi^2}} \quad\textrm{with inverse thansformation} \quad \xi = \frac{x}{\sqrt{1- x^2}} \, .
\end{equation}
The change of variables  $\xi \to x$ is a surjective map $T: [0, \infty) \to [0, 1)$
and has an interpretation as a change from cylindrical $(r,\vartheta,z)$ to spherical $(\rho,\vartheta,\varphi)$ coordinates, 
see Figure \ref{coord-sys}. Since $\xi = \frac{z}{r}$, it follows $\xi = \frac{\cos \varphi}{\sin \varphi} = \cot \varphi$
and setting $x = \cos\varphi$ yields
\begin{equation}
\xi = \frac{\cos \varphi}{\sin \varphi} = \frac{x}{\sqrt{1- x^2}}.
\end{equation}
The  functions $\hat{\theta}(\xi)$ and $\hat{V}(\xi)$ are expressed in terms of the variable $x$ as follows
\begin{equation}\label{defvar}
\Theta (x)  =  \hat{\theta}(\xi)  \quad \textrm{and} \quad V (x) = \hat{V}(\xi),
\end{equation}
where $\xi, \, x$ are related via \eqref{def_x}. 
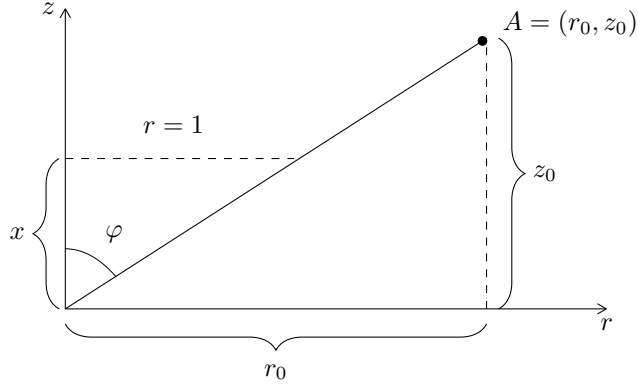
\begin{figure}
\centering
\begin{tikzpicture}[xscale=8, yscale=20]
\draw [<-] (0,0.2) -- (0,0) ; 
\draw [->]  (0,0) -- (0.9,0.0); 
\draw [dashed]  (0.7,0) -- (0.7,0.18); 
\draw [decorate, decoration = {brace, mirror,amplitude=10pt}] (0.72,0) --  (0.72,0.18);
\node [right] at (0.76,0.09) {$z_0$}; 
\node [left] at (0,0.2) {$z$}; 
\node [below] at (0.9,0) {$r$};
\draw [- {Circle}](0,0) -- (0.7,0.18); 
\node [right] at (0.71,0.19) {$A=(r_0,z_0)$}; 
\draw [very thin]  (0.0,0.04) arc (90:65:0.2);
\node [right] at (0.05,0.05) {$\varphi$}; 
\draw [decorate, decoration = {brace, mirror,amplitude=10pt}] (0,-0.01) --  (0.7,-0.01);
\node [below] at (0.35,-0.03) {$r_0$};
\draw [decorate, decoration = {brace,amplitude=10pt}] (-0.01,0) --  (-0.01,0.1);
\node [left] at (-0.05,0.05) {$x$}; 
\draw [dashed]  (0,0.1) -- (0.38,0.1); \node [above] at (0.18,0.11) {$r=1$};
\end{tikzpicture}
\caption{Coordinate System}
\label{coord-sys}
\end{figure}

Using the transformation \eqref{def_x}, \eqref{defvar},   the system \eqref{eq-ns-v1} reduces to the equivalent form
\begin{subequations}
\label{eq-ns-v3}
\begin{align}
\label{eq-ns-v3-1}
\nu \, \frac{d \Theta}{dx}  &=  \frac{\Theta^2 (x)}{2} - \mathcal{F}\left(V ; x\right), \\
\label{eq-ns-v3-2}
\nu \frac{d^2 {V}}{d^2x} &= \Theta(x) \, \frac{d {V}}{dx} ,
\\
\label{bd4_new}
&\qquad \Theta(0) = {V}(0) = 0 , \\
\label{bd3_new}
&\qquad {V} \to V_\infty, \,\,\textrm{ as} \,\ x \to 1,
\end{align}
\end{subequations}
where $\mathcal{F}$ is 
\begin{align}
\mathcal{F}\left(V ; x\right) &:= 
(1+\xi^2) \Big (  \mathcal{G} \big (\hat{V}(\xi) ; \xi \big) + E_0 \phi (\xi) \Big ) \Bigg |_{\xi = \frac{x}{\sqrt{1-x^2}}}
\nonumber 
\\
&=
\frac{x}{(1-x^2)^2}\int_x^1 \frac{1-t^2}{t^2} \bigg(\displaystyle \int_0^\frac{t}{\sqrt{1-t^2}} s \hat{V}^2(s) ds  \bigg) \, dt 
+ \, {E_0}  \,\, \frac{x - x^2}{(1-x^2)^2}
\label{defF}
\\
&= \frac{x}{(1-x^2)^2}\Bigg[\int_x^1 \frac{1-t^2}{t^2} \bigg(\displaystyle \int_0^t \frac{\sigma}{(1-\sigma^2)^2} V^2(\sigma) d\sigma \bigg) \, dt + \, {E_0} \,\, (1 - x)\Bigg]
\nonumber 
\\
&= 
\frac{1}{(1-x^2)^2}\Bigg[    x  \int_x^1  \frac{1}{ (t+1)^2} V^2 (t) dt  +  (1-x)^2 \int_0^x \frac{t}{(1 - t^2)^2} V^2 (t) dt   + \, {E_0} \, x (1 - x)\Bigg]
\nonumber
\end{align}
The first expression  in \eqref{defF} is the definition of $\mathcal{F}\left(V ; x\right)$ and the remaining equalities are obtained 
using \eqref{defphi}, \eqref{defG}, the change of variables $\xi = \frac{x}{\sqrt{1-x^2}}$ and the 
formula $V(x) = \hat V \Big ( \frac{x}{ \sqrt{1-x^2}} \Big )$ connecting $V(x)$ defined on $[0,1)$ with $\hat V(\xi)$ defined on $[0,\infty)$.

Variants of \eqref{eq-ns-v3} appear in \cite{Gold60,Serrin,SH1999,FFA00,BDSS2014} originating from different authors 
who initiate their studies by considering either spherical or cylindrical coordinates.
Serrin \cite{Serrin} provided a detailed analysis of existence of solutions for \eqref{eq-ns-v3} 
and a (indicative)  bifurcation diagram for solutions.  We refer to Section \ref{sec:numerics} for numerical computations leading
to a computational bifurcation diagram, see Figure \ref{bfd1}.
As $\Theta$ satisfies a Ricatti-type equation, there are regions of $\nu$ for which $\Theta (x)$ blows up for 
$x \in [0,1)$ leading to non-existence of solutions for such $\nu$. This is explained in the following section.

\subsection{Properties and a-priori estimates}
\label{sec:existence}


In the sequel we are interested in the limit $\nu \to 0$ and the convergence of  \eqref{eq-ns-v3} to solutions
of the Euler equations.
To fix ideas we restrict to the case $V_\infty  > 0$.  Then  \eqref{eq-ns-v3-2}, \eqref{bd3_new}  imply that $V(x)$ is strictly increasing
and yields the representation formula
\begin{equation}
\label{V-sol}
V(x) =    \displaystyle  {V_\infty \, \frac{  \int_0^x e^{ \frac{1}{\nu}  \int_0^t \Theta(s) ds} dt} { \int_0^1 e^{ \frac{1}{\nu}  \int_0^t \Theta(s) ds} dt}  } 
\end{equation}
leading to the following lemma.
\begin{lemma}
\label{lemma-V}
When  $V_\infty>0$ the function $V(x)$ is strictly increasing and satisfies $0 < V(x) < V_\infty$ independently of $\nu > 0$.
\end{lemma}

We turn next to the equation \eqref{eq-ns-v3-1},
\begin{equation}
\label{eqtheta}
\nu \, \frac{d \Theta}{dx}  =  \frac{\Theta^2 (x)}{2} -  \frac{x}{(1-x^2)^2}\Big(H(x) + \, {E_0} \,\, (1 - x)\Big)
\end{equation}
and proceed to establish properties for $\Theta$.
The functional $H(x)$ in \eqref{defF} is given by 
\begin{equation}
\label{H}
\begin{aligned}
H(x) = \int_x^1 \frac{1-t^2}{t^2} \bigg(\displaystyle \int_0^\frac{t}{\sqrt{1-t^2}} s \hat{V}^2(s) ds  \bigg) \, dt \, ,
\end{aligned}
\end{equation}
here expressed in terms of the velocity $\hat V (\xi)$ defined on $(0,\infty)$.
The derivatives of $H(x)$ are
\begin{align}
\frac{dH}{dx} (x) &=-\frac{1-x^2}{x^2} \int_0^\frac{x}{\sqrt{1-x^2}} s {\hat V}^2(s) ds < 0, \quad x \in [0,1),   
\label{deriv1}
\\
\frac{d^2 H}{dx^2} (x) &= -\frac{2}{x^3}   \int_0^\frac{x}{\sqrt{1-x^2}} s  {\hat V}(s) \frac{d {\hat V}}{d s}(s) ds < 0,  \quad x \in [0,1),
\label{deriv2}
\end{align}
and $H(x)$ satisfies 
\begin{equation}\label{defh0}
 0 <  H(0) =
 \int_0^1 \frac{1-t^2}{t^2} \bigg(\displaystyle \int_0^\frac{t}{\sqrt{1-t^2}} s \hat{V}^2(s) ds  \bigg) \, dt < \frac{V_\infty^2}{2}  \, ,
\end{equation}
$H(1) = 0$,  $\lim_{x\to0} \frac{dH}{dx}= 0$ and  $\lim_{x\to0} \frac{d^2H}{dx^2} = 0$.
Hence, $H(x)$ is decreasing and concave. A sketch of $H(x)$ is presented in Figure \ref{fig:H(x)}.
If we define
\begin{equation}\label{defalpha}
\frac{dH}{dx} (1) = - \lim_{\xi \to \infty} \frac{1}{\xi^2} \int_0^\xi s {\hat V}^2 (s) ds = : - \beta < 0
\end{equation}
we observe that due to the concavity of $H(x)$ we have the bounds
\begin{equation}
H(0) (1 - x) \le H(x) \le \beta (1 -x) \qquad x \in [0,1]
\end{equation}
with $H(0)$, $\beta$ defined in \eqref{defh0}, \eqref{defalpha} and satisfying $H(0) < \beta$.
\begin{figure}[htbp]
\centering
\begin{tikzpicture} 
\draw[->, name path=yaxis] (0,-.5) -- (0,3) node[left] {$H(x)$}; 
\draw[->] (-1,0) -- (6,0) node[below right] {$x$}; 
\path (0,1.25) node[point]{} node[left] {$H(0)$} (4,0) node[point]{} node[below] {$1$}; 
\draw (0,1.25) .. controls (1.5,1) and (2.75,.7) .. (4,0);  
\def\n{3.5}
\path[name path=tangent] (4,0) -- ++({\n*(2.75-4)},{\n*(.7-0)}); 
\draw[name intersections={of=yaxis and tangent}, dashed] (4,0) -- (intersection-1) node[left] {$\beta$}; ; 
\end{tikzpicture}
\caption{$H(x)$ \label{fig:H(x)}}
\end{figure}

As a consequence, the function $F = H + E_0 (1-x)$  in the Ricatti-type equation \eqref{eqtheta} obeys the bounds
\begin{equation} \label{boundH}
(H(0) + E_0) (1-x) \le F(x) = H(x) + E_0 (1-x) \le (\beta + E_0) (1-x)
\end{equation}
and resembles one of the following configurations, see Fig. \ref{fig:F(x)} :
\begin{enumerate}[label=(\alph*)]
\item $F$ in negative on $[ 0,1]$, which occurs when $\beta + E_0 < 0$
\item $F$ is first negative, it becomes zero at a point $0<x_1<1$ and then $F$ is positive. This
occurs when $H(0) + E_0 < 0 < \beta + E_0 $.
\item $F$ in positive for $ [0,1)$, occurring when $0 < H(0) + E_0$.
\end{enumerate}

\begin{figure}[htbp] 
\centering
\tikzset{every picture/.append style={scale=.75}}
\begin{subfigure}{.3\textwidth} 
\centering
\begin{tikzpicture}
\draw[->] (0,-2) -- (0,2) node[left] {$F$}; 
\draw[->] (-1,0) -- (4,0) node[below] {$x$}; 
\path (0,-1.5) node[point]{} node[right] {$F(0)=H(0)+E_0<0$} 
(3,0) node[point]{} node[below] {$1$}; 
\draw plot [smooth, tension=1.] coordinates {(0,-1.5) (1.0,-.5) (3,0)};
\end{tikzpicture}
\caption{ \label{fig:F_a}}
\end{subfigure} \hfill 
\begin{subfigure}{.3\textwidth} 
\centering
\begin{tikzpicture}
\draw[->] (0,-2) -- (0,2) node[left] {$F$}; 
\draw[->, name path=x_axis] (-1,0) -- (4,0) node[below] {$x$}; 
\path (0,-1.5) node[point]{} node[right] {$F(0)<0$} 
(3,0) node[point]{} node[below] {$1$}; 
\draw[name path=F] plot [smooth, tension=1.] coordinates {(0,-1.5) (1.2,.3) (3,0)};
\path[name intersections={of=x_axis and F}] (intersection-1) node[point]{} node[below right] {$x_1$}; 
\end{tikzpicture}
\caption{ \label{fig:F_b}}
\end{subfigure} \hfill
\begin{subfigure}{.35\textwidth}
\centering 
\begin{tikzpicture}
\draw[->] (0,-2) -- (0,2) node[left] {$F$}; 
\draw[->] (-1,0) -- (4,0) node[below] {$x$}; 
\path (0,1.5) node[point]{} node[above right] {$F(0)>0$} 
(3,0) node[point]{} node[below] {$1$}; 
\draw plot [smooth, tension=1.] coordinates {(0,1.5) (1.5,1.2) (3,0)};
\end{tikzpicture}
\caption{ \label{fig:F_c}}
\end{subfigure}
\caption{Possible configuration of function $F$ \label{fig:F(x)}}
\end{figure}
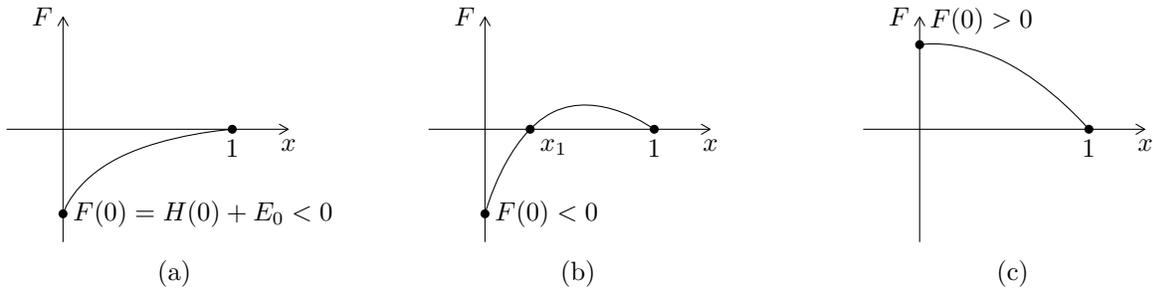

This restricts considerably the possible shapes of $\Theta$. A computation using \eqref{eqtheta} gives
$$
\nu \, \frac{d^2 \Theta}{dx^2} (0) = - (H(0) + E_0 ) = - F(0)  \, .
$$
Since $\Theta(0) = \frac{d\Theta}{dx} (0) =0$ the value of $F(0)$ determines in which half-plane $\Theta (x)$ initially starts.
Consider first the case where $F(0) = H(0) + E_0>0$. Then $F(x)>0$ for all $x\in(0,1)$, that is the setting of case ${(c)}$,  illustrated in Figure \ref{fig:F_c}. 
$\Theta(x)$ starts in the lower half-plane. If it crosses the x-axis at some point $x_1<1$, then
\begin{equation*}
0 \le  \nu \, \frac{d \Theta}{dx}  (x_1)  =  - \frac{x_1}{(1-x_1^2)^2} \, {F}(x_1) 
\end{equation*}
which contradicts that $F(x)>0$. Therefore, when $H(0) + \, {E_0} >0$,  the function $\Theta(x)$ remains negative for all $x\in (0,1)$.



Next, consider the case $H(0) + E_0<0$ which corresponds to the cases ${(a)}$ or ${(b)}$. We now have
\begin{align*}
\nu \, \frac{d^2 \Theta}{dx^2} (0)  &=  - \Big ( H(0) + \, {E_0} \Big ) > 0 \, 
\end{align*}
and $\Theta(x)$ starts at the upper half-plane. If $\Theta$ crosses the x-axis at some point $x_1<1$ then
\begin{equation}
\label{eq5}
0 \ge  \,  \nu \, \frac{d \Theta}{dx}  (x_1)  = -  \frac{x_1}{(1-x_1^2)^2}\, {F}(x_1) 
\end{equation}
It is possible to cross going downwards but it is not possible to cross a second time going upward again.
We conclude that when case $(b)$ happens then $\Theta(x)$ can possibly cross the axis but it cannot cross a second time.
By contrast when case $(a)$ happens, \eqref{eq5} contradicts $F(x)<0$, $x\in (0,1)$, {\it i.e.}  $\Theta(x)$ cannot cross the axis.
In this regime,  the differential inequality
\begin{equation*}
\nu \, \frac{d^2 \Theta}{dx^2} >  \frac{1}{2} \Theta^2, \quad x \in [0,1)
\end{equation*}
implies that the solution $\Theta(x)$ blows up in finite time, and if $\nu$ is sufficiently small, this will happen within
the interval $[0,1]$.

Summarizing there are the following possibilities:    
When $F(x) > 0$  as in Fig. \ref{fig:F_c} then $\Theta (x) < 0$; when $F(x) <  0$ as in Fig. \ref{fig:F_a} then $\Theta (x) > 0$. 
When  $F(x)$ changes sign as in Fig.  \ref{fig:F_b}, then $\Theta$ starts at the upper half plane, and either $\Theta$ stays there afterwards 
or it might cross to the lower half plane.
Accordingly, there are the following configurations:  Either $\Theta (x) < 0$ which is called Zone A;  or it starts with $\Theta (x) >0$ but crosses to 
the negative half-plane and stays there thereafter, called Zone B;  or it starts and stays $\Theta (x) >0$, called Zone C.
In Zone C the solution $\Theta$ lies in the upper half-plane and if $\nu$ is sufficiently small it will blow up before reaching $x=1$. 

\smallskip
A-priori estimates for $\Theta(x)$ are derived below.

\begin{lemma}\label{lem:bound}
(i)  If $E_0>0$, then
\begin{equation*}
{\Theta}(x) < 0, \qquad 0<x<1
\end{equation*}
(ii)  If $\beta + E_0  >0$, with $\beta$ defined in \eqref{defalpha}, then 
\begin{equation*}
{\Theta}(x) > -\sqrt{2(\beta +E_0 )} \,\,\, \frac{\sqrt{x-x^2}}{1-x^2}, \qquad 0<x<1.
\end{equation*}
\end{lemma}

\begin{proof}
Note that $\Theta (0) = \frac{d \Theta}{dx}(0) = 0$  and  $\frac{d^2 \Theta}{dx^2}(0) = - \frac{ H(0) + E_0 }{\nu} < 0 $ since  $H(0) + E_0>0$. 
If the solution crosses to the upper half plane there exists $x_1\in (0,1)$ such that $\Theta(x_1)=0$, $ \frac{d \Theta}{dx}  (x_1) > 0$ and
\begin{equation*}
 \nu \, \frac{d \Theta}{dx}  (x_1)  =  \frac{\Theta^2 (x_1)}{2} - \frac{x_1}{(1-x_1^2)^2}\Bigg[H(x_1) + \, {E_0} \,\, (1 - x_1)\Bigg]
\end{equation*}
This contradicts $H(x_1)>0$  and $E_0 > 0$. Hence, $\Theta$ remains negative for all $x\in(0,1)$.

Let now $E_0 + \beta > 0$ and set $K=2  (\,E_0 + \beta)$, $\alpha(x) = \frac{ \sqrt{x-x^2}}{1-x^2}$.  Then \eqref{eqtheta} and \eqref{boundH} imply
\begin{align}
\nu \, \frac{d \Theta}{dx} &= \frac{\Theta^2 (x)}{2} - \frac{x}{(1-x^2)^2}\Bigg[H\left(x\right) + \, {E_0} \,\, (1 - x)\Bigg], 
\nonumber
\\
&\geq \frac{\Theta^2 (x)}{2} - \frac{1}{2} K  \frac{x(1-x)}{(1-x^2)^2}  
\label{ineqfin}
\\
&= \frac{1}{2} \bigg(\Theta^2(x) - K \, \alpha^2(x)\bigg)
\nonumber
\end{align}
Consider now the quantity $Z(x) \coloneqq \Theta (x) + \sqrt{K} \,\, \alpha(x)$ 
and note that  $\alpha(x)>0$,  $\displaystyle \frac{d \alpha}{dx} >0$. Using \eqref{ineqfin}, we obtain
\begin{align*}
\nu \, \frac{d Z}{dx} > \frac{1}{2} \big(\Theta - \sqrt{K} \, \alpha(x)\big) \, Z(x) \, .
\end{align*}
Since $Z(0) = \Theta(0) + \sqrt{K} \,\, \alpha(0)=0$, we conclude that $Z(x)>0$. 
\end{proof}

When the bounds of Lemma \ref{lem:bound} hold, the solution $\Theta$ of \eqref{eqtheta} cannot blow up in $(0,1)$. 
The hypothesis $\beta + E_0 > 0$ and thus item (ii) of  Lemma \ref{lem:bound} is of theoretical interest, as the parameter $\beta$ 
cannot be a-priori determined.
Nevertheless, using the bound $H(x) < \frac{1}{2} V_\infty^2$, we can prove a variant of (ii) providing a bound from below.
\begin{corollary} \label{cor:lowerbound}
If $E_0 + \frac{1}{2} V_\infty^2 > 0$ then 
\begin{equation*}
{\Theta}(x) > -\sqrt{ 2 E_0 + V_\infty^2 } \,\,\, \frac{\sqrt{x-x^2}}{1-x^2}, \qquad 0<x<1.
\end{equation*}
independently of $\nu > 0$.
\end{corollary}

\subsection{Convergence as $\nu \to 0$}
\label{convergence}
Let $\big\{  (\Theta_\nu, V_\nu) \big\}_{\nu >0}$ be a family of solutions of \eqref{eq-ns-v3},
\begin{equation}
\label{eql}
\begin{aligned}
\nu \, \frac{d \Theta_\nu}{dx}  &=  \tfrac{1}{2} \Theta_\nu^2  - \mathcal{F}\left(V_\nu ; x\right), 
\\
\nu \frac{d^2 {V_\nu}}{d^2x} &= \Theta_\nu  \, \frac{d {V_\nu}}{dx} ,
\end{aligned}
\end{equation}
with $\nu > 0$  and we are interested in the limit $\nu \to 0$.
We assume that $\kappa := 2E_0+ V_\infty^2>0$ and that the family $(\Theta_\nu, V_\nu)$ satisfies the uniform bound
\begin{equation}
\label{hypothesis}
-\sqrt{\kappa}  \, \frac{\sqrt{x(1-x)}}{1-x^2} < \Theta_\nu(x) < 0, \quad \forall \,\, x \in (0,1).
\tag{B}
\end{equation}
The assumption $\kappa > 0$ is natural since only then solutions of stationary self-similar axisymmetric Euler equations exist, see
sections \ref{sec:euler} and \ref{sec:disceuler}. 

Regarding the uniform bounds \eqref{hypothesis}: 
The left hand inequality is guaranteed by Corollary \ref{cor:lowerbound} while the right hand inequality follows from
Lemma \ref{lem:bound} in the more restrictive range $E_0>0$.  The reader should note, that, as expected from numerical computations
in the parameter range $\kappa > 0$,  the solution of \eqref{eql} enters Zone A - the region that $\Theta(x) < 0$ - as $\nu$ decreases  and stays there. 
This is corroborated by the bifurcation diagram sketched in Figure \ref{bfd1}.

Without loss of generality, we will examine the case $V_\infty>0$. Then \eqref{eq-ns-v3-2} implies 
\begin{align}
\label{dV}
\frac{d V_\nu}{dx} &= V_\infty  \frac{e^\mathlarger{{\frac{1}{\nu} \int_0^x \Theta_\nu(s) ds}}} {\int_0^1 e^{\frac{1}{\nu} \int_0^t \Theta_\nu(s) ds} dt} \,   > 0
\\
0  &< V_{\nu} (x) < V_\infty \, , \quad x \in (0,1) \, .
\end{align}
The sequence $\{ V_\nu \}$ is of bounded variation and by Helly's theorem it admits a convergent subsequence (which will be still denoted by $V_\nu$) such that
\begin{equation}\label{convV}
V_{\nu}(x) \rightarrow V(x) \quad a.e \;  \;  x \in (0,1) \, .
\end{equation}

For the sequence $\{ \Theta_\nu \}$ the uniform bound imply that for $p \in [1,2)$
$$
\int_0^1 |\Theta_\nu (x)|^p dx \le C \int_0^1 (1-x)^{-\frac{p}{2}} dx = : K < \infty
$$
Using weak convergence techniques \cite{Bre2010}, there exists a subsequence  (still denoted by $\Theta_\nu$) and a function 
$\Theta \in L^p (0,1)$ such that 
\begin{equation}\label{wkconvTheta}
\begin{aligned}
\Theta_\nu  &\rightharpoonup \Theta \qquad \mbox{ weakly in $L^p(0,1)$}
\\
\mbox{that is} \qquad  \int \Theta_\nu \psi dx  &\to \int \Theta \psi dx \qquad \mbox{for $\psi \in L^{p^\prime}(0,1)$}
\end{aligned}
\end{equation}
with $p^\prime$ the dual exponent of $p$. 
Setting $\psi (y) = \chi_{[0,x]} (y)$, the characteristic function of $[0,x]$, we deduce
\begin{equation}
\label{limit-G}
G_\nu(x) = \int_0^x \Theta_\nu (s) ds  \rightarrow G(x) :=  \int_0^x \Theta(s) \, ds  \quad  
\end{equation}

Next, we employ ideas developed in the context of zero-viscosity limits for Riemann problems of conservation laws  \cite{Tz94,Tz96,Pap99}.
We set
\begin{equation}
\label{measure}
\mathop{\mathlarger{\pi_\nu \coloneqq \frac{d \, V_\nu}{dx}= V_\infty  \, \, \frac{e^{\frac{1}{\nu} G_\nu(x)}} {\int_0^1 e^{\frac{1}{\nu} G_\nu(t)} dt}}},
\end{equation}
where $\mathlarger{G_\nu(x) = \int_0^x \Theta_\nu(s) \, ds}$ and view $\{ \pi_\nu \}$ as a sequence of probability measures. 
By \eqref{measure}, the sequence $\{ \pi_\nu \}$ is uniformly bounded in measures and there exists a subsequence $\pi_\nu$ 
and a probability measure $\pi$  such that
\begin{equation}\label{convprob}
\pi_\nu \stackrel{\ast}{\rightharpoonup} \pi, \quad \textrm{in measures } \quad \mathcal{M } [0,1] \, .
\end{equation}
Due to the correspondence between functions of bounded variations and Borel measures \cite[Sec 3.5, Sec 7.3]{Foll2013}
one sees that the distribution function of the measure $\pi$ is precisely the function $V(x+)$ and \eqref{convprob} reflects \eqref{convV}.

\smallskip
We are now ready to state the first convergence theorem.

\begin{theorem}\label{maintheorem}
Assume that $E_0 + \frac{1}{2} V_\infty^2 > 0$ and let $\{ (\Theta_\nu, V_\nu) \}_{\nu > 0}$ be a family of solutions satisfying the uniform bound \eqref{hypothesis}.
There exists  a subsequence and a function $(\Theta, V)$ such that 
$$
 V_\nu \to V \, , \; \;  a.e.  \; x \in (0,1) \qquad \Theta_\nu \rightharpoonup \Theta  \; \; \mbox{wk-$L^p (0,1)$ with $1\le p < 2$}
$$
and $(\Theta, V)$ satisfies
\begin{itemize}
\item[(i)] either $\supp \pi = [0,a]$, $a > 0$, in which case  $\Theta = 0$ on $[0,a]$ while $V = Y(x)$ for $x \in [0,a]$, $V = V_\infty$ for 
$x \in [a,1]$ where $Y(x)$ is some nondecreasing function; 
\item[(ii)]  or $\supp \pi = \{ 0 \}$ and $V = 0$ for $x=0$ while $V= V_\infty$ for $x \in (0,1]$.
\end{itemize}
In either case $(V, \Theta)$ satisfy the differential inequality
\begin{equation}\label{ineqfinal}
\frac{1}{2} \Theta^2  (x) \le  \frac{x}{(1-x^2)^2}\Bigg(H(V; x) + \, {E_0} \,\, (1 - x)\Bigg) 
\end{equation}
in the sense of distributions where $H(V ; x)$ is given by \eqref{defF}.
\end{theorem}

\begin{proof}

For any subset $[0,\alpha] \subset [0,1)$, the Ascoli-Arzela theorem 
with hypothesis \eqref{hypothesis} implies that along a subsequence $G_\nu$ converges uniformly on $[0, \alpha] \subset [0,1)$. 
In particular, 
\begin{equation}
G_\nu(x) \rightarrow G(x) =  \int_0^x \Theta(s) \, ds  \quad  
\end{equation}
uniformly on any compact $[0,\alpha] \subset [0,1)$ and pointwise on $(0,1)$ as indicated in \eqref{limit-G}. 
The limit $G$ is  continuous on $[0,1)$.

We proceed to study the $\supp \pi$. To this end, define
\begin{equation}
\label{maxG}
S =  \{x\in [0,1): G(x) = \max_{y \in [0,1)} G(y) = 0 \},
\end{equation}
The function $G$ is nonincreasing,  $G(0) = 0$, and $S$ the set of point where $G$ attains its maximum. 

\begin{lemma}
\label{supp-lemma1}
Let  $\pi$ the weak-$*$ limit of $\pi_\nu$ defined in \eqref{measure}. Then,
$\supp\pi \subset S$.
\end{lemma}

\begin{proof}
Let $x_0 \in S$ be fixed and suppose there exists $x_1 \in [0,1)$ such that $x_1 \notin S$. Then,
\begin{align*}
G(x_1) < G(x_0) = \max_{y \in [0,1)} G(y).
\end{align*}
Since $G$ is nonincreasing and continuous, there exists $\delta>0$ such that for $x\in [x_0-\delta,x_0+\delta]$ and $y\in [x_1-\delta,x_1+\delta]$
\begin{align}
\label{ineq3}
G(y) \leq  \max_{z \in [x_1-\delta,x_1+\delta]} G(z) <  \min_{z \in [x_0-\delta,x_0+\delta]} G(z) \leq G(x).
\end{align}
Setting $d_M \coloneqq \max_{z \in [x_1-\delta,x_1+\delta]} G(z) $, $D_m \coloneqq \min_{z \in [x_0-\delta,x_0+\delta]} G(z)$, we have by
\eqref{ineq3} that  $d_M< D_m$ and we fix $\eta$ such that
\begin{equation*}
0 < \eta < \frac{D_m-d_M}{4}.
\end{equation*}
Define $J =  [x_0-\delta,x_0+\delta] \cup [x_1-\delta,x_1+\delta]  \subset [0,1)$ and select $\delta$ small so that $J \subset [0,1)$. 
Since $G_\nu \to G$ uniformly on J there exist $\nu_0>0$ such that for any $\nu<\nu_0$ 
\begin{equation}
\label{ineq4}
G_\nu(y) \leq G(y) + \eta \leq d_M +\eta < D_m - \eta \leq G(x) - \eta \leq G_\nu(x),
\end{equation} 
for $x\in [x_0-\delta,x_0+\delta]$ and $y\in [x_1-\delta,x_1+\delta]$. Now, 
\begin{align*}
\pi_\nu([x_1-\delta,x_1+\delta]) &= \int_{[x_1-\delta,x_1+\delta]} \pi_\nu (dx)  = \mathlarger{ V_\infty  \, \, \frac{\int_{[x_1-\delta,x_1+\delta]} e^{\frac{1}{\nu} G_\nu(t)} dt} {\int_0^1 e^{\frac{1}{\nu} G_\nu(t)} dt}},\\
&\leq \mathlarger{ V_\infty  \, \, \frac{\int_{[x_1-\delta,x_1+\delta]} e^{\frac{1}{\nu} G_\nu(t)} dt} {\int_{[x_0-\delta,x_0+\delta]} e^{\frac{1}{\nu} G_\nu(t)} dt}}.
\end{align*}
From \eqref{ineq4}, it follows
\begin{align*}
\pi_\nu([x_1-\delta,x_1+\delta]) \le  V_\infty  \, \, \mathlarger{e^{-\frac{1}{\nu} (D_m -d_M - 2\eta)}} \rightarrow 0, \quad \mbox{as $\nu \to 0$.}
\end{align*}
Hence,  $\pi([x_1-\delta,x_1+\delta]) = 0$ and $x_1 \notin$ $\mbox{supp}\,\pi$. 
\end{proof}

The special structure of our problem induces a structural property on the support of $\pi$.

\begin{lemma}
\label{supp-lemma2} Let $\supp \pi \ne \emptyset$. 
If  $\bar{x}>0$ satisfies $\bar{x}\in supp \,\, \pi$ then for any $0<x_0 < \bar{x}$ we have $x_0 \in supp \,\, \pi$.
\end{lemma}

\begin{proof}
Let $\bar{x}\in supp \,\, \pi$, $\bar{x}>0$. Then, for any $\delta > 0$ we have $\pi\big([\bar{x}-\delta,\bar{x}+\delta]\big) >0$.
We will prove that if $0 < x_0 < \bar{x}$ and $\delta$ as above then
\begin{equation}\label{claim}
\pi\big([x_0-\delta,x_0+\delta]\big) >0 \, .
\end{equation}
In turn, that implies  $x_0 \in \supp \pi$.

Using  \eqref{measure}, we have
\begin{align*}
0<\pi_\nu([\bar{x}-\delta,\bar{x}+\delta]) = \mathlarger{ V_\infty  \, \, \frac{\int_{[\bar{x}-\delta,\bar{x}+\delta]} e^{\frac{1}{\nu} G_\nu(t)} dt} {\int_0^1 e^{\frac{1}{\nu} G_\nu(t)} dt}}.
\end{align*}
Recall that $G_\nu$ is nonincreasing.
Using the change of variables $t = y + \tau$ with $\tau=\bar{x}-x_0 >0$, in the top integral implies
\begin{align*}
\int_{[\bar{x}-\delta,\bar{x}+\delta]} e^{\frac{1}{\nu} G_\nu(t)} dt = \int_{[x_0-\delta,x_0+\delta]} e^{\frac{1}{\nu} G_\nu(\tau+y)} dy \le
\int_{[x_0-\delta,x_0+\delta]} e^{\frac{1}{\nu} G_\nu(y)} dy
\end{align*} 
Therefore,
\[\pi_\nu([\bar{x}-\delta,\bar{x}+\delta])  \leq \pi_\nu([x_0-\delta,x_0+\delta]).\] 
Letting $\nu\to0$, we conclude
\[0< \pi([\bar{x}-\delta,\bar{x}+\delta])  \leq \pi([x_0-\delta,x_0+\delta]).\]
which shows \eqref{claim}
\end{proof}

Since $0 = V(0) < V(1) = V_\infty$ we have that $\supp \pi \ne \emptyset$. Moreover, the form of $G(x)$
leads to the following lemma.

\begin{lemma}
\label{supp-lemma3}
Suppose there exists $\bar{x}>0$ such that $\bar{x}\in \supp \pi$. Then,
\begin{equation}
\Theta(x) = 0 \quad \textrm{for a.e.} \quad x\in \mbox{supp}\, \pi.
\end{equation}
\end{lemma}

\begin{proof}
Since $\Theta \in   L^1 (0,1)$, the Lebesgue differentiation theorem implies
\[\displaystyle \lim_{\substack{|I|\to0 \\ \bar{x}\in I}} \frac{1}{|I|} \int_I \Theta(s) \, ds = \Theta(x) \quad a.e.,\]
where $I$ is any interval containing $x$. If $\bar{x}\in \mbox{supp}\,\pi \subset S$  then
\begin{align*}
\displaystyle \lim_{\substack{h\to0 \\ h>0}} \frac{1}{h} \int_{\bar{x}}^{\bar{x}+h} \Theta(s) \, ds  &= \displaystyle \lim_{\substack{h\to0 \\ h>0}} \frac{1}{h} \Big(G(\bar{x}+h) - G(\bar{x}) \Big) \leq 0,\\
\displaystyle \lim_{\substack{h\to0 \\ h<0}} \frac{1}{h} \int_{\bar{x}-h}^{\bar{x}} \Theta(s) \, ds &= \displaystyle \lim_{\substack{h\to0 \\ h>0}} \frac{1}{h}\Big(G(\bar{x}) - G(\bar{x}-h) \Big) \geq 0. 
\end{align*}
Hence it follows that $\Theta(\bar{x}) = 0 \quad a.e.$
\end{proof}

By virtue of Lemma \ref{supp-lemma2} and the fact that $\supp \pi \ne \emptyset$ there are two possibilities:
either (i)  $\supp \pi = [0, a]$ for some $a > 0$, or (ii) $\supp \pi = \{ 0 \}$. In either case 
$$
\supp \pi \subset S = \{x \in [0,1): G(x) = 0 \}   \qquad [\Theta(x) = 0 \quad  \mbox{ a.e. for $x \in \supp \pi$} ]
$$
We conclude:

\textbf{Case 1:} $\supp \pi = [0, a]$ with $a > 0$. Then $\Theta(x)=0$ on $[0,a]$ and 
\begin{align}
\label{limit-V1} 
\quad V(x) =
\left\{
\begin{aligned}
& Y(x)  \quad && x \in [0,a] \\
& V_\infty, \quad && x \in [a, 1)\\
\end{aligned}
\right. 
\end{align}
for some nondecreasing function $Y(x)$.

\noindent
\textbf{Case 2:} $\supp \pi = \{ 0 \}$ and 
\begin{align} 
\label{limit-V2}
V(x) =
\left\{
\begin{aligned}
& 0  \quad && x=0 \\
& V_\infty, \quad && x \in (0, 1)\\
\end{aligned}
\right. 
\end{align}

This provides some information on $(V, \Theta)$ but it is incomplete. Some further information is obtained by passing to
the limit $\nu \to 0$ in \eqref{eqtheta}. First, the convergence  \eqref{convV} implies
$$
H(V_\nu ; x ) \to H(V ; x) =  \Bigg[\int_x^1 \frac{1-t^2}{t^2} \bigg(\displaystyle \int_0^t \frac{\sigma}{(1-\sigma^2)^2} V^2(\sigma) d\sigma \bigg) \, dt \Bigg]
$$
Then \eqref{eqtheta} together with the property that $\Theta (x)^2 \le \mbox{wk-lim} ( \Theta_n (x)^2 )$ imply that $(\Theta, V)$ satisfy
the differential inequality
\begin{equation}
\frac{1}{2} \Theta^2  (x) \le  \frac{x}{(1-x^2)^2}\Bigg(H(V; x) + \, {E_0} \,\, (1 - x)\Bigg)
\end{equation}
in the sense of distributions.
\end{proof}

The inequality \eqref{ineqfinal} is due to the fact that only weak convergence for $\Theta$ is available under \eqref{hypothesis}. It can be improved 
if we have pointwise convergence for $\Theta$, namely if
\begin{equation}\label{hypconvergence}
\Theta_\nu \to \Theta \quad a.e., \,\, as \quad \nu\to 0.  \tag{A}
\end{equation}
We will later justify \eqref{hypconvergence} under the hypothesis $E_0 > 0$. 
Using \eqref{hypconvergence}, one may pass to the limit in \eqref{eqtheta} in the sense of distributions
and deduce
\begin{equation}
\label{limit-Th}
\frac{\Theta^2 (x)}{2} = \mathcal{F}(V,x).
\end{equation}
Let us compute $\Theta(x)$ for Case $2$. Since $V(x)$ is given by \eqref{limit-V2},  \eqref{limit-Th} yields
\begin{align*}
\frac{\Theta^2 (x)}{2} = \mathcal{F}(V_\infty,x) = \big(\frac{1}{2}V_\infty^2 + \, {E_0}\big) \,\, \frac{x(1 - x)}{(1-x^2)^2}
\end{align*}
and thus, since $\Theta(x)<0$, 
\[\Theta(x) = -\sqrt{V_\infty+2E_0}\,\, \frac{\sqrt{x(1 - x)}}{(1-x^2)} \, .  \]

For Case $1$, we recall that $V(x)$ is given by \eqref{limit-V1} and $\Theta(x)=0$ a.e. for $x\in[0,a]$, $a>0$. Then \eqref{limit-Th} 
with \eqref{defF} imply
\begin{align*}
E_0 \,\, (1-x) = - \int_x^1 \frac{1-t^2}{t^2} \bigg(\displaystyle \int_0^t \frac{\sigma}{(1-\sigma^2)^2} V^2(\sigma) d\sigma \bigg) \, dt \, ,
\quad \mbox{for $x \in [0,a]$}.
\end{align*}
If we differentiate this, we get
\begin{align*}
E_0 \,\, \frac{x^2}{1-x^2} = - \displaystyle \int_0^x \frac{\sigma}{(1-\sigma^2)^2} V^2(\sigma) d\sigma.
\end{align*}
Differentiating once more yields
\begin{align*}
 V^2(x) = - 2\,\, E_0 \quad for \,\, x\in(0,a) 
\end{align*}
which contradicts the assumption $\supp \pi = [0,a]$ with $a> 0$ . We conclude that only Case 2 can happen and $(\Theta_\nu, V_\nu)$ converges almost everywhere as $\nu\to0$ to a solution $(\Theta, V)$ of the Euler equations. This provides a criterion to select the type of Euler solution that occurs
and clearly only solutions with solutions with $\Theta<0$ are admissible.

In conclusion, we have the following theorem
\begin{theorem}\label{secondtheorem}
Assume that $E_0 + \frac{1}{2} V_\infty^2 > 0$ and $\{ (\Theta_\nu, V_\nu) \}_{\nu > 0}$ is a family of solutions satisfying the uniform bound \eqref{hypothesis}
and the convergence \eqref{hypconvergence}.
Then $(\Theta, V)$ is a smooth solution of \eqref{ssform} with the form described in section \ref{sec:euler}, but under the restriction
$\Theta (x) < 0$.
\end{theorem}

If the solution takes values in the range where $\Theta(x) < 0$ then it lies in Zone A. The analysis of section \ref{sec:existence} shows that
$F$ in \eqref{boundH} then takes values $F(x) > 0$ on $(0,1)$. The reader should note that numerical computations suggest
 that as $\nu$ decreases  the solution of \eqref{eq-ns-v3} enters Zone A,  that is the region that $\Theta(x) < 0$,  and stays there.
No oscillations in $\Theta$ are observed numerically.

We next restrict in the range $E_0 > 0$ and justify \eqref{hypconvergence}.

\begin{proposition}\label{prop:compactness}
If $E_0 > 0$ then the family of function $\{ A_\nu (x) \}$ defined by
\begin{equation}\label{defA}
A_\nu (x) = (1-x^2) \Theta_\nu (x)
\end{equation}
is of bounded variation on $[0,1]$ and along a subsequence $\{ \Theta_\nu \}$ satsfies \eqref{hypconvergence}.
\end{proposition}

\begin{proof}
Using \eqref{eqtheta} we see that the functions $\{ A_\nu \}$ in \eqref{defA} satisfy the differential equation
\begin{equation}\label{eqA}
\nu (1-x^2) \frac{dA_\nu}{dx} = \tfrac{1}{2} A_\nu^2 - 2 x \nu A_\nu - R(x)
\end{equation}
where
\begin{equation}\label{defR}
R(x) := x F(x) = x \Big ( H(x) + E_0 (1-x) \Big )
\end{equation}
Observe that  $A(0) = 0$ and using \eqref{def_x}, \eqref{defvar}, \eqref{firsttrans} and the L'Hopital rule we compute
$$
A(1) = \lim_{x \to 1} (1-x^2) \Theta (x) = \lim_{\xi \to \infty} \frac{\hat \theta (\xi)}{1+\xi^2} = \lim_{\xi \to \infty} \frac{\theta(\xi)}{\sqrt{1+\xi^2}} = \theta^\prime (\infty) = 0
$$
The uniform bounds \eqref{hypothesis} implies
\begin{equation}\label{boundofA}
- \sqrt{\kappa} \sqrt{x(1-x)} \le A_\nu (x) < 0
\end{equation}

Next, turn to \eqref{defR} and use the hypothesis $E_0 > 0$ together with \eqref{deriv1}, \eqref{deriv2} to conclude that
\begin{equation}\label{convexprop}
\begin{aligned}
R(x) > 0 \, , \quad  \frac{d^2 R}{dx^2} = x \frac{d^2 H}{dx^2} + 2 x \Big ( \frac{dH}{dx} - E_0 \Big )  < 0 \quad x \in (0,1)
\end{aligned}
\end{equation}
$\frac{dR}{dx}(0) = H(0) + E_0 > 0$, $\frac{dR}{dx} (1) = \frac{dH}{dx}(1) - E_0 < 0$. We see that $R(x)$ has a concave graph, vanishing at the endpoints,
facing downwards.

Since $A(0)=A(1) = 0$ the function $A(x)$ must have at least one minimum, that is by the nature of the boundary condition the function $A(x)$ oscillates once. 
Again due to the boundary conditions it must have an odd number of oscillations. By Sard's theorem the set of critical values of $A$ has measure zero.
If the graph of $A$ has three oscillations, then there will be a level $c < 0$ 
- which can be selected so as not to be a critical value - and four consecutive points $x_1 < x_2 < x_2 < x_4$ such that
$A(x_1) = A(x_2) = A(x_3) = A(x_4) = c$ while $\frac{dA}{dx} (x_1) < 0$, $\frac{dA}{dx} (x_2) >  0$, $\frac{dA}{dx} (x_3) < 0$ and $\frac{dA}{dx} (x_4) >  0$.

The function  $f(x) := \tfrac{1}{2}c^2 - 2 \nu x c - R(x)$ satisfies
$$
f(x_1) > 0 , \quad f(x_2) <  0 \, , \quad f(x_3) >  0 \, ,  \quad  f(x_4)  <  0 \, .
$$
Then, there are three consecutive points $y_1 < y_2 < y_3$ in $(0,1)$ where the function $f$ vanishes. This contradicts the fact that by \eqref{convexprop} the 
function $f$ is strictly convex. 

We conclude that $A(x)$ is initially decreasing, reaches a minimum and is afterwards increasing. It also satisfies \eqref{boundofA}. Hence 
$\{ A_\nu \}_{\nu > 0}$ has uniformly bounded total variation and along a subsequence 
$A_\nu  (x) \to A (x) $  for almost every $x \in (0,1)$.  Obviously $A_\nu \rightharpoonup A$ weakly and we conclude using \eqref{wkconvTheta}
that $A = (1-x^2) \Theta$ and \eqref{hypconvergence} holds.
\end{proof}

Combining Proposition \ref{prop:compactness} with Theorem \ref{secondtheorem} gives

\begin{corollary}\label{final}
If $E_0 > 0$ then the family of solutions 
$\{ (\Theta_\nu, V_\nu) \}_{\nu > 0}$  is of bounded variation and along a subsequence 
converges, $\Theta_\nu \to \Theta$ and $V_\nu \to V$ a.e. in $(0,1)$. The function
$(\Theta, V)$ is a smooth solution of \eqref{ssform} of the form described in section \ref{sec:euler}, and satisfies the restriction
$\Theta (x) < 0$.
\end{corollary}

Note that when $V_\infty > 0$ among the two solutions of the Euler equations \eqref{eulersolution} the one selected at the zero-viscosity limit
corresponds to the negative sign, see Fig. \ref{fig:th} (b). One easily checks that when $V_\infty < 0$ again the negative sign is selected.


\section{Boundary layer analysis for a  model problem}
\label{bound-layer}
In this section, we  investigate the asymptotic behavior of solutions of system \eqref{eq-ns-v3} as $\nu\to0$. We consider a model problem which is
 a simplification of the initial equations, with the objective to understand the boundary layer. 
 For small viscosities, $V(x)$ is approximated by setting  $V(x) \equiv V_\infty$, leading to 
\[\mathcal{F}(V,x) = \mathcal{F}(V_\infty,x) = \big(\frac{V_\infty^2}{2} + \, {E_0}\big) \frac{x - x^2}{(1-x^2)^2}.\]
This reduces system \eqref{eq-ns-v3} to a simpler form which will be referred to as the model problem,  namely, 
\begin{subequations}
\label{model-sys}
\begin{align}
\label{model-Th}
\frac{\bar{\Theta}^2 (x)}{2} - \nu \, \frac{d \bar{\Theta}}{dx}  (x)  &= K \frac{x - x^2}{(1-x^2)^2},\\
\label{model-V}
\nu \frac{d^2 \bar{V} }{d^2x} & = \bar{\Theta} \,\, \frac{d \bar{V}}{dx}, 
\\[5pt]
\bar{\Theta}(0)=0,\quad \bar{V}(0)&=0, \quad  \bar{V}(1)=V_\infty,
\end{align}
\end{subequations}
where $K= \frac{V_\infty^2}{2} + \, {E_0}>0$. Note that equation \eqref{model-Th} for $\bar{\Theta}$ is now independent of $\bar{V}$, while \eqref{model-V} is still coupled.

We use the method of matched asymptotic expansions from singular perturbation theory. The method aims to to construct an asymptotic approximation of the solution inside the boundary layer and a solution valid away from the boundary layer, and then combine them through a matching process. 
We refer to solutions within the boundary layer as inner solutions, and to solutions away of the layer as the outer solutions,  \cite{Lage88,Holmes2003}.

We consider first the equation \eqref{model-Th} for $\bar{\Theta}$
\begin{subequations}
\begin{align}
\label{mod-eq}
\frac{\bar{\Theta}^2 (x)}{2} - \nu \, \frac{d \bar{\Theta}}{dx}  (x)  &= K \frac{x - x^2}{(1-x^2)^2},\\
\label{mod-bc1}
\bar{\Theta}(0)&=0,
\end{align}
\end{subequations}
and apply the method of matched asymptotic expansions. 
To construct the outer solution, assume that $\bar{\Theta}$ can be written as a power series with powers of $\nu$, 
\begin{equation*}
\bar{\Theta}(x) \approx \bar{\Theta}_0(x) + \nu \bar{\Theta}_1(x) + \mathcal{O}(\nu^2),
\end{equation*}
and substitute it back to \eqref{mod-eq}. If we focus on the leading terms, i.e terms of order $\nu^0$, we obtain the equation
\begin{equation}
\bar{\Theta}_0(x) = \pm \sqrt{ 2 \, K \frac{x - x^2}{(1-x^2)^2}}.
\end{equation}
The choice for the sign $\bar{\Theta}_0$ will depend on $\bar{\Theta}$. Recall that $\bar{\Theta}$ solve the equation \eqref{mod-eq}, we have 
\[\bar{\Theta}(0) = 0, \quad \frac{d \bar{\Theta}}{dx}(0) = 0, \quad \textrm{and} \quad \frac{d^2 \bar{\Theta}}{dx^2}(0)<0,\]
which implies $\bar{\Theta}$ should be negative for all $0<x<1$. Hence, we choose $\bar{\Theta}_0$ to be negative. Note that the boundary condition at $x=0$ is automatically satisfied.
 

We proceed now with the inner solution. Expecting  the boundary layer to locate at $x=0$, we introduce the stretched variable $\eta = \nu^{-\frac{2}{3}} \, x$. 
Setting $\Phi(\eta) = \nu^{-\frac{1}{3}}{\bar{\Theta}(x)}$, the problem \eqref{mod-eq}-\eqref{mod-bc1} takes the form
\begin{subequations}
\begin{align}
\label{asy-Phi}
& \frac{\Phi^2}{2} - \frac{d \Phi}{d\eta} = K \, \frac{\eta}{(1-\nu^{\frac{2}{3}} \,\eta) (1+\nu^{\frac{2}{3}} \, \eta)^2} \\
& \Phi(0)=0,
\end{align}
\end{subequations}
If the approximation of $\Phi$ in powers of $\nu$ is given by
\begin{equation*}
\Phi(\eta) \approx \Phi_0(\eta) + \nu \Phi_1(\eta)  + \mathcal{O}(\nu^2),
\end{equation*}
and using the Taylor expansion of the right-hand side of \eqref{asy-Phi}, 
\[\frac{\eta}{(1-\nu^{\frac{2}{3}} \,\eta) (1+\nu^{\frac{2}{3}} \, \eta)^2} \approx \eta +   \mathcal{O}(\nu^\frac{2}{3}),\] 
we obtain  for the leading term
\begin{subequations}
\label{asy}
\begin{align}
\label{asy-eq}
& \frac{\Phi_0^2}{2} - \frac{d \Phi_0}{d\eta}   = K \, \eta, \\
\label{asy-bc}
& \Phi_0(0)=0.
\end{align}
\end{subequations}

Inspired by \cite{MR76}, equation \eqref{asy-eq} can be transformed into the well-known Airy equation
\[y''(t) -  t \,y(t)=0.\]
Using the transformation
\begin{equation}
\label{Airy}
\Phi_0(\eta) = -\frac{2}{\mathcal{U}(\eta)} \frac{d \mathcal{U}}{d\eta},
\end{equation}
equation \eqref{asy-eq} reduces to
\[\frac{d^2 \mathcal{U}}{d\eta^2} = \frac{K}{2}  \eta \, \mathcal{U}(\eta).\]
and its solutions are given as a linear combination of special functions $Ai$, the Airy function of the first kind, and $Bi$, the Airy function of the second kind, see \cite[Appendix B.1]{Holmes2003}. Namely,
\[\mathcal{U}(\eta) = \text{Ai}\left(\big(\frac{K}{2}\big)^{1/3} \eta \right) c_1 +\text{Bi} \left(\big(\frac{K}{2}\big)^{1/3} \eta \right) c_2,\]
where $c_1, c_2$ are integration constants. Using \eqref{Airy}, we express $\Phi_0$ as a linear combination of Airy functions $Ai, Bi$ and their derivatives  $Ai', Bi'$. Hence, $\Phi_0$ takes the form
\[\Phi_0(\eta) = -2 \bigg(\frac{K}{2}\bigg)^{1/3} \, \frac{\text{Ai}'\left(\big(\frac{K}{2}\big)^{1/3} \eta \right) c_1 +\text{Bi}'\left(\big(\frac{K}{2}\big)^{1/3} \eta \right) c_2}{\text{Ai}\left(\big(\frac{K}{2}\big)^{1/3} \eta \right) c_1 +\text{Bi} \left(\big(\frac{K}{2}\big)^{1/3} \eta \right) c_2}, \]
or equivalently,
\[\Phi_0(\eta) = -2 \bigg(\frac{K}{2}\bigg)^{1/3} \, \frac{\text{Ai}'\left(\big(\frac{K}{2}\big)^{1/3} \eta \right) C +\text{Bi}'\left(\big(\frac{K}{2}\big)^{1/3} \eta \right)}{\text{Ai}\left(\big(\frac{K}{2}\big)^{1/3} \eta \right) C +\text{Bi} \left(\big(\frac{K}{2}\big)^{1/3} \eta \right) }, \]
where $C= \frac{c_1}{c_2}$ is a constant. Imposing now the boundary condition \eqref{asy-bc}, we get
\begin{align*}
0 = \frac{\text{Ai}'\left(0\right) C +\text{Bi}'\left(0 \right)}{\text{Ai}\left(0\right) C +\text{Bi} \left(0 \right)} = \frac{-\frac{C}{3^{1/3} \, \Gamma \left(\frac{1}{3}\right)} + \frac{3^{1/6}}{\Gamma \left(\frac{1}{3}\right)}}{\frac{C}{3^{2/3} \Gamma \left(\frac{2}{3}\right)}+\frac{1}{3^{1/6} \Gamma \left(\frac{2}{3}\right)}}= \frac{3^{1/3} \, \Gamma(\frac{2}{3}) (-C+ \sqrt{3})}{\Gamma(\frac{1}{3})(C+ \sqrt{3})},
\end{align*}
where $\Gamma$ denotes the Gamma function. Thus, the constant $C$ is determined as
\[C = \sqrt{3}.\]
Consequently, the leading term $\Phi_0$ of the inner solution becomes
\[\Phi_0(\eta) = -2 \bigg(\frac{K}{2}\bigg)^{1/3} \frac{\sqrt{3}\, \text{Ai}'\left(\big(\frac{K}{2}\big)^{1/3} \eta \right)+\text{Bi}'\left(\big(\frac{K}{2}\big)^{1/3} \eta \right)}{\sqrt{3} \, \text{Ai}\left(\big(\frac{K}{2}\big)^{1/3} \eta \right)+\text{Bi} \left(\big(\frac{K}{2}\big)^{1/3} \eta \right)}.\]

The last step of the method of matched asymptotic expansions is to derive a uniform expansion of the solution of \eqref{mod-eq} over the whole domain $[0,1)$. Combining the approximations of inner and outer solution, we conclude that the asymptotic expansion of $\bar{\Theta}$ as $\nu\to0$ is
\begin{equation}\label{asymd-theta}
\bar{\Theta}(x) \approx 
\left\{
\begin{aligned}
&-\sqrt{ 2 \, K \frac{x - x^2}{(1-x^2)^2}}, \quad && A\nu^{\frac{2}{3}}<x<1\\
&\Phi_0(\nu^{-\frac{2}{3}} \, x), \quad && 0\leq x \leq A \nu^{\frac{2}{3}}
\end{aligned}
\right.
\end{equation}
where $A$ is a positive constant and
\[\Phi_0(\nu^{-\frac{2}{3}} \, x) = -  2 \nu^{{1}/{3}} \bigg(\frac{K}{2}\bigg)^{1/3} \frac{\sqrt{3} \text{Ai}'\left(\big(\frac{K}{2}\big)^{1/3} \nu^{-\frac{2}{3}} \, x \right)+\text{Bi}'\left(\big(\frac{K}{2}\big)^{1/3} \nu^{-\frac{2}{3}} \, x \right)}{\sqrt{3} \text{Ai}\left(\big(\frac{K}{2}\big)^{1/3} \nu^{-\frac{2}{3}} \, x \right)+\text{Bi} \left(\big(\frac{K}{2}\big)^{1/3} \nu^{-\frac{2}{3}} \, x \right)}.\]
Moreover, the boundary layer is formed at $x=0$ and its size is of order $\nu^{\frac{2}{3}}$.

Let us consider now the equation \eqref{model-V} for $\bar{V}$. Motivated by the asymptotic expansion of $\bar{\Theta}$, we introduce the following simplified problem
\begin{subequations}
\label{model3}
\begin{align}
\label{mod-V}
& \nu \frac{d^2 \bar{V} }{d^2x}(x) = -\sqrt{2\,K\ \frac{x - x^2}{(1-x^2)^2}} \,\, \frac{d \bar{V}}{dx} (x), \\
\label{mod-bc2}
& \bar{V}(0)=0, \quad  \bar{V}(1)=V_\infty,
\end{align}
\end{subequations}
and apply the method of matched asymptotic expansions.  As before, we expect the boundary layer to be located at $x=0$.

First, we derive the approximation of the outer solution. Assuming that $\bar{V}$ can be expressed as a power series with powers of $\nu$, i.e.
\[\bar{V}(x) \approx \bar{V}_0(x) + \nu \bar{V}_1(x) + \mathcal{O}(\nu^2),\]
the equation \eqref{mod-V} yields the following differential equation for the lead term $\bar{V}_0$
\begin{equation}
-\sqrt{2\,K\,\frac{x - x^2}{(1-x^2)^2}} \,\, \frac{d \bar{V}_0}{dx} =0 \quad \textrm{which implies} \quad \bar{V}_0 (x) \equiv c,
\end{equation}
where $c$ is a constant. To determine this constant, we consider the boundary condition away from the boundary layer, i.e at $x=1$ and thus, \eqref{mod-bc2} implies
\begin{equation}
\bar{V}_0(x) \equiv V_\infty.
\end{equation}
For the inner solution we introduce the variable $\eta = \nu^{-2/3}\,x$ and set $\bar{V}(x) = Y(\eta)$. The problem \eqref{model3} reduces to
\begin{align}
\label{asy-Y}
\frac{d^2 Y}{d^2 \eta} + \sqrt{2\,K} \, \sqrt{\frac{\eta}{(1-\nu^{\frac{2}{3}} \,\eta) (1+\nu^{\frac{2}{3}} \, \eta)^2}} \,\, \frac{d Y}{d\eta} = 0 \quad \textrm{with} \quad Y(0) =0.
\end{align}
If the approximation of $Y$ in powers of $\nu$ is expressed as
\begin{equation*}
Y(\eta) \approx Y_0(\eta) + \nu Y_1(\eta) + \mathcal{O}(\nu^2),
\end{equation*}
and the Taylor expansion of the right-hand side of \eqref{asy-Y} is 
\[\sqrt{\frac{\eta}{(1-\nu^{\frac{2}{3}} \,\eta) (1+\nu^{\frac{2}{3}} \, \eta)^2}} \approx \sqrt{\eta} +  \mathcal{O}(\nu^\frac{2}{3}),\] 
then we have to solve the following problem for leading order term $Y_0$
\begin{align*}
\frac{d^2 Y_0}{d^2 \eta} + \sqrt{2\,K\,\eta}\, \frac{d Y_0}{d\eta} &= 0,\\
Y_0(0) &= 0.
\end{align*}
Its solution is 
\begin{align*}
Y_0(\eta) = C \int_0^\eta \mathlarger{\mathlarger{e^ {-\frac{2 \sqrt{2\,K}} {3} \, s^{3/2}}}} ds,
\end{align*}
where $C$ is an integration constant. Since both boundary conditions \eqref{mod-bc2} been have taken into consideration, the unknown $C$ is determined by  matching the outer and the inner expansions. Considering that both the inner and outer solutions approximate the same function in different regions, we impose that the two solutions are equal in a transition area close to the boundary layer, \cite{Holmes2003}. Therefore, we require
\begin{equation*}
\lim_{x\to 0} \bar{V}_0(x) = \lim_{\eta \to \infty} Y_0(\eta),
\end{equation*}
that implies
\begin{align*}
C = \mathop{ \frac {V_\infty} {\int_0^\infty \mathlarger{\mathlarger{e^ {-\frac{2 \sqrt{2\,K}} {3} \, s^{3/2}}}} ds}},
\end{align*}
and completes the derivation of the inner solution. 

As the last step of the method of matched asymptotic expansions, we derive a uniform expansion of the solution of \eqref{model3} over the whole domain $[0,1)$. Therefore, as $\nu\to0$ the asymptotic expansion of $\bar{V}(x)$ takes the form
\begin{align}
\label{V-asympt}
\bar{V}(x) \approx {V_\infty} \,\, \mathlarger{\frac{\bigints_0^{\nu^{-2/3}\,x}  e^ {-\frac{2 \sqrt{2\,K}} {3} \, s^{3/2}} ds} {\bigints_0^\infty e^ {-\frac{2 \sqrt{2\,K}} {3} \, s^{3/2}} ds}} + \mathcal{O}(\nu).
\end{align}
and the boundary layer at $x=0$ is of the same size as for $\bar{\Theta}(x)$.

\section{Stationary Navier-Stokes - Numerical Results}
\label{sec:numerics}
In this section we construct a numerical scheme to solve the the coupled system of ordinary differential equations \eqref{eq-ns-v1} using an iterative algorithm. Moreover, we illustrate some numerical experiments for different values of the parameters $\nu, \ E_0$ and $V_\infty$.

Let us recall the system \eqref{eq-ns-v1}. For convenience we use the equivalent formulation of the problem, system \eqref{eq-ns-v3-1}-\eqref{bd3_new}  in variable $x$, where $0<x<1$
\begin{subequations}
\begin{align}
\label{eq-ns-v3-Th}
\frac{\Theta^2 (x)}{2} - \nu \, \frac{d \Theta}{dx}  (x)  &=  \mathcal{F}(V,x),\\
\label{eq-ns-v3-V}
\nu \frac{d^2 V}{d^2x}(x) &= \Theta(x) \, \frac{d V}{dx}  (x),
\end{align}
\end{subequations}
equipped with the boundary conditions
\begin{subequations}
\begin{align}
\label{eq-bc-v3-ThV}
\Theta = V = 0 , \,\, &\textrm{ at } \quad x = 0, \\
\label{eq-bc-v3-V}
V \to V_\infty , \,\, &\textrm{ as} \quad x \to 1,
\end{align}
\end{subequations}
where the functional $\mathcal{F}\left(V,x\right)$ is defined by
\begin{align}\label{defFE}
\mathcal{F}\left(x\right) = \mathcal{F}\left(V,x; E_0 \right)  & = \frac{x}{(1-x^2)^2}\Bigg[\int_x^1 \frac{1-t^2}{t^2} \bigg(\displaystyle \int_0^t \frac{\sigma}{(1-\sigma^2)^2} V^2(\sigma) d\sigma \bigg) \, dt + \, {E_0} \,\, (1 - x)\Bigg].
\end{align}

\subsection{Discretization}
To obtain numerical approximations to system \eqref{eq-ns-v3-Th}-\eqref{eq-bc-v3-V} we use an iterative algorithm: Given $V$ we solve numerically the initial value problem \eqref{eq-ns-v3-Th} by a Runge-Kutta method and update $\Theta$ which is then used to solve  the two-point boundary value problem \eqref{eq-ns-v3-V} using finite differences. This process is repeated until convergence. The algorithm is described in detail in Section \ref{sec-impl}.

The discretization  of \eqref{eq-ns-v3-V} is based on finite differences. On $[0,1]$  we introduce a uniform mesh of fixed width $\Delta x$.  Given $N\in\mathbb{N}$, set 
$\Delta x = \frac{1}{N}$ and define the discrete points $x_j = j\Delta x$, $j= 0,1,2,\dotsc, N$. 
We denote by $\Theta^j \approx \Theta(x_j)$ and $V^j \approx V(x_j)$
and use central finite differences to  discretize  \eqref{eq-ns-v3-V}  and obtain
\begin{equation}
\label{FM1}
\begin{aligned}
\nu \,\frac{V^{j+1} - 2V^j + V^{j-1}} {\Delta x^2} & = \Theta^j \, \frac{V^{j+1} - V^{j-1}} {2\Delta x}, \quad j=1,\dotsc, N-1, \\
V^0 = 0, & \ V^N = V_{\infty}, 
\end{aligned}
\end{equation}
or equivalently
\begin{equation}
\begin{aligned}
\label{FM2}
V^{j-1} \bigg(2\nu + {\Delta x} \,\Theta^j \bigg)  - 4 \nu\, V^j  & + V^{j+1} \bigg(2 \nu - {\Delta x} \, \Theta^j \bigg)  = 0 , \qquad j=1,\dotsc, N-1, \\
V^0 = 0, & \ V^N = V_{\infty}, 
\end{aligned}
\end{equation}
which forms a tri-diagonal linear system solved by a direct method. 

To discretize the nonlinear initial value problem \eqref{eq-ns-v3-Th} we use a Runge-Kutta method. The Runge-Kutta methods are multistage methods that compute approximations to the solution at intermediate points which are later combined to advance the solution at the next discretization point.
Equation \eqref{eq-ns-v3-Th} can be rewritten as
\begin{equation}
\nu \frac{d \Theta}{dx} (x)  = \frac{\Theta^2 (x)}{2} - \mathcal{F}(x).
\end{equation}
For simplification of the presentation, we denote the right hand side of the above relation as follows
$f(x, V, \Theta) = \tfrac{1}{2} \Theta^2 - \mathcal{F}(x)$.
The Runge-Kutta method for \eqref{eq-ns-v3-Th} is obtained by discretizing the derivative using the following method:
\begin{enumerate}
	\item Set $\Theta^{j,1} =\Theta^j$, where $\Theta^j$ is the approximate solution at grid point $x_j$.
	\item Compute the $\mathcal{K}$ intermediate stages for $m =1,\dots,\mathcal{K}$ 
		\begin{equation}
		\label{RK1}	
		\nu \Theta^{j,m} = \nu \Theta^j +  {\Delta x} \sum_{s=1}^{\mathcal{K}} \alpha_{m,s}  f \left(x_j+c_s {\Delta x},V^{j,s},\Theta^{j,s}\right),
		\end{equation}
where $V^{j,s} \approx V(x_j+c_s \Delta x)$ and $\Theta^{j,s} \approx \Theta(x_j+c_s \Delta x)$. The coefficients $\alpha_{m,s}$ are constants with sum for every row equals to $c_m$, i.e.
\[ \sum_{s=1}^{\mathcal{K}} \alpha_{m,s} = c_m, \quad m =1,\dots,\mathcal{K}. \]
	\item Set 
\begin{equation}
\nu\Theta^{j+1} = \nu \Theta^j +  {\Delta x} \sum_{s=1}^{\mathcal{K}} \beta_s f \left(x_j+c_s {\Delta x},V^{j,s},\Theta^{j,s}\right)
\end{equation} 
the solution at the next point. Here the coefficients $\beta_s$ are constants with total sum equal to 1 and values of $V^{j,s}$ are computed using interpolation.
\end{enumerate}
The coefficients $\alpha_s, \beta_s$ and $c_s$ of the associated with the Runge-Kutta method are usually presented in a matrix form referred as Butcher tableau. That is,
\[
\renewcommand\arraystretch{1.2}
\begin{array}
{c|cccc}
c_1 & \alpha_{1,1} &\alpha_{1,2} &\cdots &\alpha_{1,\mathcal{K}}\\
c_2 & \alpha_{2,1} &\alpha_{2,2} &\cdots  &\alpha_{2,\mathcal{K}} \\
\vdots & \vdots & \vdots & & \vdots\\
c_\mathcal{K} & \alpha_{\mathcal{K},1}& \alpha_{\mathcal{K},2} &\cdots &\alpha_{\mathcal{K},\mathcal{K}}\\
\hline
& \beta_1 &\beta_2 &\cdots &\beta_\mathcal{K}
\end{array}
\] 
In our numerical experiments we use the well known Runge-Kutta-Fehlberg (RKF4(5)) method which allows adaptive stepsize control. For small values of $\nu$ the  solution $\Theta$ of the  i.v.p \eqref{eq-ns-v3-Th} might blow up, see the discussion in Section 5.3. Furthermore, as $x\to 1$ we expect the  system  to be singular since the point $x=1$ corresponds to the vortex line. In that respect, choosing the stepsize adaptively allows us to capture the correct behaviour of the solution close to the singularity.

\subsection{Implementation Details}\label{sec-impl}
To compute the numerical approximation of the solution for the coupled system \eqref{eq-ns-v3-Th} - \eqref{eq-ns-v3-V}, we construct an iterative algorithm and each equation is solved separately  using information from the previous iterations. In particular, we follow the following algorithm:
\begin{enumerate}
\item Set $V_0^j \equiv V_\infty$, for all $j$
\item For every iterative step $i=1,\dotsc$:
\begin{enumerate}[label=\roman*.]
\item Given $V_{i-1}$,  compute $\Theta_i^j$, $\forall \, j$, using the RKF-method in $[0,1]$
\item Given $\Theta_i$, compute $V_i^j$, $\forall \, j$, by solving the linear system \eqref{FM2} 
\item Solution $(\Theta,V)$ is obtained when the error between two consecutive iterations is small, i.e.
\begin{equation*}
\bigg\{\Delta x \sum_j |V_{i+1}^j - V_i^j | ^2 \bigg\}^\frac{1}{2} < \epsilon \quad \textrm{and}\quad \bigg\{\Delta x \sum_j |\Theta_{i+1}^j - \Theta_i^j | ^2 \bigg\}^\frac{1}{2} < \epsilon
\end{equation*}
for some $\epsilon>0$ small.
\end{enumerate}
\end{enumerate}
Note that for the approximation of the integrals in $\mathcal{F}$ we use the composite Simpson's rule.


\subsection{Numerical Tests}
In this section we exhibit numerical approximation of the solution for  \eqref{eq-ns-v3-Th} - \eqref{eq-ns-v3-V}. We illustrate the results for a different combinations of parameters $\nu$ and $E_0$ while the parameter $V_\infty$ is fixed. For the following examples we take $V_\infty=1$.

Consider first the case for $\nu = 0.2$ and $E_0 = -0.5$. Then $\Theta(x)$ is positive, and the flow is directed inward near the plane $z = 0$ and upward near the vortex line, see Figure \ref{fig:Th-sol2}. 
\begin{figure}[htbp] 
\centering
\includegraphics[scale=0.5]{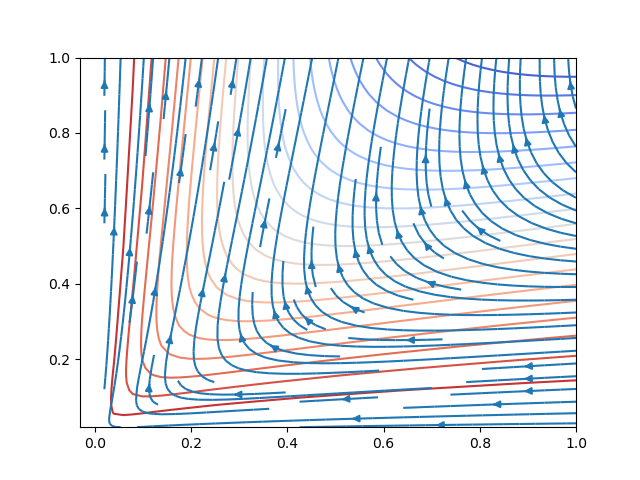}
\caption{(u,w) vector field for $\nu = 0.2, E_0 = -0.5$} \label{fig:Th-sol2}
\end{figure}

The converse behavior, i.e. $\Theta$ is negative, occurs for  $\nu = 0.05$ and $E_0 = 0.15$. The flow now has the reverse direction, it is directed outward near the plane $z = 0$ and downward near the vortex line, see Figure \ref{fig:Th-sol1}.
\begin{figure}[htbp] 
\centering
\includegraphics[scale=0.5]{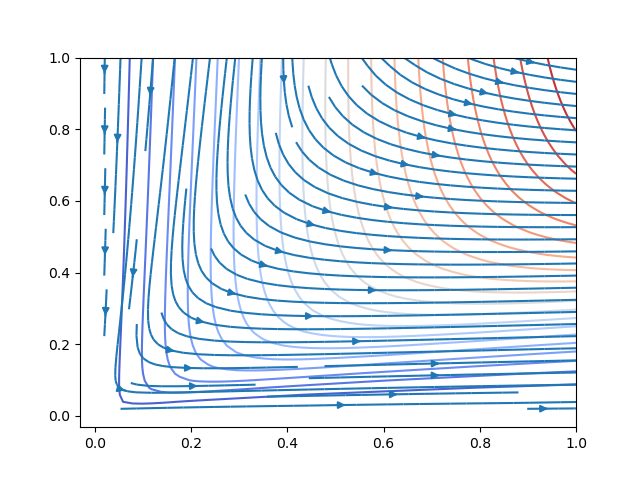}
\caption{(u,w) vector field $\nu = 0.05, E_0 = 0.15$} \label{fig:Th-sol1}
\end{figure}
For the parameters $\nu = 0.1$ and $E_0 = -0.2$, the stream function $\Theta$ is first positive and then becomes negative; the flow is directed inward near the plane $z = 0$ and downward near the vortex line, see Figure \ref{fig:Th-sol3}.
\begin{figure}[htbp] 
\centering
\includegraphics[scale=0.5]{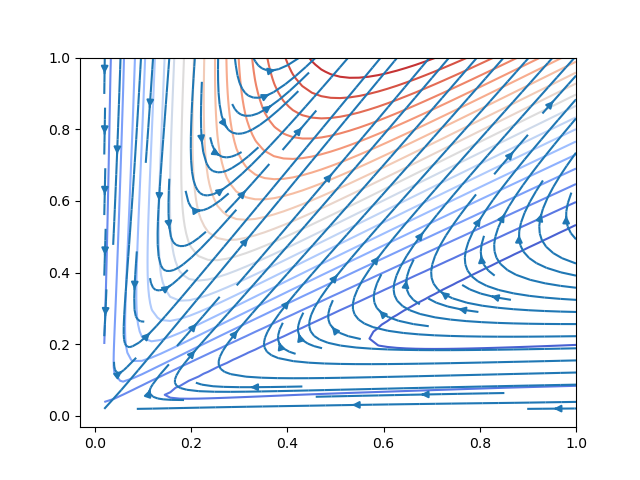}
\caption{(u,w) vector field $\nu = 0.1, E_0 = -0.2$} \label{fig:Th-sol3}
\end{figure}
The solid lines in Figures \ref{fig:Th-sol1}, \ref{fig:Th-sol2} and \ref{fig:Th-sol3}  are the contour lines of the streamfunction $\Theta$.
\subsection{Bifurcation diagram}
\label{sec:bifurcation}
We present now a $E_0 - \nu$  bifurcation diagram for \eqref{eq-ns-v3}. This diagram is computed using the methodology detailed in Section \ref{sec-impl}. To take into account the effect of the parameter $V_\infty$ we consider the following scaled variables:
$$
\phi = \frac{V}{V_\infty} \qquad \vartheta = \frac{\Theta}{V_\infty}, \qquad \mu = \frac{\nu}{V_\infty}, \qquad p_0 = \frac{E_0}{V_\infty^2}
$$
Then system \eqref{eq-ns-v3-Th}-\eqref{eq-bc-v3-V} becomes 
\begin{equation}
\label{scaledsys}
\begin{aligned}
\mu \frac{d\vartheta}{dx} & = \frac12\vartheta^2 - \mathcal{F}(x, \phi, p_0), \\
\mu \frac{d^2 \phi}{dx^2} & = \vartheta\frac{d\phi}{dx}, \\
\phi = \vartheta = 0 &  \ \text{ at } \ x=0,   \\
\phi \to 1, & \ \text{ as } \ x\to 1 ,
\end{aligned}
\end{equation}
where $\mathcal{F}$ is given by \eqref{defFE}.
The new scaled system exhibits the same behaviour as the original one, so we proceed in identifying numerically the four different zones, see Section \ref{sec:existence}:
\begin{itemize}
\item \emph{Zone A:}  $\vartheta$ is negative 
\item \emph{Zone B:}  $\vartheta$ starts positive and then changes sign 
\item \emph{Zone C:}  $\vartheta$ is positive
\item \emph{No Solution:} $\vartheta$ is positive and for small $\mu$ blows up
\end{itemize} 
We take $V_\infty=1$ and we consider values of $(E_0,  \nu) \leftrightarrow (p_0, \ \mu)$ in $\mathcal{B}= [-2, 2]\times[4\cdot10^{-5}, 0.6]$.  The set $\mathcal{B}$ is covered initially by $256=16\times 16$ patches $\mathcal{B}_{k,\ell}$ each of size $\delta p_0\times \delta\mu$. Along the horizontal axis, we take a uniform size $\delta p_0 = 0.25$, while $\delta\mu$ is non-uniform, with a finer grid around $\mu=4\cdot10^{-5}$ and gradually increasing towards $\mu=0.6$, with an average $\delta\mu \sim 3.75\cdot 10^{-2}$.  In each $\mathcal{B}_{k,\ell}$  we consider further a $51\times 51$ uniform grid of values $\{{p_0}_{i,j}^{k,\ell}, \mu_{i,j}^{k,\ell}\}, \ i,j=1,\dotsc,51$.  For all patches and for all values in the patch we identify in which zone the  solution belongs to, by solving numerically the scaled system \eqref{scaledsys}. The total computational cost of such process is rather small since the work in each patch can be  computed independently. 

The results of this process are depicted in Figure \ref{bfd1}. On the left graph, the four zones are clearly marked by their boundaries. To distinguish further each zone we ``separate" them by shifting slightly each zone along the horizontal axis and  Zone C along the vertical axis. On the right graph, the full structure of each zone is now revealed. The boundaries of all  zones meet at the point $(p_0, \mu) = (E_0, \nu) = (-\frac12, 4\cdot10^{-5})=(-\beta,4\cdot10^{-5})$, where $\beta = \frac{dH}{dx}(1)$ is the constant appearing in \eqref{defalpha}.  For any $E_0 < -\beta$ on the line $\mu=\nu= 4\cdot10^{-5}$ the solution ceases to exist. For larger values of $\nu$, where the solution does exist, the line $E_0=-\beta$ defines the border between Zone C and Zone B  as predicted theoretically, see Section \ref{sec:existence}. 

A similar diagram appears in Serrin \cite{Serrin}. It  uses a slightly different (but equivalent) formulation of the problem, but it is derived through a completely different process.  Our bifurcation diagram in Figure \ref{bfd1} is based on computation  of the numerical solution of \eqref{eq-ns-v3} as detailed in the previous paragraph. The diagram of Serrin is a consequence of a series of  theoretical estimates and bounds that provide necessary and/or sufficient conditions on the underlying parameters,  so that the solutions to the corresponding initial and boundary value problem belong to one of the desired zones; see \cite[pp 349-351]{Serrin}.
Qualitatively the two diagrams are quite similar,  in terms of the  overall structure and shape of the  borders between the various zones. However, quantitively there are some differences, due to the different sets of variables and constants used in the two approaches.  In our case solutions can exist when the pressure constant $E_0$ is negative, a region which is not addressed in Serrin's work, where solutions are obtained only for positive values of pressure. The value $P=1$ in \cite[Fig. 1]{Serrin}  is the point where all zones converge and beyond this point along the axis $\nu = 0$, solution ceases to exist. In the present work this value corresponds to $E_0=-\frac12$ and the whole diagram is mirrored along this line when compared to that in \cite[Fig. 1]{Serrin}.  In both diagrams the two vertical lines define the border between zones B and C. However, it's worth noticing that, apart from having different(opposite) orientations, the curved boundary that separates the \emph{NoSolution} zone from the other zones is very similar in both works.
\begin{figure}[htbp] 
\centering
\includegraphics[scale=0.5]{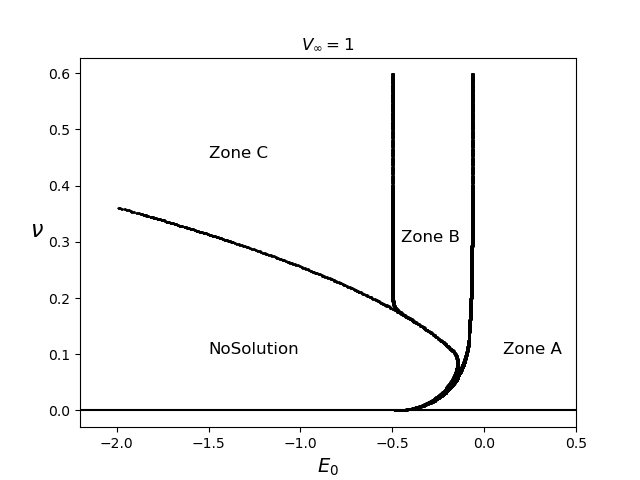}
\includegraphics[scale=0.5]{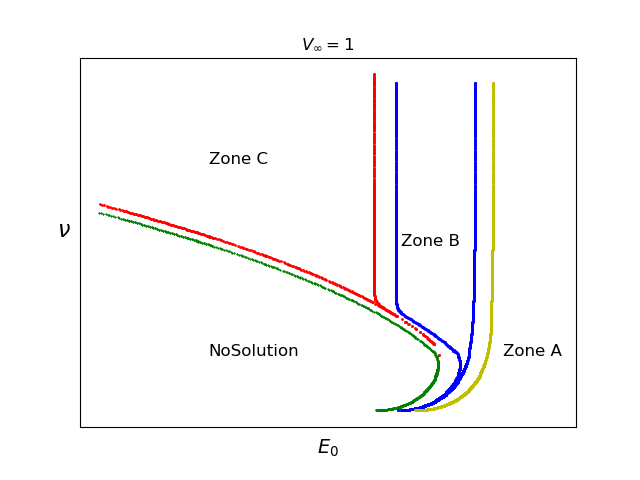}
\caption{$E_0 - \nu$ Bifurcation Diagram} 
\label{bfd1}
\end{figure}

\appendix

\section{Navier-Stokes equations in Cylindrical Coordinates}
\label{sec:cyl-coord}
Cylindrical coordinates $(r,\vartheta,z)$ are connected with rectangular coordinates through the transformation
$x_1 = r \,\cos\vartheta$, $x_2 = r\,\sin\vartheta$, $x_3 = z$,  see Fig. \ref{fig:cyl-coord-3d},  while the associated orthonormal system 
attached to $P$ is
\begin{equation*}
\vec{e}_r = (\cos\vartheta,\sin\vartheta,0), \,\  \vec{e}_\vartheta = (-\sin\vartheta,\cos\vartheta,0), \,\ \vec{e}_z = (0, 0, 1).
\end{equation*}
\begin{figure}[htbp] 
\centering
\tdplotsetmaincoords{70}{110} 
\begin{tikzpicture}[tdplot_main_coords] 
\def\ang{60} 
\def\r{4.5}
\def\z{3.0} 
\draw[->] (0,0,0) -- (6,0,0) node[below]{$x$}; 
\draw[->] (0,0,0) -- (0,6,0) node[below]{$y$}; 
\draw[->] (0,0,0) -- (0,0,3.5) node[left]{$z$}; 
\draw[dotted] (\r,0,\z) -- (\r,0,0) arc(0:90:\r) -- (0,\r,\z)[canvas is xy plane at z=\z] arc(90:0:\r); 
\draw (0,0,0) -- (xyz cylindrical cs:radius=\r,angle=\ang,z=\z) node[point](P){} node[above right]{$P$}; 
\draw[ ] (P) -- node[midway, right]{$z$} (\ang:\r) node[point]{} -- node[midway, below left]{$r$} (0,0) (0:1) arc(0:\ang:1) node[midway, below]{$\vartheta$}; 
\draw[-Latex] (P) -- ++(\ang:1.5) node[below]{$\vec{e}_r$}; 
\draw[-Latex] (P) -- ++(0,0,1) node[left]{$\vec{e}_z$}; 
\draw[-Latex] (P) -- ++({-1.5*sin(\ang)},{1.5*cos(\ang)},0) node[below right,yshift=5]{$\vec{e}_\vartheta$}; 
\end{tikzpicture}
\caption{Cylindrical Coordinate System  \label{fig:cyl-coord-3d}}
\end{figure}
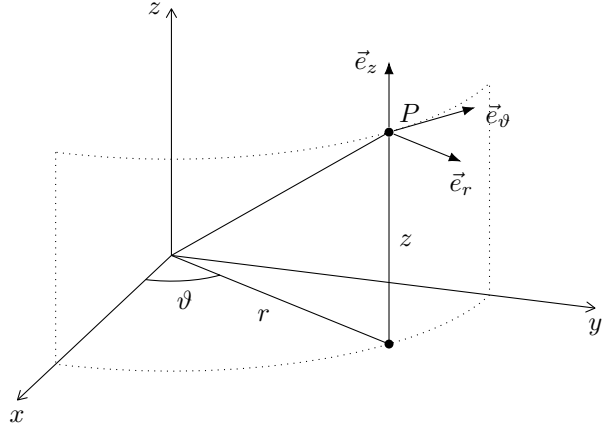
The velocity vector $\vec{u}$ is expressed in cylindrical coordinates as
\begin{equation*}
\vec{u} = u(r,\vartheta,z,t) \vec{e}_r + v(r,\vartheta,z,t) \vec{e}_{\vartheta} + w(r,\vartheta,z,t) \vec{e}_z,
\end{equation*}
while the vorticity $\vec{\omega}$ takes the form
\begin{equation*}
\vec{\omega} = \bigg(\frac{1}{r}\frac{\partial u}{\partial \vartheta} -\frac{\partial v}{\partial z}\bigg)  \vec{e}_r + \bigg( \frac{\partial u}{\partial z}-\frac{\partial w}{\partial r}\bigg)  \vec{e}_\vartheta + \bigg( \frac{1}{r}\frac{\partial}{\partial r} (rv)- \frac{1}{r}\frac{\partial w}{\partial \vartheta} \bigg)  \vec{e}_z.
\end{equation*}
A change of variables for the Navier - Stokes equations \eqref{NS} yields the following representation in cylindrical coordinates, see \cite[Table B.2]{BCAH87}, 
\begin{subequations}
\begin{align}
\label{tNS1-cyl}
\frac{\partial u}{\partial t} + (\vec{u} \cdot \nabla) u - \frac{v^2}{r} & = -\frac{\partial p}{\partial r} + \nu \Big[\Delta u - \frac{u}{r^2} -\frac{2}{r^2}\frac{\partial v}{\partial \vartheta} \Big], \\
\label{tNS2-cyl}
\frac{\partial v}{\partial t} + (\vec{u} \cdot \nabla) v + \frac{uv}{r} & = -\frac{1}{r}\frac{\partial p}{\partial \vartheta}+\nu \Big[\Delta v - \frac{v}{r^2} + \frac{2}{r^2}\frac{\partial u}{\partial \vartheta}  \Big], \\
\label{tNS3-cyl}
\frac{\partial w}{\partial t} + (\vec{u} \cdot \nabla) w \qquad  & =  - \frac{\partial p}{\partial z} + \nu \Delta w, \\
\label{tNS4-cyl}
\frac{1}{r} \frac{\partial }{\partial r} (ru) +\frac{1}{r}\frac{\partial v}{\partial \vartheta} + \frac{\partial w}{\partial z}  & = 0.
\end{align}
\end{subequations}
where the operators $\vec{u} \cdot \nabla$ and $\Delta$ are
\begin{equation*}
(\vec{u} \cdot \nabla) = u \frac{\partial }{\partial r} + \frac{1}{r} v \frac{\partial }{\partial \vartheta} + w \frac{\partial }{\partial z}  \qquad \text{and} \qquad
\Delta   = \frac{1}{r}\frac{\partial }{\partial r} \Big(r \frac{\partial }{\partial r}\Big)+\frac{1}{r^2} \frac{\partial^2  }{\partial \vartheta^2} + \frac{\partial^2  }{\partial z^2}
\end{equation*}

\section{Tornadoes}
\label{sec:tornado}
Tornadoes are considered among the most extreme and violent weather phenomena on Earth. They can occur under appropriate circumstances in all continents expect Antarctica, at various seasonal times and can be hazardous causing loss of human lives and extensive property damage. According to meteorologists, a tornado is defined as a rapidly rotating mass of air that extends downward from a cumuliform cloud, that is a cloud formed due to vertical motion of air parcels, to the ground. There exists several types of tornadoes, such as landspouts and waterspouts, but the majority of the most destructive tornadoes are known as supercells since their generation takes place within supercell thunderstorms. 

Although there is a high interest in forecasting such hazardous tornadoes, it remains a challenging task for researchers to predict when a supercell thunderstorm will lead to a tornado. It is observed that not all supercells are tornadic since a combination of atmospheric instability (caused by the storm) with a wind shear, i.e. a variation of wind speed and direction with altitude, is required for tornado formation. These two ingredients are important for tornado formation, however the process by which tornadoes are formed is still not fully understood and thus, difficult to predict. For a detailed presentation on the subject of tornadoes and tornado formation, we refer to \cite{MR10} and \cite{MR14}.

\subsection{Modeling Tornadoes}
Due to the complexity of tornadoes, the current knowledge about them comes mainly from laboratory experiments and numerical models of idealized supercell thunderstorms, \cite{Rot13}. In 1972, Ward \cite{Ward} conducted a pioneering laboratory experiment reproducing a tornado-like flow considering a fluid with constant density. The idea was to create a flow using a fan that passes through a hole of radius $r_0$ and is placed above a rotating plate in some distance $h$, under the assumption that the ratio of $h/r$ is small. Based on this work, several experimental and numerical simulations have taken place, referred to as Ward-type simulations, \cite{Rot13}. It was shown that the vortex form changes as the rotation increases, from single-celled (centerline updraft) to single-celled below to double-celled above, to double-celled (central downdraft surrounded by updraft) to multiple vortices, \cite{Rot13}. Also, this structural change is largely independent of the Reynolds number. 

Fiedler \cite{Fied95} proposed an idealization of a tornado-like flow that is defined on a closed domain and is for theoretical analysis. Here, the buoyancy force is taken into consideration. The behavior of such flows can be analyzed numerically using the axisymmetric, incompressible Navier-Stokes equations in cylindrical coordinates. Hence, the model takes the form
\begin{subequations}
\begin{align}
\frac{Du}{Dt} &= \frac{v^2}{r} + 2\Omega v + \frac{1}{Re} \Big[\frac{1}{r}\frac{\partial }{\partial r} \Big(r \frac{\partial u}{\partial r}\Big)+ \frac{\partial^2 u}{\partial z^2} -  \frac{u}{r^2}  \Big] - \frac{\partial p}{\partial r}, \\
\frac{Dv}{Dt}  & = - \frac{uv}{r} -  2\Omega u + \frac{1}{Re} \Big[\frac{1}{r}\frac{\partial }{\partial r} \Big(r \frac{\partial v}{\partial r}\Big)+ \frac{\partial^2 v}{\partial z^2} -  \frac{v}{r^2}  \Big], \\
\frac{Dw}{Dt}  & = b + \frac{1}{Re} \Big[\frac{1}{r}\frac{\partial }{\partial r} \Big(r \frac{\partial w}{\partial r}\Big)+ \frac{\partial^2 w}{\partial z^2} \Big] - \frac{\partial p}{\partial z}, \\
0&=\frac{1}{r} \frac{\partial }{\partial r} (ru) + \frac{\partial w}{\partial z}  ,
\end{align}
\end{subequations}
where $\frac{D}{Dt} = \frac{\partial }{\partial t} + u \frac{\partial }{\partial r}+ w\frac{\partial }{\partial z}$ stands for the material derivative, $b$ is the buoyancy and $\Omega$ is the non-dimensional swirl ratio which depends on both the angular momentum and the buoyancy. Numerical experiments of this model produce results analogous to Ward-type experiments for different values of $\Omega$ and $Re$, \cite{Rot13}.
 
In addition, various analytical models have been introduced to describe a tornado-like flow behavior. Assuming that a vortex line resembles the tornado core, we consider again the incompressible axisymmetric Euler and Navier - Stokes equations.

Let us review some widely used vortex models. A detailed presentation can be found in \cite{Gillmeier-p} and \cite{Kim17} and in references therein.

The Rankine vortex model is considered as the simplest one. Here the flow is assumed to be one-dimensional, steady, inviscid and all body forces are omitted. Hence, the model takes the form  
\begin{align*}
\frac{d p(r)}{d r}  =\rho \frac{v^2}{r},
\end{align*}
where $\rho$ is the density. Also, it is assumed that the velocity component is discontinuous and is written as
\begin{align*}
\bar{v}(\bar{r}) = 
\left\{
\begin{aligned}
\bar{r}	&\quad for \,\, \bar{r}<1,\\
\frac{1}{\bar{r}} &\quad \,\, for  \,\, \bar{r}>1.
\end{aligned}
\right.
\end{align*}
where $\bar{v} = \frac{v}{v_{max}}$ is the normalized velocity and $\bar{r}=\frac{r}{R}$ is the normalized distance for $R$ the radius of the core vortex. Sometimes, a modified version of velocity is used that is
\[\bar{v} = \frac{2\bar{r}}{(1+\bar{r}^2)}.\]
If the discontinuous velocity is considered, solving the differential equation yields the normalized pressure $\bar{p}(\bar{r}) = \frac{p(r)}{\rho v_{max}^2}$ that is
\begin{align*}
\bar{p}(\bar{r}) = 
\left\{
\begin{aligned}
\bar{p}(0) + \frac{1}{2} \bar{r}^2	&\quad for \,\, \bar{r}<1,\\
\bar{p}|_{r\to\infty} -\frac{1}{\bar{r}^2} &\quad \,\, for \,\, \bar{r}>1.
\end{aligned}
\right.
\end{align*}

Another vortex model is the Burgers-Rott, where the flow is assumed to be steady, with constant viscosity and zero body forces. Moreover, it is assumed that $u=u(r)$, $v=v(r)$, $w=w(z)$ and $p=p(r,z)$. The model then has the following form
\begin{subequations}
\label{s1}
\begin{align}
u \frac{\partial u}{\partial r} - \frac{v^2}{r} & = \mu \Big[\frac{1}{r}\frac{\partial }{\partial r} \Big(r \frac{\partial u}{\partial r}\Big) -  \frac{u}{r^2}  \Big] - \frac{1}{\rho}\frac{\partial p}{\partial r}, \\
u \frac{\partial v}{\partial r} + \frac{uv}{r} & = \mu \Big[\frac{1}{r}\frac{\partial }{\partial r} \Big(r \frac{\partial v}{\partial r}\Big) -  \frac{v}{r^2}  \Big], \\
w \frac{\partial w}{\partial z} \qquad  & = - \frac{1}{\rho}\frac{\partial p}{\partial z}, \\
\frac{1}{r} \frac{\partial }{\partial r} (ru) + \frac{\partial w}{\partial z}  & = 0,
\end{align}
\end{subequations}
where $\rho$ is the density and $\mu$ the dynamic viscosity. It is also assumed that
\begin{align*}
\bar{w}(\bar{z}) &= 2\,\alpha \,\bar{z}, \\
\bar{u}(\bar{r}) &= -\alpha \,\bar{r},
\end{align*}
where $\bar{z}=\frac{z}{R}$ is the normalized vertical height, $\bar{u}=\frac{u}{v_{max}}$ and $\bar{w}=\frac{w}{v_{max}}$ are the normalized velocities and $\alpha = \frac{2\mu}{v_{max}}$. Under these assumptions, solving the system \eqref{s1} implies the following
\begin{align*}
\bar{v}(\bar{r}) &= \frac{1}{\bar{r}} (1-exp(-\bar{r}^2)), \quad \textrm{and} \\
\bar{p}(\bar{r}, \bar{z}) &=  \bar{p}(0,0) + \int_0^{\bar{r}} \frac{\bar{v}^2(s)}{s} ds - \frac{\bar{\alpha}}{2} (\bar{r}^2 + 4 \bar{z}^2)
\end{align*}

The Sullivan vortex model has also been used widely to model tornado-like flows. As in the case of the Burgers-Rott model, we consider a flow that is stationary, with constant viscosity and zero body forces. In addition, it is considered that velocity components are given in the form $u=u(r)$, $v=v(r)$, $w=w(r,z)$, while pressure is of the form $p=p(r,z)$. One may conclude to the following
\begin{align*}
\bar{u}(\bar{r}) &= -\bar{\alpha}\bar{r} +\frac{2b\bar{v}}{\bar{r}} (1-e^{-\bar{r}^2}),\\
\bar{v}(\bar{r}) &= \frac{1}{\bar{r}} \frac{H(x)}{H(\infty)}, \\
\bar{w}(\bar{r},\bar{z}) &= 2 \bar{\alpha}\bar{z} (1-b\,e^{-\bar{r}^2}), 
\end{align*}
where $H(x) = \int_0^{x} e^{-s + 3 \int_0^s \frac{1}{\sigma} (1-e^{-\sigma^2}) d\sigma} ds$ for $x=\bar{r}^2$.
It is worth mentioning that although the Sullivan and the Burgers-Rott models have some similarities, the Sullivan model allows the generation of a double-celled vortex while Burgers-Rott model does not. 

\subsection{Mathematical Approach}
A  theoretically sound approach towards study of tornadoes was introduced by Long  \cite{Long58,Long61}. Assuming the tornado core is modeled by a semi-infinite vortex line in a fluid interacting with a plane boundary surface, he presented the reduction of incompressible Navier-Stokes equations 
to a system of differential equations motivated by boundary layer theory. 
Several subsequent studies \cite{Hall72, BF77,SH1999,Shtern12} took a similar direction and studied the formation of a boundary layer considering the near-axis boundary layer approximation to the incompressible axisymmetric Navier-Stokes equations. This is usually referred as quasi-cylindrical approximation and leads to the system
\begin{subequations}
\begin{align*}
\frac{v^2}{r} & = \frac{\partial p}{\partial r}, \\
u \frac{\partial v}{\partial r} + w \frac{\partial v}{\partial z} + \frac{uv}{r} & = \nu \Big[\frac{1}{r}\frac{\partial }{\partial r} \Big(r \frac{\partial v}{\partial r}\Big) -  \frac{v}{r^2}  \Big], \\
u \frac{\partial w}{\partial r} + w \frac{\partial w}{\partial z} \qquad  & = \nu \Big[\frac{1}{r}\frac{\partial }{\partial r} \Big(r \frac{\partial w}{\partial r}\Big) \qquad \Big] - \frac{\partial p}{\partial z}, \\
\frac{1}{r} \frac{\partial }{\partial r} (ru) + \frac{\partial w}{\partial z}  & = 0,
\end{align*}
\end{subequations}
The boundary layer is associated in the literature with the vortex breakdown, that is the change of direction of the flow near the boundary as $\nu\to0$.

Independently, Goldshtik (1960) showed that a similar reduction of incompressible axisymmetric Navier-Stokes equations to a system of differential equations leads to a 'paradoxical' exact solution that vanishes for some values of Reynolds number, \cite{Gold60}. Serrin (1972) broadened this class of solutions and described the existence of three different solution profiles depending on an arbitrary parameter and the kinematic viscosity, \cite{Serrin}. Following this work, several authors have extended the study to the generalized case of conical flows, \cite{SH1999}, \cite{FFA00}, \cite{Shtern12}. The ideas of Long and Goldshtik have been applied to investigate the formation of a boundary layer and the loss of existence of such solutions using different boundary conditions or a modified self-similar ansatz, \cite{Hall72,BF77,Gold90,GS89,GS90}. This line of research is systematized in the present manuscript by studying stationary solutions of the axisymmetric Navier-Stokes equations.

The above References concern the interaction of a vortex with a boundary when the flow is assumed stationary.
Models have been devised concerning the motion of the vortex core structure and its coupling with environmental flows,
 and the reader is referred to \cite{PMOK12,LMS16} and references therein regarding that subject.
 
 \bigskip
 \noindent
{\bf Acknowledgment.} 
Research partially supported by King Abdullah University of Science and Technology (KAUST) baseline funds.

\medskip
\noindent
{\bf Conflicts of interest.} 
The authors have no conflicts of interest to report.

%


\end{document}